%% file: Paper1.tex
\documentclass[final,leqno]{siamltex}
\usepackage{graphicx}
\usepackage[T1]{fontenc}
\usepackage{psfrag}
\usepackage{amsmath}
\usepackage{amssymb}
\usepackage{color}
\usepackage{times}
\usepackage{afterpage}
\usepackage{mathtools}
\usepackage{bbm}

\DeclareMathOperator{\sech}{sech}
\DeclareMathOperator{\cosech}{cosech}
\DeclarePairedDelimiter\floor{\lfloor}{\rfloor}

\usepackage[letterpaper=true,colorlinks=true,
linkcolor=red,filecolor=green,citecolor=blue,
pdfpagemode=None]{hyperref}

\usepackage{xcolor}
\colorlet{linkequation}{cyan}
\newcommand*{\SavedEqref}{}
\let\SavedEqref\eqref
\renewcommand*{\eqref}[1]{%
  \begingroup
    \hypersetup{
      linkcolor=linkequation,
      linkbordercolor=linkequation,
    }%
    \SavedEqref{#1}%
  \endgroup
}

\makeatletter
\newcommand{\clonelabel}[2]{\@bsphack
  \expandafter\ifx\csname r@#2\endcsname\relax
  \else\protected@write\@auxout{}{\string\newlabel{#1}%
    {\csname r@#2\endcsname}}%
  \fi
  \expandafter\ifx\csname r@#2@cref\endcsname\relax
  \else\protected@write\@auxout{}{\string\newlabel{#1@cref}%
    {\csname r@#2@cref\endcsname}}%
  \fi
  \@esphack}
\makeatother

\newcommand{\LL}{\mathcal{L}}

\begin{document}

\title{Dirichlet-Neumann and Neumann-Neumann Waveform Relaxation
  Algorithms for Time Fractional sub-diffusion and Diffusion-Wave Equations}

\author{Soura Sana\footnotemark[1]\and Bankim C. Mandal\footnotemark[2]}

\maketitle

\begin{abstract}
  In this article, we have studied the convergence behavior of the Dirichlet-Neumann and Neumann-Neumann waveform relaxation algorithms for time-fractional sub-diffusion and diffusion-wave equations in 1D \& 2D for regular domains, where the dimensionless diffusion coefficient takes different constant values in different subdomains. We first observe that different diffusion coefficients lead to different relaxation parameters for optimal convergence. Using these optimal relaxation parameters, our analysis estimates the slow superlinear convergence of the algorithms when the fractional order of the time derivative is close to zero, almost finite step convergence when the order is close to two, and in between, the superlinear convergence becomes faster as fractional order increases. So, we have successfully caught the transition of convergence rate with the change of fractional order of the time derivative in estimates and verified them with the numerical experiments.
\end{abstract}



\renewcommand{\thefootnote}{\fnsymbol{footnote}}
\footnotetext[1]{School of Basic Sciences, IIT Bhubaneswar, India ({\tt ss87@iitbbs.ac.in}).}
\footnotetext[2]{School of Basic Sciences, IIT Bhubaneswar, India ({\tt bmandal@iitbbs.ac.in}).}

\input{Introduction}

\input{Model_Problem_and_Algorithms}

\input{Lemmas}

\input{Proof_DNWR}

\input{Proof_NNWR}

\input{Proof_NNWR2D}

\input{Numerical}

\input{Conclusion}

\bibliographystyle{siam}
\bibliography{paper}
\end{document}

%% file: Introduction.tex
\section{Introduction}
Parallel algorithms for fractional partial differential equations (FPDEs) \cite{xu2015parareal,lorin2020parallel,wu2017convergence} for numerical solutions are currently active research topics due to the ever-growing application of fractional PDEs in various fields like control, robotics, bio-engineering, solid and fluid mechanics, etc. \cite{hilfer2000applications,magin2004fractional,rossikhin2010application}. The numerical complexity of FPED models is much higher than its classical counterpart due to the dense matrix structure. To overcome these difficulties, parallel algorithms are a good choice to fulfill the high demand for better simulation within a reasonable time. However, due to the memory effects of fractional derivatives, it is not easy to always find an efficient parallel method, so we are forced to use classical parallel techniques for PDEs in FPDEs models whenever possible.

Two basic ways to use parallel algorithms in PDE models: one before discretization, i.e., the continuous procedure like Classical and Optimized Schwarz, Neumann-Neumann \cite{schwarz1870ueber,lions1990schwarz,bourgat1988variational}, etc. developed by Schwarz, Lions, Bourgat, and others and another one is after discretization like Additive Schwarz, Multiplicative Schwarz, Restrictive Additive Schwarz \cite{dryja1991additive,dryja1990some,cai1999restricted}, etc. introduced by Dryja, Widlund and Cai et al. And there is an ample number of other processes at algorithmic levels. Domain Decomposition (DD) algorithms on their own can handle evolution problems in discretizing time domain and are executed on each step. However, it increases the total execution time and communication cost for the higher volume of transferring data among processors, so the better choice is to use waveform relaxation \cite{lelarasmee1982waveform} developed by Lelarsmee et al. and introduced in DD by Gander, Stuart, Giladi, Keller and others \cite{gander1998space,giladi2002space}. For further parallelization, one can use the parallel technique in time like Parareal, Paradiag, Paraexp, etc. introduced by Lions, Maday, Gander, and others see  \cite{lions2001resolution,maday2008parallelization,gander2013paraexp}. In this article, we have used continuous approaches, particularly Dirichlet-Neumann and Neumann-Neumann waveform relaxation, due to the efficiency of handling diffusion-type models relatively better than Classical Schwarz and Optimized Schwarz. For more details, refer to the articles \cite{etna_vol45_pp424-456,gander2021dirichlet,mandal2017neumann}. 

 The connection between fractional diffusion and anomalous diffusion under the framework of continuous time random walk was established by the work of Compte and Hifler \cite{hilfer1995fractional,compte1996stochastic} in 1995. Basically, random walk with the characterization of Markovian and Gaussian properties leads to the linear time dependence of the mean square displacement, i.e., $\langle x^2(t)\rangle \sim t$, which gives classical diffusion \cite{pearson1905problem,einstein1905motion,von1906kinetischen}. Moreover, when the mean square displacement is nonlinear in time, i.e., $\langle x^2(t)\rangle \sim t^{\lambda}$ (sub-diffusion: $\lambda < 1$, super-diffusion: $\lambda > 1$) then this non-Gaussian, non-Markovian, Levy Walks types model is considered as anomalous diffusion. Compte \cite{compte1996stochastic} showed that the time-fractional (Riemann-Liouville) and space-fractional (Riesz) diffusion equations are produced by the dynamical equation of all continuous time random walk with decoupled temporal and spatial memories with either temporal or spatial scale invariance in the limiting situation. In this article, we use Caputo fractional time derivative. The model we get from Compte's work using the Riemann-Liouville derivative and our model on fractional diffusion, also known as the sub-diffusion equation, using Caputo derivative are equivalent without having the forcing term for fractional time derivative order $0< \alpha < 1$, see \cite{evangelista2018fractional}. $\alpha = 1$ imply classical diffusion corresponding to $\lambda = 1$. The model with the time-fractional derivative $\alpha \in (1,2)$ is known as diffusion-wave, enhanced diffusion, or super-diffusion. Except $1D$, the solutions of this model in $2D$ and $3D$  do not follow the probability distribution rule, i.e., the solution may change sign, see \cite{hanygad2002multidimensional}, so in that sense $\alpha >1$ does not follow the super-diffusive random walk model corresponding to $\lambda>1$ in higher dimension. Nonetheless, this model has physical applications in viscoelasticity and constant Q seismic-wave propagation \cite{pipkin2012lectures,caputo1967linear}.

The analytic solution of the fractional diffusion wave (FDW) equation in terms of Fox's H function was first investigated by Wess and Schneider \cite{schneider1989fractional}. Mainardi \cite{mainardi1996fundamental} used Laplace transformation to obtain the fundamental solution of FDW equations in terms of the M-Wright function. Agrawal \cite{agrawal2003response} used the separation of variables technique to reduce the FDW equation in terms of the set of infinite Fractional Differential equations and then identify the eigenfunctions that produce the solutions in the form of Duhamel's integral.

From the numerical standpoint, the solution of fractional diffusion (FD) equations  Yuste and Acedo \cite{yuste2005explicit} used forward time central space with the combination of  Grunwald–Letnikov discretization of the Riemann–Liouville derivative to obtain an explicit scheme with the first order accuracy in time. Laglands and Henry \cite{langlands2005accuracy} used backward time central space with the combination of the L1 scheme developed by Oldham and Spanier for fractional time derivative to obtain a fully implicit finite difference scheme, and its unconditional stability was proved by Chen et al. \cite{chen2007fourier}. Later Stynes et al. \cite{stynes2017error}  introduced the idea of applying the graded mesh to capture the weak singularity effect in computation and showed how the order of convergence of the scheme related to the mesh grading, and they found and proved optimal mesh grading technique. Sanz-Serna \cite{sanz1988numerical} investigated the numerical solution of a partial integrodifferential equation which may be considered a $3/2$ order time fractional wave equation. Sun and Wu \cite{sun2006fully} gave a fully discrete, unconditionally stable difference scheme of order $O(t^{3-\alpha}+h^2)$ for the fractional wave equation. 

We have organized this article by introducing the model problem in Section 2. In Sections 3 \& 4, we have briefly described the DNWR and NNWR algorithms. In Section 5, we present the necessary auxiliary results for proving the main convergence theorem. In Section 6, we have considered the DNWR algorithm for two subdomain problems and prove its convergence by estimating the error. In Sections 7 \& 8, we have done the same but for the NNWR algorithm in multiple subdomains for 1D and two subdomains for 2D. In Section 9, we have chosen a particular model and verified all theoretical results with the numerical ones.

%% file: Model_Problem_and_Algorithms.tex
\section{Model problem}
To extend DNWR and NNWR algorithm for fractional PDEs, we have taken the linear time-fractional diffusion wave model\cite{mainardi2003wright,evangelista2018fractional}, where the fractional order $\nu$ takes any constant value from zero to one, i.e. $2\nu \in(0,2)$. When $2\nu = 1$ the anomalous model converts to normal diffusion. On the bounded domain $\Omega \subset \mathbb{R}^d$, $0<t<T$, our model problem reads as follows:
\begin{equation}\label{model_problem}
\begin{cases}
D^{2\nu}_{t}u = \nabla\cdot\left(\kappa(\boldsymbol{x},t)\nabla u\right)+f(\boldsymbol{x},t), & \textrm{in}\; \Omega\times(0,T),\\
u(\boldsymbol{x},t) = g(\boldsymbol{x},t), & \textrm{on}\; \partial\Omega\times(0,T),\\
u(\boldsymbol{x},0) = u_{0}(\boldsymbol{x}), & \textrm{in}\; \Omega,
\end{cases}\,.
\end{equation}
Here, $\kappa(\boldsymbol{x},t)>0$ is the dimensionless diffusion coefficient and $D^{\alpha}_{t}$ is the Caputo fractional derivative \cite{caputo1967linear} defined for order $\alpha$, $n-1<\alpha<n$ and $n \in \mathbb{N}$  as follows:
$$D^{\alpha}_{t}x(t) := \frac{1}{\Gamma(n-\alpha)}\int_0^t(t-\tau)^{n-\alpha-1}x^{(n)}(\tau)d\tau. $$
As the diffusion coefficient takes different constant value on different subdomain and the source term is sufficiently smooth, so we introduced continuity of $u$ and the flux on the interface. For existence and uniqueness of the weak solution of \eqref{model_problem} see \cite{van2021existence}.
We will introduce DNWR and NNWR algorithms for \eqref{model_problem}:

\section{The Dirichlet-Neumann Waveform Relaxation algorithm}
The DNWR method is a semi-parallel (except two sub-domain case) type iterative algorithm, combining the substructuring DD method for space and waveform relaxation in time. To define the DNWR algorithm for the model problem
\eqref{model_problem}, the spatial domain $\Omega$ is partitioned into two non-overlapping subdomains $\Omega_{1}$ and $\Omega_{2}$ with the interface $\Gamma:=\partial\Omega_{1}\cap\partial\Omega_{2}$. $u_{i},\boldsymbol{n}_{i}$ are respectively the restriction of solution $u$ and the unit outward normal on $\Gamma$ for $\Omega_{i}$, $i=1,2$.
The DNWR algorithm starts with an initial guess $h^{(0)}(\boldsymbol{x},t)$ along the interface $\Gamma\times(0,T)$, and compute the  $u_i^{(k)},i = 1,2, \text{ for } k=1,2,\ldots$ by the followings Dirichlet-Neumann steps:
\begin{equation}\label{DNWR}
\arraycolsep0.008em
\begin{array}{rcll}
\begin{cases}
D^{2\nu}_{t}u_{1}^{(k)} = \nabla\cdot\left(\kappa_1\nabla u_{1}^{(k)}\right)+f, & \textrm{in}\; \Omega_{1},\\
u_{1}^{(k)}(\boldsymbol{x},0) = u_{0}(\boldsymbol{x}), & \textrm{in}\; \Omega_{1},\\
u_{1}^{(k)} = h^{(k-1)}, & \textrm{on}\; \Gamma,\\
u_{1}^{(k)}=g, & \textrm{on}\; \partial\Omega_{1}\setminus\Gamma,
\end{cases}
\end{array}\
\begin{array}{rcll}
\begin{cases}
D^{2\nu}_{t}u_{2}^{(k)} = \nabla\cdot\left(\kappa_2\nabla u_{2}^{(k)}\right)+f, & \textrm{in}\; \Omega_{2},\\
u_{2}^{(k)}(\boldsymbol{x},0) = u_{0}(\boldsymbol{x}), & \textrm{in}\; \Omega_{2},\\
\kappa_2\partial_{\boldsymbol{n}_{2}} u_{2}^{(k)}  =  -\kappa_1\partial_{\boldsymbol{n}_{1}} u_{1}^{(k)}, & \textrm{on}\; \Gamma,\\
u_{2}^{(k)}=g, & \textrm{on}\;\partial\Omega_{2}\setminus\Gamma,
\end{cases}
\end{array}
\end{equation}
and then with the relaxation parameter $\theta\in(0,1]$ update the interface data using
\begin{equation}\label{DNWR2}
  h^{(k)}(\boldsymbol{x},t)=\theta u_{2}^{(k)}\left|_{\Gamma\times(0,T)}\right.+(1-\theta)h^{(k-1)}(\boldsymbol{x},t).
\end{equation}
The main goal of our analysis is to study how the update part error $w^{(k-1)}(\boldsymbol{x},t) := u|_{\Gamma\times(0,T)} - h^{(k-1)}(\boldsymbol{x},t)$, where $u|_{\Gamma\times(0,T)}$ is the exact solution on the interface, converges to zero, and by linearity it suffices to consider the convergence of the so called error equations, with $f(\boldsymbol{x},t)=0$, $g(\boldsymbol{x},t)=0$,
$u_{0}(\boldsymbol{x})=0$ in \eqref{DNWR}. This will be studied in section \ref{sectionDNWR}.

\section{The Neumann-Neumann Waveform Relaxation algorithm}\label{Section3}
To introduce the fully parallel NNWR algorithm for the model problem \eqref{model_problem} on multiple subdomains, the domain $\Omega$ is partitioned into non-overlapping subdomains $\Omega_{i}$, $1\leq i\leq N$.
Set $\Gamma_i:=\partial\Omega_{i}\setminus\partial\Omega$,
$\Lambda_i:=\{j\in\{1,\ldots,N\}: \Gamma_{i}\cap\Gamma_{j} \, \mbox{has nonzero measure}\}$ and
$\Gamma_{ij}:=\partial\Omega_{i}\cap\partial\Omega_{j}$, so that the
interface of $\Omega_{i}$ can be rewritten as
$\Gamma_i=\bigcup_{j\in\Lambda_i}\Gamma_{ij}$. We denote by $\boldsymbol{n}_{ij}$
the unit outward normal for $\Omega_{i}$ on the interface
$\Gamma_{ij}$.
The NNWR algorithm starts with an initial guess
$h_{ij}^{(0)}(\boldsymbol{x},t)$ along the interfaces
$\Gamma_{ij}\times (0,T)$, $j\in\Lambda_i$, $i=1,\ldots,N$, and then
performs the following two-step iteration: at each iteration $k$, one
first solves Dirichlet problem on each $\Omega_{i}$ in parallel,
\begin{equation}\label{NNWRD}
  \begin{array}{rcll}
  \begin{cases}
   D^{2\nu}_{t}u_{i}^{(k)} = \nabla\cdot\left(\kappa(\boldsymbol{x},t)\nabla u_{i}^{(k)}\right) + f, & \mbox{in $\Omega_{i}$},\\
    u_{i}^{(k)}(\boldsymbol{x},0) = u_{0}(\boldsymbol{x}), & \mbox{in $\Omega_{i}$},\\
    u_{i}^{(k)} = g, & \mbox{on $\partial\Omega_{i}\setminus\Gamma_i$},\\
   u_{i}^{(k)} = h_{ij}^{(k-1)}, & \mbox{on $\Gamma_{ij}, j\in\Lambda_i$},
   \end{cases}
  \end{array}
\end{equation}
and then solves the Neumann problems on all subdomains in parallel,
\begin{equation}\label{NNWRN}
\begin{array}{rcll}
\begin{cases}
D^{2\nu}_{t}\psi_{i}^{(k)} = \nabla\cdot\left(\kappa(\boldsymbol{x},t)\nabla \psi_{i}^{(k)}\right), & \mbox{in $\Omega_{i}$},\\
   \psi_{i}^{(k)}(\boldsymbol{x},0) = 0, & \mbox{in $\Omega_{i}$},\\
  \psi_{i}^{(k)} = 0, & \mbox{on $\partial\Omega_{i}\setminus\Gamma_i$},\\
   \partial_{\boldsymbol{n}_{ij}}\psi_{i}^{(k)} = \partial_{\boldsymbol{n}_{ij}}u_{i}^{(k)}+\partial_{\boldsymbol{n}_{ji}}u_{j}^{(k)}, & \mbox{on $\Gamma_{ij}, j\in\Lambda_i$}.
\end{cases}
\end{array}
\end{equation}
The interface values are then updated with the formula
\begin{equation}\label{NNWR2}
  w_{ij}^{(k)}(\boldsymbol{x},t)=w_{ij}^{(k-1)}(\boldsymbol{x},t)-\theta_{ij}
  \left( \psi_{i}^{(k)}\left|_{\Gamma_{ij}\times(0,T)}\right.+\psi_{j}^{(k)}\left|_{\Gamma_{ij}\times(0,T)}\right.\right),
\end{equation}
where $\theta_{ij}\in(0,1]$ is a relaxation parameter. Our goal is to study the convergence of  $w_{ij}^{(k-1)}(\boldsymbol{x},t) := u_i|_{\Gamma\times(0,T)} - h_{ij}^{(k-1)}(\boldsymbol{x},t)$ as $k \to \infty$. 

Before presenting the main convergence result, we discuss a few necessary lemmas in the next section, which will be helpful for sections 6, 7 \& 8.

%% file: Lemmas.tex
\section{Auxiliary Results}\label{Section4}
The convergence results of DNWR and NNWR are based on the kernel estimates arising in the Laplace transform of Caputo derivatives from the error equations. In this section, we will prove necessary lemmas relevant to obtain the estimates. 

\begin{lemma}
\label{SimpleLaplaceLemma}
Let $g$ and $w$ be two real-valued functions in $(0,\infty)$ with
$\hat{w}(s)=\mathcal{L}\left\{ w(t)\right\}$ the Laplace transform of
$w$. Then for $t\in(0,T)$, we have the following properties:
\begin{enumerate}
  \item[(i)] \label{L1}If $g(t)\geq0$ and $w(t)\geq0$, then $(g*w)(t)\geq0$.

  \item[(ii)] \label{L2} $\| g*w\|_{L^{1}(0,T)}\leq\| g\|_{L^{1}(0,T)}\| w\|_{L^{1}(0,T)}.$

  \item[(iii)] \label{L3} $\| g*w\|_{L^{\infty}(0,T)}\leq\| g\|_{L^{\infty}(0,T)}\|w\|_{L^{1}(0,T)}.$

 \item[(iv)] \label{L5}If $w(t)\geq 0$ be $L^1$-integrable function on $(0,T)$, then $\int_{0}^tw(\tau)d\tau\leq \lim_{s\rightarrow0+}\hat{w}(Re(s))$.
\end{enumerate}
\end{lemma}
\begin{proof}
		  The first four proofs $(i)-(iv)$ follow directly from the definitions.
		  \begin{enumerate}
		  	\item [(v)] For $w(t) \geq 0$ on $t \in (0,T)$, we have 
		  	\begin{align*}
		  		\int_{0}^t|w(\tau)|d\tau &= \int_{0}^{t} \lim_{s \to 0+} |\exp(-s\tau) w(\tau)| d\tau,
		  	\end{align*}
	  		swapping the order of limit and integration using the dominated convergence theorem, which is possible as $w \in L^1(0,T)$, gives
	  		\begin{align*}
	  		\int_{0}^{t} \lim_{s \to 0+} |\exp(-s\tau) w(\tau)| d\tau
		  		&= \lim_{s \to 0+} \int_{0}^{t}  |\exp(-s\tau) w(\tau)| d\tau \\
		  		& \leq \lim_{s \to 0+} \int_{0}^{\infty}  |\exp(-s\tau) w(\tau)| d\tau \\
		  		&= \lim_{s \to 0+} \hat{w}(Re(s)).
		  	\end{align*}
		  \end{enumerate}
\hfill \end{proof}

\begin{lemma}\label{PositivityLemma}
  Let $0 \leq l_1 <l_2$ and $s$ be a complex variable. Then, for $t\in(0,\infty)$ and
  \begin{enumerate}
  	\item [(i)] 
  	for $0<\alpha< 1$
  	\begin{equation*}
  		\Phi(t):=\mathcal{L}^{-1}\left\{ \frac{\sinh(l_1s^{\alpha})}
  		{\sinh(l_2s^{\alpha})}\right\} \quad\mbox{and}\quad
  		\Psi(t):=\mathcal{L}^{-1}\left\{\frac{\cosh(l_1s^{\alpha})}{\cosh(l_2s^{\alpha})}\right\}
  	\end{equation*}
  exist.
  \item [(ii)]
  for $0<\alpha \leq 1/2$
  \begin{equation*}
  	\Phi(t) \geq 0 \textit{ and } \Psi(t) \geq 0.
  \end{equation*}
  \end{enumerate}
\end{lemma}
\begin{proof}
	\begin{enumerate}
		\item [(i)]
	     Setting $s = re^{i\theta}, \theta \in (-\pi/2,\pi/2)$, a short calculation shows that for $0 \leq l_1 <l_2$ and for every positive
		$p$
		\begin{align*}
			\left|\frac{s^p \sinh(l_1s^{\alpha})}{\sinh(l_2s^{\alpha})}\right|
			&\leq r^p \left|\frac{\exp(l_1r^{\alpha}\cos(\alpha\theta)) + \exp(-l_1r^{\alpha}\cos(\alpha\theta))}{\exp(l_2r^{\alpha}\cos(\alpha\theta)) - \exp(-l_2r^{\alpha}\cos(\alpha\theta))}\right| \\
			&= \frac{r^p}{\exp((l_2-l_1)r^{\alpha}\cos(\alpha\theta))}\frac{1+\exp(-2l_1r^{\alpha}\cos(\alpha\theta))}{1-\exp(-2l_2r^{\alpha}\cos(\alpha\theta))} \\
			&\to 0 \quad \text{as} \ r \to \infty \textrm{ for fixed } \alpha \in (0,1) \textrm{ and } \theta \in (-\pi/2, \pi/2).
		\end{align*}
		So by \cite[p.~178]{churchill1971operational}, its inverse Laplace transform exists and is continuous (in fact, infinitely differentiable). Thus, $\Phi(t)$ is a continuous function. A similar argument holds for $\Psi(t)$.
		\item [(ii)]
		To prove the non-negativity of $\Phi(t)$ and $\Psi(t)$ for $0<\alpha \leq 1/2$, first consider the following sub-diffusion equation:
		$D^{2\alpha}_{t}u- \Delta u=0$ on $(0,l_2)$ with initial condition $u(x,0)=0$ and boundary conditions
		$u(0,t)=0$, $u(l_2,t)=g(t)$.  Now performing a Laplace transform on the Caputo derivative, we have the solution of sub-diffusion equation
		\[
		\hat{u}(x,s)
		=\hat{g}(s)\frac{\sinh(xs^{\alpha})}{\sinh(l_2s^{\alpha})}
		\quad\Longrightarrow\quad
		u(x,t)=\int_{0}^{t}g(t-\tau)\Phi(\tau)d\tau.
		\]
		If $g$ is non-negative, then by the maximum principle \cite{luchko2009maximum}, this IBVP has a non-negative solution $u(x,t)$ for all $x\in[0,l_2]$, $t>0$. It is then straightforward to show that the kernel $\Phi(t) \geq 0$.
		
		To prove $\Psi(t) \geq 0$, we consider the same sub-diffusion equation $D^{2\alpha}_{t}u- \Delta u=0$, $u(x,0)=0$, on the domain $(-l_2,l_2)$ and with boundary conditions $u(-l_2,t)=u(l_2,t)=g(t)$. Using Laplace transform in time we get the solution at $x=l_1$ as:
		\begin{equation*}
			\hat{u}(l_1,s)=
			\hat{g}(s)\frac{\cosh(l_1s^{\alpha})}{\cosh(l_2s^{\alpha})},
		\end{equation*}
		and hence a similar argument as in the first case proves that $\Psi(t)$ is also non-negative.
	\end{enumerate}
\hfill \end{proof}

\begin{lemma} \label{invL_exp}
	For $\alpha \in (0,1)$ and $l,t>0$, the following results hold:
\begin{enumerate}
\item[(i)]  The inverse Laplace transform of $  e^{-l s^{\alpha}}$ is:
\begin{equation*} 
	\mathcal{L}^{-1}\left\{e^{-l s^{\alpha}}\right\} 
    =l\alpha t^{-(\alpha+1)}M_{\alpha}\left(lt^{-\alpha}\right),
\end{equation*} 
where $M_{\alpha}(x), x\in(0,\infty)$ be the M-Wright function.
\item[(ii)] For $\phi \in (0,\pi)$,
\begin{equation*} 
	 \left(\frac{\sin \alpha\phi}{ \sin \phi}\right)^{\alpha/(1-\alpha)} \frac{\sin (1-\alpha)\phi}{\sin \phi} \geq (1-\alpha)\alpha^{\alpha/(1-\alpha)}.
 \end{equation*}
\item[(iii)]
\begin{equation*} 
	\left\|\mathcal{L}^{-1}\left\{e^{-l s^{\alpha}}\right\}\right\|_{L^1(0,t)} \leq \exp\left(-(1-\alpha)\left(\frac{\alpha}{t}\right)^{\alpha/(1-\alpha)}l^{1/(1-\alpha)}\right).
 \end{equation*}  
\end{enumerate}
\end{lemma}

\begin{proof}
\begin{enumerate}
\item[(i)]
 We know
\[
\frac{d}{ds}e^{-l s^{\alpha}} = -l \alpha s^{\alpha-1}e^{-l s^{\alpha}}
\]
now taking inverse Laplace transform on both side and then by \cite[p.~11]{mainardi2020wright}, we have
\begin{align*}
	\mathcal{L}^{-1}\left\{\frac{d}{ds}e^{-l s^{\alpha}}\right\} &= -l \alpha \mathcal{L}^{-1}\left(s^{\alpha-1}e^{-l s^{\alpha}}\right) \\
	-t \mathcal{L}^{-1}\left\{e^{-l s^{\alpha}}\right\} &= -l \alpha t^{-\alpha} M_{\alpha}(l t^{-\alpha}).
\end{align*}
 Hence, \[
\mathcal{L}^{-1}\left\{e^{-l s^{\alpha}}\right\}
    =l\alpha t^{-(\alpha+1)} M_{\alpha}(l t^{-\alpha}).
\]
\item[(ii)]
 $\sin(x)$ is a concave function on $[0,\pi]$, so, $-\sin(x)$ is convex function. Using Jensen inequality, we have $-\alpha \sin \phi \geq -\sin \alpha\phi$. Hence, $\left(\frac{\sin \alpha\phi}{\alpha \sin \phi}\right)^{\alpha/(1-\alpha)} \geq 1$. Similarly,$\frac{\sin (1-\alpha)\phi}{(1-\alpha)\sin \phi} \geq 1$. Multiplying both we have our proof.   

\item[(iii)] 
 To prove this lemma, we use the expression of the M-Wright function, first introduced in \cite{Mikusiski1959OnTF} as:
\begin{equation}\label{lam_3_3}
	M_{\alpha}(x) = \frac{x^{\alpha/(1-\alpha)}}{\pi(1-\alpha)}\int_0^{\pi}u(\phi)\exp(-u(\phi)x^{\frac{1}{1-\alpha}}) d\phi,
\end{equation}
where $u(\phi) = \left(\frac{\sin \alpha\phi}{\sin \phi}\right)^{\alpha/(1-\alpha)} \frac{\sin (1-\alpha)\phi}{\sin \phi}$, and  $\phi \in(0,\pi)$.
Using part (ii), we have $M_{\alpha}(x) >0 \,\forall x \in (0,\infty)$; therefore,
\begin{align*}
	\left\|\mathcal{L}^{-1}\left\{e^{-l s^{\alpha}}\right\}\right\|_{L^1(0,t)}
	&=\int_0^t l\alpha \tau^{-(\alpha+1)}M_{\alpha}(l\tau^{-\alpha}) d\tau \\
	&= \int_{lt^{-\alpha}}^{\infty}M_{\alpha}(x)dx \\
	&= \frac{1}{\pi(1-\alpha)}\int_{lt^{-\alpha}}^{\infty}dx \,x^{\frac{\alpha}{1-\alpha}}\int_0^{\pi}u(\phi)\exp(-u(\phi)x^{\frac{1}{1-\alpha}}) \, d\phi \\
	&= \frac{1}{\pi}\int_0^{\pi}d\phi \, \int_{lt^{-\alpha}}^{\infty} \frac{u(\phi)}{1-\alpha}x^{\frac{\alpha}{1-\alpha}}\exp(-u(\phi)x^{\frac{1}{1-\alpha}}) \, dx \\
	& = \frac{1}{\pi}\int_0^{\pi}d\phi \, \exp(-u(\phi)(lt^{-\alpha})^{\frac{1}{1-\alpha}}) \\
	& \leq \frac{1}{\pi}\int_0^{\pi}d\phi \exp\left(-(1-\alpha)\left(\frac{\alpha}{t}\right)^{\alpha/(1-\alpha)}l^{1/(1-\alpha)}\right) \\
	&= \exp\left(-(1-\alpha)\left(\frac{\alpha}{t}\right)^{\alpha/(1-\alpha)}l^{1/(1-\alpha)}\right).
\end{align*}
\end{enumerate}
\hfill\end{proof}

%% file: Proof_DNWR.tex
\section{Convergence of DNWR Algorithm}\label{sectionDNWR}
We are now in a position to present the main convergence result for the DNWR algorithm \eqref{DNWR}-\eqref{DNWR2}. For theoretical clarity and algebraic simplicity, we choose the 1D model of the sub-diffusion and diffusion-wave equation. We consider the heterogeneous problem with $\kappa(\boldsymbol{x},t)=\kappa_1$, on $\Omega_{1}=(-a,0)$ and $\kappa(\boldsymbol{x},t)=\kappa_2$, on $\Omega_{2}=(0,b)$. Define $w^{(k)}(t)$ be the error along the interface at $x=0$. The Laplace transform in time converts the error subproblems into the ODEs:
\begin{equation}\label{DNWRL}
	\begin{array}{rcll}
	\begin{cases}
	  (s^{2\nu}-\kappa_1\partial_{xx})\hat{u}_{1}^{(k)} = 0 & \textrm{on $(-a,0)$},\\
	  \hat{u}_{1}^{(k)}(-a,s) = 0,\\
	  \hat{u}_{1}^{(k)}(0,s) = \hat{w}^{(k-1)}(s),
	  \end{cases}
	\end{array}\quad
	\begin{array}{rcll}
	\begin{cases}
	  (s^{2\nu}-\kappa_2\partial_{xx})\hat{u}_{2}^{(k)} = 0 & \textrm{on $(0,b)$},\\
	  \kappa_2\partial_{x}\hat{u}_{2}^{(k)}(0,s) = \kappa_1\partial_{x}\hat{u}_{1}^{(k)}(0,s),\\
	  \hat{u}_{2}^{(k)}(b,s) = 0,
	  \end{cases}
	\end{array}
\end{equation}
followed by the updating step
\begin{equation}\label{DNWRL2}
	\hat{w}^{(k)}(s)=\theta\hat{u}_{2}^{(k)}(0,s)+(1-\theta)\hat{w}^{(k-1)}(s).
\end{equation}
Solving the BVP in Dirichlet and Neumann step in \eqref{DNWRL}, we get
\begin{align}
	\hat{u}_{1}^{(k)}(x,s)&= \frac{\hat{w}^{(k-1)}(s)}{\sinh(a\sqrt{s^{2\nu}/\kappa_1})}\sinh\left(
	  (x+a)\sqrt{s^{2\nu}/\kappa_1}\right), \\
	\hat{u}_{2}^{(k)}(x,s)&= \sqrt{\kappa_1/\kappa_2}\hat{w}^{(k-1)}(s)\frac{\coth(a\sqrt{s^{2\nu}/\kappa_1})}
	   {\cosh(b\sqrt{s^{2\nu}/\kappa_2})}\sinh((x-b)\sqrt{s^{2\nu}/\kappa_2}). \label{DNWR sol1}
\end{align}
Substituting \eqref{DNWR sol1} into \eqref{DNWRL2}, the recurrence relation for $k\in \mathbb{N}$ becomes: 
\begin{align}
	\hat{w}^{(k)}(s) &=\left(
	  1-\theta-\theta\sqrt{\kappa_1/\kappa_2}\tanh(b\sqrt{s^{2\nu}/\kappa_2})\coth(a\sqrt{s^{2\nu}/\kappa_1})\right)^{k}
	\hat{w}^{(0)}(s). \label{DNWR sol2}
\end{align}
Defining $A := a/\sqrt{\kappa_1} \text{ and } B := b/\sqrt{\kappa_2}$, reduce \eqref{DNWR sol2} to
\begin{equation}\label{DNWRLup}
   \hat{w}^{(k)}(s)=\left(1-\theta-\theta\sqrt{\kappa_1/\kappa_2}\tanh(Bs^{\nu})\coth(As^{\nu})\right)^{k}
   \hat{w}^{(0)}(s),\quad k=1,2,3,\ldots
\end{equation}

\begin{theorem}[Convergence of DNWR for $A=B$]
	 When $A = B$ in \eqref{DNWRL}-\eqref{DNWRL2}, the DNWR algorithm for both sub-diffusion and diffusion-wave cases converges linearly for
     $0<\theta<1$, and $\theta\neq 1/(1+\sqrt{\kappa_1/\kappa_2})$. For $\theta=1/(1+\sqrt{\kappa_1/\kappa_2})$, it converges in two iterations; moreover, convergence is independent of the time window size $T$.
\end{theorem}

\begin{proof}
	For $A=B$, the equation \eqref{DNWRLup} reduces to
	$\hat{w}^{(k)}(s)=(1-(1+\sqrt{\kappa_1/\kappa_2})\theta)^{k}\hat{w}^{(0)}(s),$ which has the
	back transform $w^{(k)}(t)=(1-(1+\sqrt{\kappa_1/\kappa_2})\theta)^{k}w^{(0)}(t)$.  Thus the
	convergence is linear for $0<\theta<1$, $\theta\neq 1/(1+\sqrt{\kappa_1/\kappa_2})$. If $\theta=1/(1+\sqrt{\kappa_1/\kappa_2})$,
	we have $w^{(1)}(t)=0$, and hence two step convergence.
\hfill\end{proof}

For $A \neq B$, define
\begin{equation}\label{Gdef}
   \hat{f_1}(s):=\tanh(Bs^{\nu})\coth(As^{\nu})-1
    =\frac{\sinh((B-A)s^{\nu})}{\sinh(As^{\nu})\cosh(Bs^{\nu})},
\end{equation}
 and the recurrence relation \eqref{DNWRLup} can be rewritten as
\begin{equation}\label{hrecurrenceLaplace}
	\hat{w}^{(k)}(s)=\left\{\begin{array}{ll}
	\left( (1-(1+\sqrt{\kappa_1/\kappa_2})\theta)-\sqrt{\kappa_1/\kappa_2}\theta \hat{f_1}(s)\right)^{k}\hat{w}^{(0)}(s), & \theta\neq1/(1+\sqrt{\kappa_1/\kappa_2}),\\
	\left(-1\right)^{k}(\sqrt{\kappa_1/\kappa_2}/(1+\sqrt{\kappa_1/\kappa_2}))^k \hat{f_1}^{k}(s)\hat{w}^{(0)}(s), & \theta=1/(1+\sqrt{\kappa_1/\kappa_2}).
	\end{array}\right.
\end{equation}
Note that for $\textrm{Re}(s)>0,$ $\hat{f_1}(s)$
is $\mathcal{O}(s^{-p})$ for every positive $p$, which can be seen as
follows: setting $s=re^{i\theta}$, we obtain for $A \geq B$
the bound
\[
  \left|s^{p}\hat{f_1}(s)\right|\leq\left|\frac{s^{p}}{\cosh(Bs^{\nu})}\right|
  \leq\frac{2r^{p}}{\left|e^{Br^{\nu}cos(\nu\theta)}-e^{-Br^{\nu}cos(\nu\theta)}\right|}\rightarrow0
  \quad\mbox{as $r\rightarrow\infty$},
\]
and for $A < B$, we get the bound
\[
  \left|s^{p}\hat{f_1}(s)\right|\leq\left|\frac{s^{p}}{\sinh(As^{\nu})}
  \right|\leq\frac{2r^{p}}{\left|e^{Ar^{\nu}cos(\nu\theta)}-e^{-Ar^{\nu}cos(\nu\theta)}\right|}\rightarrow0
  \quad\mbox{as $r\rightarrow\infty$}.
\]
Therefore, by \cite[p.~178]{churchill1971operational}, $\hat{f_1}(s)$ is the Laplace transform of an
infinitely differentiable function $f_{1}(t)$. We now define
\begin{equation} \label{Fkdef}
	f_{k}(t):=\mathcal{L}^{-1}\left\{
	\hat{f_1}^{k}(s)\right\} \quad{\rm for}\; k=1,2,\ldots.
\end{equation}
Now we will study the convergence result for the special case $\theta =1/(1+\sqrt{\kappa_1/\kappa_2})$, as it leads to the super-linear convergence, in comparison to other values of $\theta$ which gives linear convergence \cite{mandal2021substructuring}. Therefore, we define $\theta^* :=\sqrt{\kappa_1}/(\sqrt{\kappa_1}+\sqrt{\kappa_2})$. 

Before going to the main theorems for $A \neq B$, we first prove a few lemmas related to the estimates of inverse Laplace transform.

\begin{lemma} \label{invL_cosech}
	For $\alpha \in(0,1), l>0 \textrm{ and }  k \in \mathbb{N}$, the following results hold:
	\begin{enumerate}
		\item[(i)] 
		\begin{align} \label{lam_4_1}
			\mathcal{L}^{-1}\left\{\cosech^{k}(ls^{\alpha})\right\}
			& =  2^{k}{\displaystyle \sum_{m=0}^{\infty}}\binom{m+k-1}{m}\frac{(2m+k)l\alpha}{t^{\alpha+1}}M_{\alpha}\left(\frac{(2m+k)l}{t^{\alpha}}\right)
		\end{align}
		\item[(ii)] 
		\begin{equation*} 
			\| \LL^{-1}\left\{\cosech^{k}(l s^{\alpha})\right\}\|_{L^1(0,t)}
			\leq \left(\frac{2}{1-e^{-A_1B_1}}\right)^k e^{-A_1k^{1/(1-\alpha)}},
		\end{equation*}
		where $A_1 = (1-\alpha)\left(\frac{\alpha}{t}\right)^{\alpha/1-\alpha}l^{1/1-\alpha}$, $c = \left\lfloor\frac{1}{1-\alpha}\right\rfloor$ and $B_1 = [(2+k)^c - k^c]^{1/c(1-\alpha)}$.   
	\end{enumerate}
\end{lemma}

\begin{proof}
	\begin{enumerate}
		\item[(i)]
		Taking the Laplace transform on both sides of \eqref{lam_4_1} and using the definition, we have
		\begin{equation*}
			\cosech^k(ls^{\alpha}) = 2^k\int_0^\infty e^{-st}\sum_{m=0}^\infty {m+k-1 \choose m}
			\frac{(2m+k)l\alpha}{t^{\alpha+1}}M_{\alpha}\left(\frac{(2m+k)l}{t^{\alpha}}\right)\,dt.
		\end{equation*}
		We can interchange the sum and the integral provided the conditions of Fubini's theorem holds i.e., if
		\[
		\sum_{m=0}^\infty {m+k-1 \choose m}\int_0^\infty \left|e^{-st}\frac{(2m+k)l\alpha}{t^{\alpha+1}}M_{\alpha}\left(\frac{(2m+k)l}{t^{\alpha}}\right)\right|\,dt < \infty.
		\]
		From part (i) of Lemma \ref{invL_exp}, we know that
		\begin{equation*}
			\mathcal{L}^{-1}\left\{e^{-\lambda s^{\alpha}}\right\}
			=\frac{\lambda\alpha}{t^{\alpha+1}}M_{\alpha}\left(\frac{\lambda}{t^{\alpha}}\right),
			\quad\lambda > 0.
		\end{equation*}
		So for $\mbox{Re}(s) \geq s_0 > 0$, we have
		\begin{align*}
			\int_0^\infty \left|e^{-st}\frac{(2m+k)l\alpha}{t^{\alpha+1}}M_{\alpha}\left(\frac{(2m+k)l}{t^{\alpha}}\right)\right|\,dt
			&= \exp(-(2m+k)l\sqrt{\mathrm{Re}(s^{2\alpha})}) \\ 
			&\leq \exp(-(2m+k)l s_0^{\alpha}).
		\end{align*}
		Thus, using the binomial series
		\begin{equation*}\label{binomial}
			\frac{1}{(1-z)^{k}}=\sum_{m\geq0}\binom{m+k-1}{m}z^{m}
			\qquad\mbox{for $|z|<1$}
		\end{equation*}
		with $z = e^{-2ls_0^{\alpha}} < 1$, we get
		\begin{align*}
			&\sum_{m=0}^\infty {m+k-1 \choose m}\int_0^\infty \left|e^{-st}\frac{(2m+k)l\alpha}{t^{\alpha+1}}M_{\alpha}\left(\frac{(2m+k)l}{t^{\alpha}}\right)\right|\,dt\\
			&\qquad\qquad\leq \frac{e^{-kls_0^{\alpha}}}{(1-e^{-2l s_0^{\alpha}})^k}=\frac{\cosech^k(l s_0^{\alpha})}{2^k} <\infty.
		\end{align*}
		Therefore, term-by-term integration is possible, and hence the result.
		
		\item[(ii)]
		Let, $c = \left\lfloor\frac{1}{1-\alpha}\right\rfloor$, i.e. $1\leq \frac{1}{c(1-\alpha)}$,
		and
		\begin{align}
			(2m+k)^{1/1-\alpha} &= \left[(2m+k)^c\right]^{1/c(1-\alpha)} \nonumber\\
			&= \left[k^c + \sum_{p=0}^{c-1}\binom{c}{p} (2m)^{c-p} k^p\right]^{1/c(1-\alpha)} \nonumber\\
			&\geq k^{1/1-\alpha} + \left[\sum_{p=0}^{c-1}\binom{c}{p} (2m)^{c-p} k^p\right]^{1/c(1-\alpha)} \nonumber\\
			&\geq k^{1/1-\alpha} + mB_1. \label{lam_4_2}
		\end{align}
		Using part (iii) of Lemma \ref{invL_exp} in \eqref{lam_4_1} leads to
		\begin{align}
			&\| \LL^{-1}\left\{\cosech^{k}(l s^{\alpha})\right\}\|_{L^1(0,t)} \nonumber\\
			& =  2^{k}\int_0^t \left|{\displaystyle \sum_{m=0}^{\infty}}\binom{m+k-1}{m}\frac{(2m+k)l\alpha}{\tau^{\alpha+1}}M_{\alpha}\left(\frac{(2m+k)l}{\tau^{\alpha}}\right)\right|d\tau \nonumber\\
			& = 2^{k}{\displaystyle \sum_{m=0}^{\infty}}\binom{m+k-1}{m}\int_0^t \frac{(2m+k)l\alpha}{\tau^{\alpha+1}}M_{\alpha}\left(\frac{(2m+k)l}{\tau^{\alpha}}\right)d\tau \nonumber\\
			& \leq  2^{k}\sum_{m=0}^{\infty}\binom{m+k-1}{m}
			\exp\left(-A_1(2m+k)^{\frac{1}{1-\alpha}}\right). \label{lam_4_3}
		\end{align}
		Finally, using \eqref{lam_4_2} in \eqref{lam_4_3}, we have our result:
		\begin{align*}
			\| \LL^{-1}\left\{\cosech^{k}(l s^{\alpha})\right\}\|_{L^1(0,t)} & \leq  2^{k}e^{-A_1k^{\frac{1}{1-\alpha}}}\sum_{m=0}^{\infty}\binom{m+k-1}{m}
			\exp(-mA_1B_1)\\
			& \leq  \left(\frac{2}{1-e^{-A_1B_1}}\right)^{k}
			e^{-A_1k^{\frac{1}{1-\alpha}}}.
		\end{align*}
	\end{enumerate}
	\hfill\end{proof}

\begin{lemma} \label{invL_sinh_sinh}
	For $ 0<l_1<l_2 \textrm{ and } k \in \mathbb{N}$, we have
	\begin{enumerate}
		\item[(i)] 
		\begin{multline*} 
			\mathcal{L}^{-1}\left\{\frac{\sinh^{k}((l_2-l_1)s^{\alpha})}{\sinh^{k}(l_2s^{\alpha})}\right\} \\
			= {\displaystyle \sum_{j=0}^{k}}(-1)^j\binom{k}{j} {\displaystyle \sum_{m=0}^{\infty}} \binom{m+k-1}{m}\frac{(2ml_2+kl_1+2j(l_2-l_1))\alpha}{t^{\alpha+1}}M_{\alpha}\left(\frac{(2ml_2+kl_1+2j(l_2-l_1))}{t^{\alpha}}\right),
		\end{multline*}
		\item[(ii)] 
		\begin{equation*} 
			\left\|\mathcal{L}^{-1}\left\{\frac{\sinh^{k}((l_2-l_1)s^{\alpha})}{\sinh^{k}(l_2s^{\alpha})}\right\}\right\|_{L^1(0,t)}
			\leq \left(\frac{1+e^{-A_2}}{1-e^{-B_2C_2}}\right)^k e^{-B_2k^{1/(1-\alpha)}},
		\end{equation*}
		where $B_2 = (1-\alpha)\left(\frac{\alpha}{t}\right)^{\alpha/1-\alpha}l_1^{1/1-\alpha}$, $c = \left\lfloor\frac{1}{1-\alpha}\right\rfloor$, $C_2 = \left[(2l_2/l_1 +k)^c - k^c\right]^{1/c(1-\alpha)}$ and $A_2 =  (1-\alpha)\left(\frac{\alpha}{t}\right)^{\alpha/1-\alpha}(2l_2-2l_1)^{1/1-\alpha}$.
	\end{enumerate}
\end{lemma}
\begin{proof}
	These proofs are similar to Lemma \ref*{invL_cosech}, hence omitted.
	\hfill\end{proof}

\begin{theorem}[Convergence of DNWR for $\nu\leq 1/2$]\label{Theorem2}
In sub-diffusion and diffusion case the interface error $w^{(k)}(t)$ from the error equation \eqref{DNWRL}-\eqref{DNWRL2} follows the following estimates:
 \begin{enumerate}
 	\item[(i)] for $A > B$,
 	\[
 	\left\| w^{(k)}\right\|_{L^{\infty}(0,T)}\leq
 	\left(2\theta^{*}\frac{A-B}{A}\right)^{k}\exp\left(-\mu_1 k^{1/(1-\nu)}\right)\left\| w^{(0)}\right\|_{L^{\infty}(0,T)},
 	\]
 	where, $\mu_1 = (1-\nu)\nu^{\nu/(1-\nu)}(B/T^{\nu})^{1/(1-\nu)}$,
 	\item[(ii)] for $A < B$,
 	\[
 	\left\| w^{(2k)}\right\|_{L^{\infty}(0,T)}\leq \left(\frac{2\sqrt{2}\theta^{*}}{1-\exp(-2\mu_1)}\right)^{2k}\exp\left(-\mu_1 (2k)^{1/(1-\nu)}\right)\left\| w^{(0)}\right\|_{L^{\infty}(0,T)}.
 	\]
 \end{enumerate}
\end{theorem}

\begin{proof}
\begin{enumerate}
	\item[(i)]
	We define $\hat{v}_c^{(k)}(s) := \cosh^k(Bs^{\nu})\hat{w}^{(k)}(s)$. For $k \in \mathbb{N}$, and $\theta=\theta^*$ in \eqref{hrecurrenceLaplace}, $\hat{v}_c^{(k)}(s)$ follows the recurrence relation:
\begin{equation}
	\hat{v}_c^{(k)}(s) = -\theta^{*}\hat{g}_s(s)\hat{v}_c^{(k-1)}(s),
\end{equation} 
	where $\hat{g}_s(s) = \frac{\sinh((B-A)s^{\nu})}{\sinh(As^{\nu})}$. We know from part (i) of Lemma \ref{PositivityLemma} that the inverse Laplace transform $ g_s(t) := \mathcal{L}^{-1}\{\hat{g}_s(s)\}$ exists, so does the inverse Laplace transform of  $\hat{v}_c^{(k)}(s)$. Now using part (iii) of Lemma \ref{SimpleLaplaceLemma} in the above recurrence relation, we have
\begin{align}\label{hkestimate}
	\left\| v^{(k)}_c(.)\right\|_{L^{\infty}(0,T)}
	&=\theta^{*}\left\|(g_s*v_c^{(k-1)})(.)\right\|_{L^{\infty}(0,T)}\\
	&\leq\theta^{*}\left\| g_s(.)\right\|_{L^1(0,T)}\left\| v_c^{(k-1)}(.)\right\|_{L^{\infty}(0,T)} \nonumber \\
	&\leq\left(\theta^{*}\left\| g_s(.)\right\|_{L^1(0,T)}\right)^k\left\| v_c^{(0)}(.)\right\|_{L^{\infty}(0,T)}.\nonumber
\end{align}
Using the positivity of $g_s(t)$ from part (ii) of Lemma \ref{PositivityLemma}, $L^{1}$-integrable from part (ii) of Lemma \ref{invL_sinh_sinh} and finally, from part (iv) of Lemma \ref{SimpleLaplaceLemma}, we have
\begin{equation} 
	\| g_s(.)\|_{L^1(0,T)}
	= \int_0^T |g_s(\tau)|\,d\tau
	\leq \lim_{s\rightarrow0+}(-1)\hat{g}_s(s)
	=\left(\frac{A-B}{A}\right). \label{proofpart2}
\end{equation}
Now,
\begin{align*}
	\hat{w}^{(k)}(s) &= \frac{1}{\cosh^k(Bs^{\nu})}\hat{v}_c^{(k)}(s)\\
	w^{(k)}(t) &= \mathcal{L}^{-1}\left( \frac{1}{\cosh^k(Bs^{\nu})}\right)*v_c^{(k)}(t). 
\end{align*}
Using part (iii) of Lemma \ref{SimpleLaplaceLemma}, we have
\begin{equation}
	\left\|w^{(k)}(.)\right\|_{L^1(0,T)} = \left\|\mathcal{L}^{-1}\left( \frac{1}{\cosh^k(Bs^{\nu})}\right)\right\|_{L^1(0,T)} \left\|v_c^{(k)}(.)\right\|_{L^1(0,T)}. 
\end{equation}
To prove further we first show $\mathcal{L}^{-1}\left\{2^k e^{-kBs^{\nu}} - \sech^k(Bs^{\nu}) \right\} \geq 0$, which is as follows:
\begin{align*}
	2^k e^{-kBs^{\nu}} - \sech^k(Bs^{\nu}) &= 2^k \frac{(1+e^{-2Bs^{\nu}})^k-1}{(e^{Bs^{\nu}}+e^{-Bs^{\nu}})},\\
	&= \sum_{j=0}^{k} \binom{k}{j}e^{-2jBs^{\nu}}\sech^k(Bs^{\nu}).
\end{align*}
We know from part(i) of Lemma \ref{invL_exp} that $\mathcal{L}^{-1}\left\{e^{-2jBs^{\nu}}\right\} \geq 0$ for all $j$ and from Lemma \ref{PositivityLemma} that $\mathcal{L}^{-1}\left\{\sech^k(Bs^{\nu})\right\} \geq 0$. Hence, their convolution is also positive.
Therefore, we have
\begin{align}
	\left\| \mathcal{L}^{-1}\left( \frac{1}{\cosh^k(Bs^{\nu})}\right)\right\|_{L^1(0,T)} 
	&\leq \left\| \mathcal{L}^{-1}\left( \frac{2^k}{\exp(kBs^{\nu})}\right)\right\|_{L^1(0,T)} \nonumber\\
	&\leq 2^k\exp\left(-\mu_1 k^{1/(1-\nu)}\right), \label{proofpart3}
\end{align}
where $\mu = (1-\nu)\nu^{\nu/(1-\nu)}(B/T^{\nu})^{1/(1-\nu)}$.
Hence, we have the super-linear estimate by combining \eqref{hkestimate}-\eqref{proofpart3}
\begin{equation*}
	\left\| w^{(k)}\right\|_{L^{\infty}(0,T)} \leq
	\left(2\theta^{*}\frac{A-B}{A}\right)^{k}\exp\left(-\mu_1 k^{1/(1-\nu)}\right)\left\| w^{(0)}\right\|_{L^{\infty}(0,T)}.
\end{equation*}
	
\item[(ii)]
For $B>A$ consider the iteration,
	\[
	\hat{w}^{(2k)}(s) = (\theta^{*})^2 \frac{\sinh^2((B-A)s^{\nu})}{\sinh^2(As^{\nu})\cosh^2(Bs^{\nu})}\hat{w}^{(2k-2)}(s).
	\]
Now we define the notation $\hat{v}^{(k)}_s(s) := \sinh^k(As^{\nu})\hat{w}^{(k)}(s)$. Using this, we rewrites:
	\[
	\hat{v}^{(2k)}_s(s) = (\theta^{*})^2\hat{g}_c^2(s)\hat{v}^{(2k-2)}_s(s),
	\] 
where $\hat{g}_c(s) = \frac{\sinh((B-A)s^{\nu})}{\cosh(Bs^{\nu})}$.
Denoting the inverse transform of $\hat{g}^2_c(s)$ by $f_c(t)$, we get
\begin{equation} \label{proofpart2_1}
	\left\| v_s^{(2k)}(.)\right\|_{L^{\infty}(0,T)}
	\leq (\theta^{*})^2\left\| f_c(.)\right\|_{L^1(0,T)}\left\| v_s^{(2k-2)}(.)\right\|_{L^{\infty}(0,T)}.
\end{equation}
Applying the hyperbolic identity and triangular inequality of norm we obtain:
\begin{align*}
	\mathcal{L}^{-1}\left\{\hat{g}_c^2(s)\right\} &= \mathcal{L}^{-1}\left\{\frac{\cosh^2((B-A)s^{\nu})}{\cosh^2(Bs^{\nu})}- \frac{1}{\cosh^2(Bs^{\nu})}\right\}, \\
	\left\| f_c(.)\right\|_{L^1(0,T)} &\leq  \left\| \mathcal{L}^{-1}\left\{\frac{\cosh^2((B-A)s^{\nu})}{\cosh^2(Bs^{\nu})}\right\}\right\|_{L^1(0,T)} + \left\| \left\{\frac{1}{\cosh^2(Bs^{\nu})}\right\}\right\|_{L^1(0,T)}.
\end{align*}
Then using part (v) of Lemma \ref{SimpleLaplaceLemma} yields
\begin{equation}\label{proofpart2_2}
	\| f_c(.)\|_{L^1(0,T)} \leq 2.
\end{equation}
Therefore, using part (ii) of Lemma \ref{invL_cosech}, we have:
\begin{equation}\label{proofpart2_3}
	\left\| \mathcal{L}^{-1}\left( \frac{1}{\sinh^{2k}(As^{\nu})}\right)\right\|_{L^1(0,T)} 
	\leq \left(\frac{2}{1-\exp(-2\mu_1)}\right)^{2k}\exp\left(-\mu_1 (2k)^{1/(1-\nu)}\right).
\end{equation} 
Finally combining \eqref{proofpart2_1}-\eqref{proofpart2_3}, we get
\[
	\left\|w^{(2k)}\right\|_{L^\infty(0,T)} \leq  \left(\frac{2\sqrt{2}\theta^{*}}{1-\exp(-2\mu_1)}\right)^{2k}\exp\left(-\mu_1 (2k)^{1/(1-\nu)}\right)\left\|w^{(0)}\right\|_{L^\infty(0,T)}.
\]
\end{enumerate}
\hfill\end{proof}

\begin{theorem}[Convergence of DNWR for $\nu>1/2$]\label{Theorem3}
For diffusion-wave case, the DNWR algorithm converges for $\theta^{*} = 1/(1+\sqrt{\kappa_1/\kappa_2})$ superlinearly with the estimate:
\begin{equation}
	\| w^{(k)}\|_{L^{\infty}(0,T)}\leq
	\left(\frac{2\theta^{*}(1+\exp(-\delta))}{(1-\exp(-\mu_2\beta_1))(1-\exp(-\mu_2\beta_2))}\right)^{k}\exp\left(-2\mu_2 k^{1/(1-\nu)}\right)\| w^{(0)}\|_{L^{\infty}(0,T)},
\end{equation}
where $ c = \floor*{\frac{1}{1-\nu}}, C = \max(A,B), D = \min(A,B) $, $\beta_1 = ((2C/D + k)^c - k^c)^{1/c(1-\nu)},$ $\beta_2 = ((2 + k)^c - k^c)^{1/c(1-\nu)}$, $\mu_2 = (1-\nu)\nu^{\nu/(1-\nu)}(D/t^{\nu})^{1/(1-\nu)}$ and  $\delta = (1-\nu)\nu^{\nu/(1-\nu)}(2|A-B|/t^{\nu})^{1/(1-\nu)}.$ 
\end{theorem}

\begin{proof}
When $A>B$, we rewrite $f_k(t)$ as the convolution $(-1)^k(p_k * q_k)(t)$, where
\begin{equation}
	\LL(p_k(t)) = \hat{p}^k(s) = \frac{\sinh^k((A-B)s^{\nu})}{\sinh^k(As^{\nu})} \quad\text{and}\quad
	\LL(q_k(t)) = \hat{q}^k(s) = \frac{1}{\cosh^k(Bs^{\nu})}.
\end{equation}
Using part (ii) of Lemma \ref{invL_sinh_sinh}, we obtain a bound
\begin{equation}\label{proof3_1}
	\|p_k(t)\|_{L^1(0,T)} \leq \left(\frac{(1+\exp(-\delta))}{1-\exp(-\mu_2\beta_1)}\right)^k\exp\left(-\mu_2 k^{\frac{1}{1-\nu}}\right), 
\end{equation}
where $\beta_1 = \left[\sum_{p=0}^{c-1}\binom{c}{p} (\frac{2A}{B})^{c-p} k^p\right]^{1/c(1-\nu)}$ and  $c = \floor*{\frac{1}{1-\nu}}$ and $\mu_2 = (1-\nu)\nu^{\nu/(1-\nu)}(B/t^{\nu})^{1/(1-\nu)}$.
Now to find the bound of $q_k(t)$ we can use the same procedure as in part (i) of Lemma \ref{invL_cosech} and obtain:
\begin{equation*}
	\mathcal{L}^{-1}\left(\sech^{k}(Bs^{\nu})\right)
	 =  2^{k}{\displaystyle \sum_{m=0}^{\infty}}(-1)^m\binom{m+k-1}{m}\frac{(2m+k)B\nu}{t^{\nu+1}}M_{\nu}\left(\frac{(2m+k)B}{t^{\nu}}\right).
\end{equation*} 
Moreover we have,
\begin{equation*}
	\| \mathcal{L}^{-1}\left\{\sech^{k}(Bs^{\nu})\right\}\|_{L^{1}(0,T)} \leq \| \mathcal{L}^{-1}\left\{\cosech^{k}(Bs^{\nu})\right\}\|_{L^{1}(0,T)}
\end{equation*}
Therefor using part (ii) of Lemma \ref{invL_cosech}, we get
\begin{equation} \label{proof3_2}
	\|q_k(t)\|_{L^1(0,T)} \leq \left(\frac{2}{1-\exp(-\mu_2\beta_2)}\right)^k\exp\left(-\mu_2 k^{\frac{1}{1-\nu}}\right),
\end{equation}
where $\beta_2 = \left[\sum_{p=0}^{c-1}\binom{c}{p} 2^{c-p} k^p\right]^{1/c(1-\nu)}$.
Hence combining \eqref{proof3_1} and \eqref{proof3_2}, we get
\begin{equation} \label{NNWR_AgB}
	\| f_k\|_{L^{1}(0,T)}\leq
	\left(\frac{2\theta^{*}(1+\exp(-\delta))}{(1-\exp(-\mu_2\beta_1))(1-\exp(-\mu_2\beta_2))}\right)^{k}\exp\left(-2\mu_2 k^{1/(1-\nu)}\right).
\end{equation}

For the case $B>A$, we rewrite $f_k(t)$ as the convolution $(-1)^k(r_k * m_k)(t)$, where
$$ \LL(r_k(t)) = \hat{r}^k(s) = \frac{\sinh^k((B-A)s^{\nu})}{\cosh^k(Bs^{\nu})} \quad\text{and}\quad
\LL(m_k(t)) = \hat{m}^k(s) = \frac{1}{\sinh^k(As^{\nu})}.$$
To find an estimate of $r_k(t)$, we can use the same procedure as in part (i) of Lemma \ref{invL_sinh_sinh} and obtain:
\begin{multline*}
	\mathcal{L}^{-1}\left(\frac{\sinh^{k}((B-A)s^{\nu})}{\cosh^{k}(Bs^{\nu})}\right) \\
	= {\displaystyle \sum_{j=0}^{k}}(-1)^j\binom{k}{j} {\displaystyle \sum_{m=0}^{\infty}} (-1)^m\binom{m+k-1}{m}\frac{(2mB+kA+2j(B-A))\nu}{t^{\nu+1}}M_{\nu}\left(\frac{(2mB+kA+2j(B-A))}{t^{\nu}}\right).
\end{multline*}
Moreover we have,
\begin{equation*}
	\left\| \mathcal{L}^{-1}\left\{\frac{\sinh^{k}((B-A)s^{\nu})}{\cosh^{k}(Bs^{\nu})}\right\}\right\|_{L^{1}(0,T)} \leq \left\| \mathcal{L}^{-1}\left\{\frac{\sinh^{k}((B-A)s^{\nu})}{\sinh^{k}(Bs^{\nu})}\right\}\right\|_{L^{1}(0,T)}.
\end{equation*}
Using the same reasoning as in part (i), we obtain the bounds
\begin{equation}
	\|r_k(t)\|_{L^1(0,T)} \leq \left(\frac{(1+\exp(-\delta))}{1-\exp(-\mu_2\beta_1)}\right)^k\exp\left(-\mu_2 k^{\frac{1}{1-\nu}}\right), 
\end{equation}
\begin{equation}
	 \left\|m_k(t)\right\|_{L^1(0,T)} \leq \left(\frac{2}{1-\exp(-\mu_2\beta_2)}\right)^k\exp\left(-\mu_2 k^{\frac{1}{1-\nu}}\right),
\end{equation}
where $\beta_1 =  \left[\sum_{p=0}^{c-1}\binom{c}{p} (\frac{2B}{A})^{c-p} k^p\right]^{1/c(1-\nu)}, \beta_2 = \left[\sum_{p=0}^{c-1}\binom{c}{p} 2^{c-p} k^p\right]^{1/c(1-\nu)}$, $c = \floor*{\frac{1}{1-\nu}}$ and $\mu_2 = (1-\nu)\nu^{\nu/(1-\nu)}(A /t^{\nu})^{1/(1-\nu)}$.
Therefore, we have
\begin{equation}  \label{NNWR_AlB}
	\| f_k\|_{L^{1}(0,T)}\leq
	\left(\frac{2\theta^{*}(1+\exp(-\delta))}{(1-\exp(-\mu_2\beta_1))(1-\exp(-\alpha\beta_2))}\right)^{k}\exp\left(-2\mu_2 k^{1/(1-\nu)}\right)
\end{equation}
Hence from \eqref{NNWR_AgB}\&\eqref{NNWR_AlB}, we have the required estimates.
\hfill\end{proof}

%% file: Proof_NNWR.tex
\section{Convergence of NNWR Algorithm} \label{NNWR_con}
For the convergence study of the NNWR algorithm \eqref{NNWRD}-\eqref{NNWR2} in 1D, the domain $\Omega = (0,L)$ is divided into $N$ non-overlapping subdomains $\Omega_i = (x_{i-1},x_i), i = 1,2,...,N$. Define subdomain length $h_i := x_i - x_{i-1}$ and  different constant diffusion coefficients $\kappa(\boldsymbol{x},t)=\kappa_i \text{ in each } \Omega_i$. Let, $w_i^{(k)}(t)$ be the initial error along the artificial boundary at $\{x_i\}_{i=1}^{N-1}$ and on the original boundary $w_0^{(k)}(t)=0$ at $x=x_0$, $w_N^{(k)}=0$ at $x=x_N$. Applying the Laplace transform in time, the equations \eqref{NNWRD}-\eqref{NNWR2} reduce to:
\begin{equation}\label{NNWRL}
	\begin{array}{rcll}
		\begin{cases}
			s^{2\nu}\hat{u}_{i}^{(k)}=\kappa_i\partial_{xx}\hat{u}_{i}^{(k)}, & \textrm{in } \Omega_i,\\
			\hat{u}_{i}^{(k)} = \hat{w}_{i-1}^{(k-1)}, & \textrm{at }x_{i-1}, i \neq 1,\\
			\hat{u}_{i}^{(k)} = \hat{w}_{i}^{(k-1)}, & \textrm{at }x_{i},i\neq N,\\
		\end{cases}
	\end{array}\!\!\!\!
	\begin{array}{rcll}
		\begin{cases}
			s^{2\nu}\hat{\psi}_{i}^{(k)}=\kappa_i\partial_{xx}\hat{\psi}_{i}^{(k)}, & \textrm{in }\Omega_i,\\
			-\kappa_i\partial_{x}\hat{\psi}_{i}^{(k)} = \kappa_{i-1}\partial_{x}\hat{\psi}_{i-1}^{(k)}-\kappa_{i}\partial_{x}\hat{\psi}_{i}^{(k)}, & \textrm{at }x_{i-1},i \neq 1,\\
			\kappa_i\partial_{x}\hat{\psi}_{i}^{(k)} = \kappa_{i}\partial_{x}\hat{\psi}_{i}^{(k)}-\kappa_{i+1}\partial_{x}\hat{\psi}_{i+1}^{(k)}, & \textrm{at }x_{i},i\neq N,\\
		\end{cases}
	\end{array}
\end{equation}
except for the original boundary at the first and last subdomain, where a homogeneous Dirichlet condition replaces the Neumann boundary condition. Then the trace is updated by
\begin{equation}\label{NNWRL2}
	\hat{w}_i^{(k)}(s)= \hat{w}_i^{(k-1)}(s) - \theta_i(\hat{\psi}_{i}^{(k)}(x_i,s)+\hat{\psi}_{i+1}^{(k)}(x_i,s)),\quad i=1,\ldots,N-1.
\end{equation}
The solutions to the Dirichlet subproblems for $i=1,\ldots,N$ are respectively:
\begin{align}\label{NNWR_1}
	&\hat{u}^{(k)}_{i}(x,s)  \\
	 &=\frac{1}{\sinh(h_{i}\sqrt{\frac{s^{2\nu}}{\kappa_{i}}})}
	\left(\hat{w}^{(k-1)}_{i}(s)\sinh\left((x-x_{i-1})\sqrt{\frac{s^{2\nu}}{\kappa_{i}}}\right)
	+\hat{w}^{(k-1)}_{i-1}(s)\sinh\left((x_{i}-x)\sqrt{\frac{s^{2\nu}}{\kappa_{i}}}\right)\right) \nonumber.
\end{align}
And using the solutions \eqref{NNWR_1}, Neumann subproblems give:
\begin{align*}
	\hat{\psi}^{(k)}_{1}(x,s) &= C^{(k-1)}_{1}(s) \sinh\left((x-x_{0})\sqrt{\frac{s^{2\nu}}{\kappa_1}}\right),\\ 
	\intertext{for $i = 2,\ldots,N-1,$}
    \hat{\psi}^{(k)}_{i}(x,s) &= C^{(k-1)}_{i}(s)\cosh\left((x-x_{i-1})\sqrt{\frac{s^{2\nu}}{\kappa_{i}}}\right)+D^{(k-1)}_{i}(s)\cosh\left((x_{i}-x)\sqrt{\frac{s^{2\nu}}{\kappa_{i}}}\right), \nonumber\\
    \intertext{and}
    \hat{\psi}^{(k)}_{N}(x,s) &= D^{(k-1)}_{N}(s)\sinh\left((x_{N}-x)\sqrt{\frac{s^{2\nu}}{\kappa_N}}
    \right). \nonumber
\end{align*}
For simplification, we use the notation $\sigma_{i}:=\sinh\left(h_{i}\sqrt{s^{2\nu}/\kappa_i}\right)$ and $\gamma_{i}:=\cosh\left(h_{i}\sqrt{s^{2\nu}/\kappa_i}\right)$. Then using the boundary conditions in Neumann subproblems, we obtain:
\begin{align*}
	C^{(k-1)}_{1} & = \frac{1}{\gamma_{1}}\left(\hat{w}^{(k-1)}_{1}\left(
	\frac{\gamma_{1}}{\sigma_{1}}+\sqrt{\frac{\kappa_2}{\kappa_{1}}}\frac{\gamma_{2}}{\sigma_{2}}\right)
	-\sqrt{\frac{\kappa_2}{\kappa_{1}}}\frac{\hat{w}^{(k-1)}_{2}}{\sigma_{2}}\right),\\
	C^{(k-1)}_{i} & =  \frac{1}{\sigma_{i}}\left(\hat{w}^{(k-1)}_{i}
	\left(\frac{\gamma_{i}}{\sigma_{i}}+\sqrt{\frac{\kappa_{i+1}}{\kappa_{i}}}\frac{\gamma_{i+1}}{\sigma_{i+1}}\right)
	-\frac{\hat{w}^{(k-1)}_{i-1}}{\sigma_{i}}-\sqrt{\frac{\kappa_{i+1}}{\kappa_{i}}}\frac{\hat{w}_{i+1}}{\sigma_{i+1}}\right),\quad i = 2,\ldots,N-1, \nonumber\\
	D^{(k-1)}_{i} & =  \frac{1}{\sigma_{i}}\left(\hat{w}^{(k-1)}_{i-1}
	\left(\frac{\gamma_{i}}{\sigma_{i}}+\sqrt{\frac{\kappa_{i-1}}{\kappa_{i}}}\frac{\gamma_{i-1}}{\sigma_{i-1}}\right)
	-\sqrt{\frac{\kappa_{i-1}}{\kappa_{i}}}\frac{\hat{w}^{(k-1)}_{i-2}}{\sigma_{i-1}}-\frac{\hat{w}_{i}}{\sigma_{i}}\right),\quad i = 2,\ldots,N-1, \nonumber\\
	D^{(k-1)}_{N} & = \frac{1}{\gamma_{N}}\left(\hat{w}^{(k-1)}_{N-1}\left(
	\sqrt{\frac{\kappa_{N-1}}{\kappa_{N}}}\frac{\gamma_{N-1}}{\sigma_{N-1}}+\frac{\gamma_{N}}{\sigma_{N}}\right)
	-\sqrt{\frac{\kappa_{N-1}}{\kappa_{N}}}\frac{\hat{w}^{(k-1)}_{N-2}}{\sigma_{N-1}}\right).\nonumber
\end{align*}
Substituting $\hat{\psi}_{i}$ values in updating equation \eqref{NNWRL2} and using the identity $\gamma_{i}^{2}-\sigma_{i}^{2} = 1$, we have
\begin{align}
	\hat{w}_{1}^{(k)} &=\hat{w}_{1}^{(k-1)}-\theta_1\left(\hat{w}_{1}^{(k-1)}
	\left(2+\sqrt{\frac{\kappa_{1}}{\kappa_{2}}}\frac{\gamma_{1}\gamma_{2}}{\sigma_{1}\sigma_{2}}
	+\sqrt{\frac{\kappa_{2}}{\kappa_{1}}}\frac{\sigma_{1}\gamma_{2}}{\gamma_{1}\sigma_{2}}\right)\right. \label{NN1}\\
	&\left.\kern-\nulldelimiterspace+\frac{\hat{w}_{2}^{(k-1)}}{\sigma_{2}}\left(
	\sqrt{\frac{\kappa_{3}}{\kappa_{2}}}\frac{\gamma_{3}}{\sigma_{3}}-\sqrt{\frac{\kappa_{2}}{\kappa_{1}}}\frac{\sigma_{1}}{\gamma_{1}}\right)
	-\sqrt{\frac{\kappa_{3}}{\kappa_{2}}}\frac{\hat{w}_{3}^{(k-1)}}{\sigma_{2}\sigma_{3}}\right), \nonumber\\
		\hat{w}_{i}^{(k)}&=\hat{w}_{i}^{(k-1)}-\theta_i\left(\hat{w}_{i}^{(k-1)}
		\left(2+\left(\sqrt{\frac{\kappa_i}{\kappa_{i+1}}}+\sqrt{\frac{\kappa_{i+1}}{\kappa_{i}}}\right)\frac{\gamma_{i}\gamma_{i+1}}{\sigma_{i}\sigma_{i+1}}\right)
		+\frac{\hat{w}_{i+1}^{(k-1)}}{\sigma_{i+1}}
		\sqrt{\frac{\kappa_{i+2}}{\kappa_{i+1}}}\left(\frac{\gamma_{i+2}}{\sigma_{i+2}}-\sqrt{\frac{\kappa_{i+1}}{\kappa_{i}}}\frac{\gamma_{i}}{\sigma_{i}}\right)\right.\\
		&\left.\kern-\nulldelimiterspace+\frac{\hat{w}_{i-1}^{(k-1)}}{\sigma_{i}}
		\left(\sqrt{\frac{\kappa_{i-1}}{\kappa_{i}}}\frac{\gamma_{i-1}}{\sigma_{i-1}}-\sqrt{\frac{\kappa_{i}}{\kappa_{i+1}}}\frac{\gamma_{i+1}}{\sigma_{i+1}}\right)
		-\sqrt{\frac{\kappa_{i+2}}{\kappa_{i+1}}}\frac{\hat{w}_{i+2}^{(k-1)}}{\sigma_{i+1}\sigma_{i+2}}
		-\sqrt{\frac{\kappa_{i-1}}{\kappa_{i}}}\frac{\hat{w}_{i-2}^{(k-1)}}{\sigma_{i}\sigma_{i-1}}\right), \quad i = 2,\ldots,N-2, \nonumber\\
	\hat{w}_{N-1}^{(k)} &= \hat{w}_{N-1}^{(k-1)}-\theta_{N-1}\left(
	\hat{w}_{N-1}^{(k-1)}\left(2+\sqrt{\frac{\kappa_{N}}{\kappa_{N-1}}}\frac{\gamma_{N-1}\gamma_{N}}{\sigma_{N-1}\sigma_{N}}
	+\sqrt{\frac{\kappa_{N-1}}{\kappa_{N}}}\frac{\sigma_{N}\gamma_{N-1}}{\gamma_{N}\sigma_{N-1}}\right)\right. \label{NN3}\\
	&\left.\kern-\nulldelimiterspace+\frac{\hat{w}_{N-2}^{(k-1)}}{\sigma_{N-1}}\left(
	\sqrt{\frac{\kappa_{N-2}}{\kappa_{N-1}}}\frac{\gamma_{N-2}}{\sigma_{N-2}}-\sqrt{\frac{\kappa_{N-1}}{\kappa_{N}}}\frac{\sigma_{N}}{\gamma_{N}}\right)-\sqrt{\frac{\kappa_{N-2}}{\kappa_{N-1}}}\frac{\hat{w}_{N-3}^{(k-1)}}{\sigma_{N-1}\sigma_{N-2}}\right). \nonumber	
\end{align}
Substituting $\theta_i= \theta_i^{*} := \frac{1}{2+\sqrt{\kappa_i/\kappa_{i+1}}+\sqrt{\kappa_{i+1}/\kappa_{i}}}$ into \eqref{NN1}-\eqref{NN3} leads to the following equations:
\begin{align}
	\hat{w}_{1}^{(k)}&=-\theta_1^*\left[\hat{w}_{1}^{(k-1)}
	\left(\sqrt{\frac{\kappa_{1}}{\kappa_{2}}}\left(\frac{\gamma_{1}\gamma_{2}}{\sigma_{1}\sigma_{2}}-1\right)
	+\sqrt{\frac{\kappa_{2}}{\kappa_{1}}}\left(\frac{\sigma_{1}\gamma_{2}}{\gamma_{1}\sigma_{2}}-1\right)\right)\right. \label{NN4}\\ 
	&\left.\kern-\nulldelimiterspace+\frac{\hat{w}_{2}^{(k-1)}}{\sigma_{2}}\left(
	\sqrt{\frac{\kappa_{3}}{\kappa_{2}}}\frac{\gamma_{3}}{\sigma_{3}}-\sqrt{\frac{\kappa_{2}}{\kappa_{1}}}\frac{\sigma_{1}}{\gamma_{1}}\right)-\sqrt{\frac{\kappa_{3}}{\kappa_{2}}}\frac{\hat{w}_{3}^{(k-1)}}{\sigma_{2}\sigma_{3}}\right],\nonumber 
\end{align}
\begin{align}
	  \hat{w}_{i}^{(k)}&= -\theta_i^*\left[\hat{w}_{i}^{(k-1)}
	  \left(\sqrt{\frac{\kappa_i}{\kappa_{i+1}}}+\sqrt{\frac{\kappa_{i+1}}{\kappa_{i}}}\right)\left(\frac{\gamma_{i}\gamma_{i+1}}{\sigma_{i}\sigma_{i+1}}-1\right) +\frac{\hat{w}_{i+1}^{(k-1)}}{\sigma_{i+1}}
	  \left(\sqrt{\frac{\kappa_{i+2}}{\kappa_{i+1}}}\frac{\gamma_{i+2}}{\sigma_{i+2}}-\sqrt{\frac{\kappa_{i+1}}{\kappa_{i}}}\frac{\gamma_{i}}{\sigma_{i}}\right)\right. \\
	  &\left.\kern-\nulldelimiterspace +\frac{\hat{w}_{i-1}^{(k-1)}}{\sigma_{i}} \left(\sqrt{\frac{\kappa_{i-1}}{\kappa_{i}}}\frac{\gamma_{i-1}}{\sigma_{i-1}}-\sqrt{\frac{\kappa_{i}}{\kappa_{i+1}}}\frac{\gamma_{i+1}}{\sigma_{i+1}}\right)-\sqrt{\frac{\kappa_{i+2}}{\kappa_{i+1}}}\frac{\hat{w}_{i+2}^{(k-1)}}{\sigma_{i+1}\sigma_{i+2}}
	  -\sqrt{\frac{\kappa_{i-1}}{\kappa_{i}}}\frac{\hat{w}_{i-2}^{(k-1)}}{\sigma_{i}\sigma_{i-1}}\right],  \nonumber \\
	\hat{w}_{N-1}^{(k)}&= -\theta_{N-1}^*\left[
	\hat{w}_{N-1}^{(k-1)}\left(\sqrt{\frac{\kappa_{N}}{\kappa_{N-1}}}\left(\frac{\gamma_{N-1}\gamma_{N}}{\sigma_{N-1}\sigma_{N}}-1\right)
	+\sqrt{\frac{\kappa_{N-1}}{\kappa_{N}}}\left(\frac{\sigma_{N}\gamma_{N-1}}{\gamma_{N}\sigma_{N-1}}-1\right)\right) \right. \label{NN6}\\
	&\left.\kern-\nulldelimiterspace +\frac{\hat{w}_{N-2}^{(k-1)}}{\sigma_{N-1}}\left(
	\sqrt{\frac{\kappa_{N-2}}{\kappa_{N-1}}}\frac{\gamma_{N-2}}{\sigma_{N-2}}-\sqrt{\frac{\kappa_{N-1}}{\kappa_{N}}}\frac{\sigma_{N}}{\gamma_{N}}\right) -\sqrt{\frac{\kappa_{N-2}}{\kappa_{N-1}}}\frac{\hat{w}_{N-3}^{(k-1)}}{\sigma_{N-1}\sigma_{N-2}}\right]. \nonumber
\end{align}
Before going to the main theorem for convergence estimate, we first prove a few lemmas.
\begin{lemma} \label{estimate_1}
	For $0 < l_1 < l_2, \alpha \in (0,1)$, the following inequality holds: 
	\begin{equation*}
		\left\|\mathcal{L}^{-1}\left\{\frac{\exp(-l_1s^{\alpha})}{1-\exp(-l_2s^{\alpha})}\right\}\right\|_{L^1(0,t)} \leq \left[1+\frac{t^{\alpha}\Gamma(2-\alpha)}{l_2\Lambda^{(1-\alpha)}}\right] \exp\left(-\Lambda\left(\frac{l_1}{t^{\alpha}}\right)^{1/(1-\alpha)}\right),
	\end{equation*}	
	when $\Lambda = (1-\alpha)\alpha^{\alpha/(1-\alpha)}$.
\end{lemma}	

\begin{proof}
	Let,
	\begin{align*}
		\chi(s) := \frac{\exp(-l_1s^{\alpha})}{1-\exp(-l_2s^{\alpha})}
		= \sum_{n = 0}^{\infty}\exp(-(l_1+nl_2)s^{\alpha}).
	\end{align*}	
	So using the part(iii) of Lemma \ref{invL_exp}, we have
	\begin{align*}
		\left\|\mathcal{L}^{-1}\left\{\chi(s)\right\}\right\|_{L^1(0,t)} &\leq\sum_{n=0}^{\infty} \exp\left(-\Lambda\left(\frac{l_1+nl_2}{t^{\alpha}}\right)^{1/(1-\alpha)}\right)\\
		&\leq \exp\left(-\Lambda\left(\frac{l_1}{t^{\alpha}}\right)^{1/(1-\alpha)}\right) \sum_{n=0}^{\infty} \exp\left(-\Lambda\left(\frac{nl_2}{t^{\alpha}}\right)^{1/(1-\alpha)}\right).
	\end{align*}
	Now, we can further simplify:
	\begin{align*}
		\sum_{n=0}^{\infty} \exp\left(-\Lambda\left(\frac{nl_2}{t^{\alpha}}\right)^{1/(1-\alpha)}\right)
		&= 1+\sum_{n=1}^{\infty} \exp\left(-\Lambda\left(\frac{nl_2}{t^{\alpha}}\right)^{1/(1-\alpha)}\right)\\
		&\leq 1+ \int_{0}^{\infty}\exp\left(-\Lambda\left(\frac{xl_2}{t^{\alpha}}\right)^{1/(1-\alpha)}\right)dx \\
		&= 1+ \frac{t^{\alpha}\Gamma(2-\alpha)}{l_2\Lambda^{(1-\alpha)}}.
	\end{align*}
	These completes the result.
	\hfill\end{proof}	

\begin{lemma} \label{lam_7}
	For $l_0,l,l_1,l_2\rho_1,\rho_2 >0$, $ \beta = \frac{1}{1-\alpha}, \alpha \in(0,1)$ and $L(t) := \frac{t^{\alpha}\Gamma(2-\alpha)}{2\Lambda^{(1-\alpha)}}$ for all $t>0$, the following inequalities hold:
	\begin{enumerate}
		\item [(i)] for $l<l_1,l_2$,
		\begin{align*}
			&\left\|\mathcal{L}^{-1}\left\{\exp(2ls^{\alpha})\left(\frac{\cosh(l_1s^{\alpha})\cosh(l_2s^{\alpha})}{\sinh(l_1s^{\alpha})\sinh(l_2s^{\alpha})}-1\right)\right\}\right\|_{L^1(0,t)} \\
			&\leq 2\left[1+\frac{L(t)}{l_1}\right]\left[1+\frac{L(t)}{l_2}\right]\left(\exp\left(-2\Lambda\left(\frac{l_1-l}{t^{\alpha}}\right)^{\beta}\right)+\exp\left(-2\Lambda\left(\frac{l_2-l}{t^{\alpha}}\right)^{\beta}\right)\right),
		\end{align*}
		\item [(ii)] for $l<l_1,l_2$,
		\begin{align*}
			&\left\|\mathcal{L}^{-1}\left\{\exp(2ls^{\alpha})\left(\frac{\sinh(l_1s^{\alpha})\cosh(l_2s^{\alpha})}{\sinh(l_1s^{\alpha})\cosh(l_2s^{\alpha})}-1\right)\right\}\right\|_{L^1(0,t)} \\
			&\leq 2\left[1+\frac{L(t)}{l_1}\right]\left[1+\frac{L(t)}{l_2}\right]\left(\exp\left(-2\Lambda\left(\frac{l_1-l}{t^{\alpha}}\right)^{\beta}\right)+\exp\left(-2\Lambda\left(\frac{l_2-l}{t^{\alpha}}\right)^{\beta}\right)\right),
		\end{align*}
		\item [(iii)] for $l<l_1,l_2$,
		\begin{align*}
			&\left\|\mathcal{L}^{-1}\left\{\frac{\exp(2ls^{\alpha})}{\sinh(l_1s^{\alpha})\sinh(l_2s^{\alpha})}\right\}\right\|_{L^1(0,t)} \\
			&\leq 4\left[1+\frac{L(t)}{l_1}\right]\left[1+\frac{L(t)}{l_2}\right]\exp\left(-\Lambda\left(\frac{l_1-l}{t^{\alpha}}\right)^{\beta}-\Lambda\left(\frac{l_2-l}{t^{\alpha}}\right)^{\beta}\right),
		\end{align*}
		\item [(iv)] for $2l<l_0$,
		\begin{align*}
			&\left\|\mathcal{L}^{-1}\left\{\frac{\exp(2ls^{\alpha})}{\sinh(l_0s^{\alpha})}\left(\rho_1\frac{\cosh(l_1s^{\alpha})}{\sinh(l_1s^{\alpha})}-\rho_2\frac{\cosh(l_2s^{\alpha})}{\sinh(l_2s^{\alpha})}\right)\right\}\right\|_{L^1(0,t)} \\
			& \leq 2\left[1+\frac{L(t)}{l_0}\right]\left[\rho_1\left[1+\frac{L(t)}{l_1}\right]\left(\exp\left(-2\Lambda\left(\frac{l_0/2-l}{t^{\alpha}}\right)^{\beta}\right)+\exp\left(-2\Lambda\left(\frac{l_0/2+l_1-l}{t^{\alpha}}\right)^{\beta}\right)\right)\right.\\
			&\left.\kern-\nulldelimiterspace+\rho_2\left[1+\frac{L(t)}{l_2}\right]\left(\exp\left(-2\Lambda\left(\frac{l_0/2-l}{t^{\alpha}}\right)^{\beta}\right)+\exp\left(-2\Lambda\left(\frac{l_0/2+l_2-l}{t^{\alpha}}\right)^{\beta}\right)\right)\right],
		\end{align*}
		\item [(v)] for $2l<l_0$,
		\begin{align*}
			&\left\|\mathcal{L}^{-1}\left\{\frac{\exp(2ls^{\alpha})}{\sinh(l_0s^{\alpha})}\left(\rho_1\frac{\cosh(l_1s^{\alpha})}{\sinh(l_1s^{\alpha})}-\rho_2\frac{\sinh(l_2s^{\alpha})}{\cosh(l_2s^{\alpha})}\right)\right\}\right\|_{L^1(0,t)} \\
			& \leq 2\left[1+\frac{L(t)}{l_0}\right]\left[\rho_1\left[1+\frac{L(t)}{l_1}\right]\left(\exp\left(-2\Lambda\left(\frac{l_0/2-l}{t^{\alpha}}\right)^{\beta}\right)+\exp\left(-2\Lambda\left(\frac{l_0/2+l_1-l}{t^{\alpha}}\right)^{\beta}\right)\right)\right.\\
			&\left.\kern-\nulldelimiterspace+\rho_2\left[1+\frac{L(t)}{l_2}\right]\left(\exp\left(-2\Lambda\left(\frac{l_0/2-l}{t^{\alpha}}\right)^{\beta}\right)+\exp\left(-2\Lambda\left(\frac{l_0/2+l_2-l}{t^{\alpha}}\right)^{\beta}\right)\right)\right].
		\end{align*}
	\end{enumerate}	
\end{lemma}	
\begin{proof}
	\begin{enumerate}
		\item [(i)] Upon simplification, we get
		\begin{align*}
			&\left\|\mathcal{L}^{-1}\left\{\exp(2ls^{\alpha})\left(\frac{\cosh(l_1s^{\alpha})\cosh(l_2s^{\alpha})}{\sinh(l_1s^{\alpha})\sinh(l_2s^{\alpha})}-1\right)\right\}\right\|_{L^1(0,t)} \\
			&\leq 2\left\|\mathcal{L}^{-1}\left\{\frac{\exp(-2(l_1-l)s^{\alpha})}{(1-\exp(-2l_1s^{\alpha}))(1-\exp(-2l_2s^{\alpha}))}\right\}\right\|_{L^1(0,t)}\\ &\hspace{10mm}+ 2\left\|\mathcal{L}^{-1}\left\{\frac{\exp(-2(l_2-l)s^{\alpha})}{(1-\exp(-2l_1s^{\alpha}))(1-\exp(-2l_2s^{\alpha}))}\right\}\right\|_{L^1(0,t)}\\
			&=: I_1 + I_2.
		\end{align*}
		Now for the first part, we get by part (ii) of the Lemma \ref{SimpleLaplaceLemma}
		\begin{align*}
			I_1 &= 2\left\|\mathcal{L}^{-1}\left\{\frac{\exp(-2(l_1-l)s^{\alpha})}{(1-\exp(-2l_1s^{\alpha}))(1-\exp(-2l_2s^{\alpha}))}\right\}\right\|_{L^1(0,t)} \\
			& \leq  2\left\|\mathcal{L}^{-1}\left\{\frac{\exp(-(l_1-l)s^{\alpha})}{(1-\exp(-2l_1s^{\alpha}))}\right\}\right\|_{L^1(0,t)}\left\|\mathcal{L}^{-1}\left\{\frac{\exp(-(l_1-l)s^{\alpha})}{(1-\exp(-2l_2s^{\alpha}))}\right\}\right\|_{L^1(0,t)}.
		\end{align*}
		Similar simplification is done for the second part $I_2$. Finally, we have our result using Lemma \ref{estimate_1} .\\
		Proofs of parts (ii) and (iii) are similar to part (i), hence omitted. 
		\item [(iv)] 
		\begin{align*}
			&\left\|\mathcal{L}^{-1}\left\{\frac{\exp(2ls^{\alpha})}{\sinh(l_0s^{\alpha})}\left(\rho_1\frac{\cosh(l_1s^{\alpha})}{\sinh(l_1s^{\alpha})}-\rho_2\frac{\cosh(l_2s^{\alpha})}{\sinh(l_2s^{\alpha})}\right)\right\}\right\|_{L^1(0,t)} \\
			& \leq \rho_1\left\|\mathcal{L}^{-1}\left\{\frac{\exp(2ls^{\alpha})}{\sinh(l_0s^{\alpha})}\frac{\cosh(l_1s^{\alpha})}{\sinh(l_1s^{\alpha})}\right\}\right\|_{L^1(0,t)}+\rho_2\left\|\mathcal{L}^{-1}\left\{\frac{\exp(2ls^{\alpha})}{\sinh(l_0s^{\alpha})}\frac{\cosh(l_2s^{\alpha})}{\sinh(l_2s^{\alpha})}\right\}\right\|_{L^1(0,t)} \\
			& =: I_3 + I_4
		\end{align*}
		Now, \begin{align*}
			I_3 &= 2\rho_1\left\|\mathcal{L}^{-1}\left\{\frac{\exp(-(l_0-2l)s^{\alpha})+\exp(-(l_0+2l_1-2l)s^{\alpha})}{(1-\exp(-2l_0s^{\alpha}))(1-\exp(-2l_1s^{\alpha}))}\right\}\right\|_{L^1(0,t)} \\
			&\leq I_5 +I_6
		\end{align*}
		where
		\begin{align*}
			I_5 &= 2\rho_1\left\|\mathcal{L}^{-1}\left\{\frac{\exp(-(l_0-2l)s^{\alpha})}{(1-\exp(-2l_0s^{\alpha}))(1-\exp(-2l_1s^{\alpha}))}\right\}\right\|_{L^1(0,t)}\\
			&\leq 2\rho_1\left\|\mathcal{L}^{-1}\left\{\frac{\exp(-(l_0/2-l)s^{\alpha})}{(1-\exp(-2l_0s^{\alpha}))}\right\}\right\|_{L^1(0,t)}\left\|\mathcal{L}^{-1}\left\{\frac{\exp(-(l_0/2-l)s^{\alpha})}{(1-\exp(-2l_1s^{\alpha}))}\right\}\right\|_{L^1(0,t)}
		\end{align*}
		similarly for 
		\begin{align*}
			I_6 \leq 2\rho_1\left\|\mathcal{L}^{-1}\left\{\frac{\exp(-(l_0/2+l_1-l)s^{\alpha})}{(1-\exp(-2l_0s^{\alpha}))}\right\}\right\|_{L^1(0,t)}\left\|\mathcal{L}^{-1}\left\{\frac{\exp(-(l_0/2+l_1-l)s^{\alpha})}{(1-\exp(-2l_1s^{\alpha}))}\right\}\right\|_{L^1(0,t)}
		\end{align*}
		similar for $I_4$. Finally, using Lemma \ref{estimate_1} we have our result.
		\item [(v)] The proof is similar to part (iii) and thus omitted.
	\end{enumerate}
	\hfill\end{proof}	

With these preparation we are in position to present the following convergence theorems:
\begin{theorem}[Convergence of NNWR for sub-diffusion and diffusion-wave]
	For $\theta_i = \frac{1}{2+\sqrt{\kappa_i/\kappa_{i+1}}+\sqrt{\kappa_{i+1}/\kappa_{i}}}$, the NNWR algorithm for the fractional order $ (0<\nu<1) $ in \eqref{NNWRD}-\eqref{NNWR2} converges superlinearly with the estimate:
\begin{equation}
	\max_{1 \leq i \leq N-1}\| w_i^{(k)}\|_{L^{\infty}(0,T)}\leq
	c^{k}\exp\left(-\mu (2k)^{1/(1-\nu)}\right)\max_{1 \leq i \leq N-1}\| w_i^{(0)}\|_{L^{\infty}(0,T)},
\end{equation}
where $\mu = (1-\nu)\nu^{\nu/(1-\nu)}(h_{\min}/T^{\nu})^{1/(1-\nu)}$, and value of $c$ is given in \eqref{c_val_super}.
\end{theorem}
\begin{proof}
Define $\sigma := \exp(h_{\min}s^{\nu})$ for $0<\nu<1$, where $h_{\min} = \min_{1\leq i\leq N} h_i/2/\sqrt{\kappa_i}$. Setting 
$\hat{v}^{(k)}_i(s):= \sigma^{2k}\hat{w}^{(k)}_i(s)$ reduce the equations \eqref{NN4}-\eqref{NN6} to:
\begin{align} \label{NN7}
\hat{v}^{(k)}_1(s) &= -\theta_1^{*}\left(\hat{t}_{1,1}\hat{v}^{(k-1)}_1(s)+\hat{t}_{1,2}\hat{v}^{(k-1)}_{2}(s)-\hat{t}_{1,3}\hat{v}^{(k-1)}_{3}(s)\right),  \\
\hat{v}^{(k)}_i(s) &= -\theta_i^{*}\left(\hat{t}_{i,i}\hat{v}^{(k-1)}_i(s)+\hat{t}_{i,i+1}\hat{v}^{(k-1)}_{i+1}(s)+\hat{t}_{i,i-1}\hat{v}^{(k-1)}_{i-1}(s)-\hat{t}_{i,i+2}\hat{v}^{(k-1)}_{i+2}(s)-\hat{t}_{i,i-2}\hat{v}^{(k-1)}_{i-2}(s)\right), \nonumber\\
\hat{v}^{(k)}_{N-1}(s) &= -\theta_{N-1}^{*}\left(\hat{t}_{N-1,N-1}\hat{v}^{(k-1)}_{N-1}(s)+\hat{t}_{N-1,N-2}\hat{v}^{(k-1)}_{N-2}(s)-\hat{t}_{N-1,N-3}\hat{v}^{(k-1)}_{N-3}(s)\right), \nonumber
\end{align}
where the weights are given by
\begin{gather} \label{NNWR3}
\hat{t}_{1,1} = \sigma^{2}\left(\sqrt{\frac{\kappa_1}{\kappa_{2}}}\left(\frac{\gamma_{1}\gamma_{2}}{\sigma_{1}\sigma_{2}}-1\right)+\sqrt{\frac{\kappa_2}{\kappa_{1}}}\left(\frac{\sigma_{1}\gamma_{2}}{\gamma_{1}\sigma_{2}}-1\right)\right), \,
\hat{t}_{1,2} = \frac{\sigma^{2}}{\sigma_{2}}\left(\sqrt{\frac{\kappa_3}{\kappa_{2}}}\frac{\gamma_{3}}{\sigma_{3}}-\sqrt{\frac{\kappa_2}{\kappa_{1}}}\frac{\sigma_{1}}{\gamma_{1}}\right), \\
\hat{t}_{1,3} = \frac{\sigma^{2}}{\sigma_{2}\sigma_{3}}\sqrt{\frac{\kappa_3}{\kappa_{2}}}, \nonumber 
\end{gather}
for $i = 2,\cdots,N-2$,
\begin{gather*} 
\hat{t}_{i,i} = \left(\sqrt{\frac{\kappa_i}{\kappa_{i+1}}}+\sqrt{\frac{\kappa_{i+1}}{\kappa_{i}}}\right)\sigma^{2} \left(\frac{\gamma_i \gamma_{i+1}}{\sigma_i \sigma_{i+1}}-1\right), \,
\hat{t}_{i,i+1} = \frac{\sigma^{2}}{\sigma_{i+1}}\left(\sqrt{\frac{\kappa_{i+2}}{\kappa_{i+1}}}\frac{\gamma_{i+2}}{\sigma_{i+2}}-\sqrt{\frac{\kappa_{i+1}}{\kappa_{i}}}\frac{\gamma_{i}}{\sigma_{i}}\right), \\
\hat{t}_{i,i-1} = \frac{\sigma^{2}}{\sigma_{i}}\left(\sqrt{\frac{\kappa_{i-1}}{\kappa_{i}}}\frac{\gamma_{i-1}}{\sigma_{i-1}}-\sqrt{\frac{\kappa_i}{\kappa_{i+1}}}\frac{\gamma_{i+1}}{\sigma_{i+1}}\right), \,
\hat{t}_{i,i+2} = \frac{\sigma^{2}}{\sigma_{i+1}\sigma_{i+2}}\sqrt{\frac{\kappa_{i+2}}{\kappa_{i+1}}}, \,
\hat{t}_{i,i-2} = \frac{\sigma^{2}}{\sigma_{i}\sigma_{i-1}}\sqrt{\frac{\kappa_{i-1}}{\kappa_{i}}}, \nonumber
\end{gather*}
and
\begin{gather*} 
\hat{t}_{N-1,N-1} = \sigma^{2}\left(\sqrt{\frac{\kappa_N}{\kappa_{N-1}}}\left(\frac{\gamma_{N-1}\gamma_{N}}{\sigma_{N-1}\sigma_{N}}-1\right)+\sqrt{\frac{\kappa_{N-1}}{\kappa_{N}}}\left(\frac{\sigma_{N}\gamma_{N-1}}{\gamma_{N}\sigma_{N-1}}-1\right)\right), \\
\hat{t}_{N-1,N-2} = \frac{\sigma^{2}}{\sigma_{N-1}}\left(\sqrt{\frac{\kappa_{N-2}}{\kappa_{N-1}}}\frac{\gamma_{N-2}}{\sigma_{N-2}}-\sqrt{\frac{\kappa_{N-1}}{\kappa_{N}}}\frac{\sigma_{N}}{\gamma_{N}}\right),\,
\hat{t}_{N-1,N-3} = \frac{\sigma^{2}}{\sigma_{N-1}\sigma_{N-2}}\sqrt{\frac{\kappa_{N-2}}{\kappa_{N-1}}}.
\end{gather*}
Now, we use Lemma \ref{lam_7} on \eqref{NNWR3} to obtain the following estimates:\\
for $i = 1,2,\cdots,N-1$,
\begin{align*}
	||t_{i,i}(.)||_{L^{1}(0,T)} &\leq 2\left(\sqrt{\frac{\kappa_i}{\kappa_{i+1}}}+\sqrt{\frac{\kappa_{i+1}}{\kappa_i}}\right)\left(1+\frac{D}{h_i}\right)\left(1+\frac{D}{h_{i+1}}\right)\left[e^{-2Q_i}+e^{-2Q_{i+1}}\right]=:W_{i,i}\\
\end{align*}
for $i = 1,\cdots,N-2,$
\begin{align*}
	||t_{i,i+1}(.)||_{L^{1}(0,T)} &\leq 2\left(1+\frac{D}{h_{i+1}}\right)\left[\sqrt{\frac{\kappa_{i+2}}{\kappa_{i+1}}}\left(1+\frac{D}{h_{i+2}}\right)\left(e^{-2Q_{(i+1)/2}}+e^{-2Q_{(i+1)/2,i+2}}\right)\right. \\
	&\left.\kern-\nulldelimiterspace+\sqrt{\frac{\kappa_{i+1}}{\kappa_i}}\left(1+\frac{D}{h_i}\right)\left(e^{-2Q_{(i+1)/2}}+e^{-2Q_{(i+1)/2,i}}\right)\right]=:W_{i,i+1}, \\
	||t_{i,i+2}(.)||_{L^{1}(0,T)} &\leq 4\sqrt{\frac{\kappa_{i+2}}{\kappa_{i+1}}}\left(1+\frac{D}{h_{i+1}}\right)\left(1+\frac{D}{h_{i+2}}\right)\left(e^{-(Q_{i+1}+Q_{i+2})}\right)=:W_{i,i+2},
\end{align*}
for $i = 2,\cdots,N-1,$
\begin{align*}
	||t_{i,i-1}(.)||_{L^{1}(0,T)} &\leq 2\left(1+\frac{D}{h_{i}}\right)\left[\sqrt{\frac{\kappa_{i-1}}{\kappa_{i}}}\left(1+\frac{D}{h_{i-1}}\right)\left(e^{-2Q_{i/2}}+e^{-2Q_{i/2,i-1}}\right)\right. \\
	&\left.\kern-\nulldelimiterspace+\sqrt{\frac{\kappa_{i}}{\kappa_{i+1}}}\left(1+\frac{D}{h_{i+1}}\right)\left(e^{-2Q_{i/2}}+e^{-2Q_{i/2,i+1}}\right)\right]=:W_{i,i-1}, \\
	||t_{i,i-2}(.)||_{L^{1}(0,T)} &\leq 4\sqrt{\frac{\kappa_{i-1}}{\kappa_{i}}}\left(1+\frac{D}{h_{i-1}}\right)\left(1+\frac{D}{h_{i}}\right)\left(e^{-(Q_{i-1}+Q_{i})}\right)=:W_{i,i-2},
\end{align*}
where $D = \frac{T^{\nu}\Gamma(2-\nu)}{2\mu^{(1-\nu)}}$, and we use the following notation for convenience: $Q_j =\mu \left[\frac{h_j-h}{T^{\nu}}\right]^{1/(1-\nu)}$, $Q_{j/2} = \mu\left[\frac{h_j/2 - h}{T^{\nu}}\right]^{1/(1-\nu)},$ $ Q_{j/2,k} = \mu\left[\frac{h_j/2 + h_k - h}{T^{\nu}}\right]^{1/(1-\nu)}$.
After taking the inverse Laplace transform in equations \eqref{NN7} and using the part (iii) of Lemma \ref{SimpleLaplaceLemma} we have:
\begin{align} \label{NNWR4}
	||v_1^{(k)}(.)||_{L^{\infty}(0,T)} &\leq c_1 \max_{1 \leq j \leq N-1}||v_1^{(k-1)}(.)||_{L^{\infty}(0,T)}, \nonumber\\
	||v_i^{(k)}(.)||_{L^{\infty}(0,T)} &\leq c_i \max_{1 \leq j \leq N-1}||v_j^{(k-1)}(.)||_{L^{\infty}(0,T)},  \\
	||v_{N-1}^{(k)}(.)||_{L^{\infty}(0,T)} &\leq c_{N-1} \max_{1 \leq j \leq N-1}||v_j^{(k-1)}(.)||_{L^{\infty}(0,T)}. \nonumber
\end{align}
 where
\begin{align*}
	c_1 &=  W_{1,1} + W_{1,1+1} + W_{1,1+2}, \\
	c_i &= W_{i,i} + W_{i,i-1} + W_{i,i+1} + W_{i,i-2} + W_{i,i+2}, \quad i = 2,\cdots,N-2,\\
	c_{N-1} &= W_{N-1,N-1} + W_{N-1,N-2} + W_{N-1,N-3}. \\
\end{align*}
Define for $i = 2,\cdots,N-2$
\begin{equation} \label{c_val_super}
	c := \max\{\theta^*_1c_1,\theta^*_ic_i,\theta^*_{N-1}c_{N-1}\}.
\end{equation}
Then from equations \eqref{NNWR4}, we have 
\[||v_i^{(k)}(.)||_{L^{\infty}(0,T)} \leq c \max_{1 \leq j \leq N-1}||v_j^{(k-1)}(.)||_{L^{\infty}(0,T)}\]
which further yields
\begin{equation*}
	\max_{1 \leq i \leq N-1}||v_i^{(k)}(.)||_{L^{\infty}(0,T)} \leq c^k \max_{1 \leq j \leq N-1}||v_j^{(0)}(.)||_{L^{\infty}(0,T)}.
\end{equation*}
Now since
\[ w_i^{(k)}(t) = \left(\phi^{2k}*v_i^{(k)}\right)(t) = \int_0^{\infty}\phi^{2k}(t-\tau)v_i^k(\tau)d\tau,\]
with $\phi^{k} = \mathcal{L}^{-1}\left\{\exp(-kh_{\min}s^{\nu})\right\}$ for $0<\nu <1.$
And finally using part (iii) of Lemma \ref{SimpleLaplaceLemma} and part (iii) of Lemma \ref{invL_exp} we have the proof.	
\hfill\end{proof}

%% file: Proof_NNWR2D.tex
From now on we define Fourier transform of $f(t,x)$ in $x$ by $\mathcal{F}\{f\}(t,\xi)$, and the Laplace transform in $t$ of $\mathcal{F}\{f\}(t,\xi)$ is as usual denoted by $\widehat{\mathcal{F}\{f\}}(s,\xi)$.
\section{Convergence of NNWR Algorithm in 2D} \label{NNWR2D_con}
For the convergence study of the NNWR \eqref{NNWRD}-\eqref{NNWR2} algorithm in two subdomains in 2D domain $\Omega = (-a,b)\times \mathbb{R}$ is divided into two non-overlapping subdomains $\Omega_1 = (-a,0)\times \mathbb{R} \textrm{ and } \Omega_2 = (0,b)\times \mathbb{R}$. For convenience, we take $\kappa(\boldsymbol{x},t)=\kappa \text{ in } \Omega$. Let, $u_1,\psi_1$ are restricted to $\Omega_{1}$ and $u_2,\psi_2$ are restricted to $\Omega_{2}$. We apply Fourier transform in $y$ and then Laplace transform in $t$ to reduce the iteration as follows:
\begin{equation}\label{NNWR2DL}
	\begin{array}{rcll}
		\begin{cases}
			(s^{2\nu}+\kappa\xi^2)\widehat{\mathcal{F}\{u_1\}}^{(k)} = \kappa\partial_{xx}\widehat{\mathcal{F}\{u_1\}}^{(k)},\\
			\widehat{\mathcal{F}\{u_1\}}^{(k)}(-a,s,\xi) = 0,\\
			\widehat{\mathcal{F}\{u_1\}}^{(k)}(0,s,\xi) = \widehat{\mathcal{F}\{w\}}^{(k-1)}(s,\xi),
		\end{cases}
	\end{array}\!\!
	\begin{array}{rcll}
		\begin{cases}
			(s^{2\nu}+\kappa\xi^2)\widehat{\mathcal{F}\{u_2\}}^{(k)} = \kappa\partial_{xx}\widehat{\mathcal{F}\{u_2\}}^{(k)},\\
			\widehat{\mathcal{F}\{u_2\}}^{(k)}(0,s,\xi) = \widehat{\mathcal{F}\{w\}}^{(k-1)}(s,\xi),\\
			\widehat{\mathcal{F}\{u_2\}}^{(k)}(b,s,\xi) = 0,
		\end{cases}
	\end{array}
\end{equation}
\begin{equation}\label{NNWR2DR}
	\begin{array}{rcll}
		\begin{cases}
			(s^{2\nu}+\kappa\xi^2)\widehat{\mathcal{F}\{\psi_1\}}^{(k)} = \kappa\partial_{xx}\widehat{\mathcal{F}\{\psi_1\}}^{(k)},\\
			\widehat{\mathcal{F}\{\psi_1\}}^{(k)}(-a,s,\xi) = 0,\\
			\partial_{x}\widehat{\mathcal{F}\{\psi_1\}}^{(k)} =
			\partial_{x}\widehat{\mathcal{F}\{u_1\}}^{(k)}-\partial_{x}\widehat{\mathcal{F}\{u_2\}}^{(k)},
		\end{cases}
	\end{array}\!\!\!\!\!\!\!\!\!
	\begin{array}{rcll}
		\begin{cases}
			(s^{2\nu}+\kappa\xi^2)\widehat{\mathcal{F}\{\psi_2\}}^{(k)} = \kappa\partial_{xx}\widehat{\mathcal{F}\{\psi_2\}}^{(k)},\\
			\partial_{x}\widehat{\mathcal{F}\{\psi_2\}}^{(k)} =
			\partial_{x}\widehat{\mathcal{F}\{u_2\}}^{(k)}-\partial_{x}\widehat{\mathcal{F}\{u_1\}}^{(k)},\\
			\widehat{\mathcal{F}\{\psi_2\}}^{(k)}(b,s,\xi) = 0,
		\end{cases}
	\end{array}
\end{equation}
followed by the updating step:
\begin{equation}\label{NNWR2DL2}
	\widehat{\mathcal{F}\{w\}}^{(k)}(s,\xi)= \widehat{\mathcal{F}\{w\}}^{(k-1)}(s,\xi) - \theta\left(\widehat{\mathcal{F}\{\psi_1\}}^{(k)}(0,\xi,s)+\widehat{\mathcal{F}\{\psi_2\}}^{(k)}(0,\xi,s)\right).
\end{equation}
After finding the solutions of \eqref{NNWR2DL}\&\eqref{NNWR2DR} and substituting in \eqref{NNWR2DL2}, we obtain
\begin{equation*}
	\widehat{\mathcal{F}\{w\}}^{(k)} = \left[1-\theta\left(2+\frac{\tanh(B\sqrt{s^{2\nu}+\kappa\xi^2})}{\tanh(A\sqrt{s^{2\nu}+\kappa\xi^2})}+\frac{\tanh(A\sqrt{s^{2\nu}+\kappa\xi^2})}{\tanh(B\sqrt{s^{2\nu}+\kappa\xi^2})}\right)\right]\widehat{\mathcal{F}\{w\}}^{(k-1)},
\end{equation*}
where, $A:=a/\sqrt{\kappa}$ and $B:=b/\sqrt{\kappa}$. Then put $\theta = 1/4$, we obtain
\begin{equation*}
	\widehat{\mathcal{F}\{w\}}^{(k)} = \frac{\sinh^2((A-B)\sqrt{s^{2\nu}+\kappa\xi^2})}{\sinh(2A\sqrt{s^{2\nu}+\kappa\xi^2})\sinh(2B\sqrt{s^{2\nu}+\kappa\xi^2})}\widehat{\mathcal{F}\{w\}}^{(k-1)},
\end{equation*}
that further leads to the relation
\begin{equation}\label{NNWR_2d_1}
	\widehat{\mathcal{F}\{w\}}^{(k)} = \frac{\sinh^{2k}((A-B)\sqrt{s^{2\nu}+\kappa\xi^2})}{\sinh^k(2A\sqrt{s^{2\nu}+\kappa\xi^2})\sinh^k(2B\sqrt{s^{2\nu}+\kappa\xi^2})}\widehat{\mathcal{F}\{w\}}^{(0)}.
\end{equation}

Before going to the main theorem, we need a few lemmas, which estimate the inverse Laplace transformed of a few special functions. 
\begin{lemma} \label{lam_8}
	For  $\eta\geq 0$ and $l>0$, the following results hold:
	\begin{enumerate}
		\item [(i)]
		for $0<\alpha \leq 1/2$,
		\begin{align*}
			&\mathcal{L}^{-1}\left\{\exp(-l\sqrt{s^{2\alpha}+\eta^2})\right\} \\ 
			&= \int_0^{\infty} \exp(- \eta^2\tau)  2\alpha\tau t^{-(2\alpha+1)} M_{\alpha}(\tau t^{-2\alpha}) \frac{l}{\sqrt{4\pi t^3}} \exp(-\frac{l^2}{4\tau}) d\tau,
		\end{align*}
		\item[(ii)]
		for $0<\alpha < 1$, 
		\begin{equation*}
			\mathcal{L}^{-1}\left\{\exp(-l\sqrt{s^{2\alpha}+\eta^2})\right\} = -\int_0^{\infty} l\alpha  t^{-(\alpha+1)} J_0(\eta\tau) \frac{d}{d\tau} M_{\alpha}(\sqrt{\tau^2+l^2} t^{-\alpha})  d\tau,
		\end{equation*}
		where, $J_0(x)$ is the Bessel function of first kind of order zero,
		\item[(iii)] 
		for $l > \alpha^{1-\alpha}t^{\alpha}/(1-\alpha)^{1-\alpha}\alpha^{\alpha}$,
		\begin{equation*}
			\int_0^{\infty} |-l\alpha  t^{-(\alpha+1)}\frac{d}{d\tau} M_{\alpha}(\sqrt{\tau^2+l^2} t^{-\alpha})|  d\tau = \mathcal{L}^{-1}\left\{\exp(-ls^{\alpha})\right\}.
		\end{equation*}
	\end{enumerate}
\end{lemma}
\begin{proof}
	\begin{enumerate}
		\item [(i)]
		We prove this result using Efros theorem, which reads:
		if, $\hat{f}(s)$ and $\hat{g}(s)\exp(-q(s)\tau)$ be the Laplace transform of $f(t)$ and $g(t,\tau)$ respectively in $t$, where $\tau$ is a parameter, then $\hat{f}(q(s))\hat{g}(s)$ is the Laplace transform of $\int_0^{\infty}g(t,\tau)f(\tau)d\tau.$\\
		Let, $\hat{f}(s) = \exp(-l\sqrt{s})$, then the inverse Laplace transform is 
		\[f(t) = \frac{l}{\sqrt{4\pi t^3}} \exp(-\frac{l^2}{4t}).\]
		Choose, $\hat{g}(s) = 1$ and $q(s) = s^{2\alpha} + \eta^2$. Then we have,
		\begin{align*}
			\mathcal{L}^{-1}\left\{\hat{g}(s)\exp(-q(s)\tau)\right\} &= \exp(-\eta^2\tau) \mathcal{L}^{-1}\left\{\exp(-\tau s^{2\alpha})\right\}   \\
			&= \exp(-\eta^2\tau)  2\alpha\tau t^{-(2\alpha +1)} M_{2\alpha}(\tau t^{-2\alpha}).
		\end{align*}
		Therefore, 
		\begin{align} \label{InvLap}
			&\mathcal{L}^{-1}\left\{\exp(-l\sqrt{s^{2\alpha}+\eta^2})\right\} \\  &=\mathcal{L}^{-1}\left\{\hat{g}(s)\hat{f}(q(s))\right\} \nonumber \\
			&= \int_0^{\infty} \exp(- \eta^2\tau)  2\alpha\tau t^{-(2\alpha+1)} M_{2\alpha}(\tau t^{-2\alpha}) \frac{l}{\sqrt{4\pi t^3}} \exp(-\frac{l^2}{4\tau}) d\tau. \nonumber
		\end{align}
		\item[(ii)] 
		Choose $\hat{f}(s) = \exp(-l\sqrt{s^2+\eta^2})$; then, the inverse Laplace transform be 
		\begin{align*}
			\begin{aligned}
				&f(t)=\left \{ \begin{array}{ll}0, &\mbox{ if } 0< t <l,\\
					\delta(t-l) -\frac{l\eta}{\sqrt{t^2-l^2}}J_1(\eta\sqrt{t^2-l^2}), &\mbox{ if } t \geq l. \end{array} \right.
			\end{aligned}
		\end{align*}
		Choose $\hat{g}(s) = 1$ and $q(s) = s^{\alpha}$; then, we obtain
		\begin{align*}
			\mathcal{L}^{-1}\left\{\hat{g}(s)\exp(-q(s)\tau)\right\} &=  \alpha\tau t^{-(\alpha +1)} M_{\alpha}(\tau t^{-\alpha}).
		\end{align*}
		Therefore, by Efros theorem, we have
		\begin{align*}
			& \mathcal{L}^{-1}\left\{\exp(-l\sqrt{s^{2\alpha}+\eta^2})\right\} \nonumber\\
			&= \mathcal{L}^{-1}\left\{\hat{g}(s)\hat{f}(q(s))\right\} \nonumber\\
			&= \int_l^{\infty}  \alpha \tau t^{-(\alpha+1)} M_{\alpha}(\tau t^{-\alpha}) \left[\delta(\tau-l) -\frac{l\eta}{\sqrt{\tau^2-l^2}}J_1(\eta\sqrt{\tau^2-l^2})\right] d\tau \nonumber\\
			&= l \alpha t^{-(\alpha+1)} M_{\alpha}(l t^{-\alpha}) + \int_l^{\infty} l\alpha t^{-(\alpha+1)} M_{\alpha}(\tau t^{-\alpha}) \frac{d}{d\tau} J_0(\eta\sqrt{\tau^2 -l^2}) d\tau \nonumber\\
			&= -\int_l^{\infty} l\alpha t^{-(\alpha+1)} J_0(\eta\sqrt{\tau^2 -l^2}) \frac{d}{d\tau} M_{\alpha}(\tau t^{-\alpha}) d\tau \nonumber\\
			&= -\int_0^{\infty} l\alpha t^{-(\alpha+1)} J_0(\eta\tau) \frac{d}{d\tau} M_{\alpha}(\sqrt{\tau^2 +l^2} t^{-\alpha}) d\tau. 
		\end{align*}
		\item [(iii)]
		From \eqref{lam_3_3}, we have
		\begin{equation*}
			M_{\alpha}(x) = \frac{x^{\alpha/(1-\alpha)}}{\pi(1-\alpha)}\int_0^{\pi}u(\phi)\exp(-u(\phi)x^{\frac{1}{1-\alpha}}) d\phi,
		\end{equation*}
		taking derivative w.r.t. $x$ gives:
		\begin{equation*}
			\frac{d}{dx}M_{\alpha}(x) = \frac{1}{\pi(1-\alpha)^2x^2}\int_0^{\pi}(\alpha-u(\phi)x^{\frac{1}{1-\alpha}})u(\phi)x^{\frac{1}{1-\alpha}}\exp(-u(\phi)x^{\frac{1}{1-\alpha}}) d\phi,
		\end{equation*}
		using part (ii) of Lemma \ref{invL_exp}, we get
		$\frac{d}{dx}M_{\alpha}(x) < 0$, when $x > \alpha^{1-\alpha}/(1-\alpha)^{1-\alpha}\alpha^{\alpha}$.
		i.e. $\frac{d}{d\tau}M_{\alpha}(\sqrt{\tau^2 +l^2} t^{-\alpha}) <0, \forall \tau$, when $l>\alpha^{1-\alpha}t^{\alpha}/(1-\alpha)^{1-\alpha}\alpha^{\alpha}$.
		Finally from part (i) of Lemma \ref{invL_exp}, we have:
		\begin{align*}
			&\int_0^{\infty} |-l\alpha t^{-(\alpha+1)} \frac{d}{d\tau} M_{\alpha}(\sqrt{\tau^2 +l^2} t^{-\alpha})| d\tau \\ 
			&= l\alpha t^{-(\alpha+1)}  M_{\alpha}(lt^{-\alpha})\\
			&= \mathcal{L}^{-1}\left\{\exp(-ls^{\alpha})\right\}. 
		\end{align*}
	\end{enumerate}
	\hfill\end{proof}

\begin{lemma} \label{estimate_9}
	For $\lambda>0$ and $0< l_1<l_2$,
	let 
	\begin{equation} \label{lam_9_1}
		\widehat{\mathcal{F}\{f_1\}}(s,\xi) = \exp(-l_1{\sqrt{s^{2\alpha}+\lambda\xi^2}})\widehat{\mathcal{F}\{f_2\}}(s,\xi).
	\end{equation}
	Then for $\Lambda = (1-\alpha) \alpha^{\alpha/(1-\alpha)}$, the following results hold:
	\begin{enumerate}
		\item [(i)] for $0<\alpha \leq 1/2$,
		\begin{equation*}
			|f_1(t,x)| \leq  \exp\left(-\Lambda\left(\frac{l_1}{t^{\alpha}}\right)^{1/(1-\alpha)}\right) \|f_2(.,.)\|_{L^{\infty}(0,t;L^{\infty})}.
		\end{equation*}
		\item [(ii)] for $1/2<\alpha<1$ and $l_1 > \alpha^{1-\alpha}t^{\alpha}/\Lambda^{\alpha}$,
		\begin{equation*}
			|f_1(t,x)| \leq  \exp\left(-\Lambda\left(\frac{l_1}{t^{\alpha}}\right)^{1/(1-\alpha)}\right) \|f_2(.,.)\|_{L^{\infty}(0,t;L^{\infty})}.
		\end{equation*}
	\end{enumerate}
\end{lemma}
\begin{proof}
	\begin{enumerate}
		\item [(i)]
		Taking inverse Laplace transform on both sides of \eqref{lam_9_1} and using part (i) of Lemma \ref{lam_8}, we have	
		\begin{align}\label{lam_9_2}
			\mathcal{F}\{f_1\}(t,\xi) &= \mathcal{L}^{-1}\left\{\exp(-l_1\sqrt{s^{2\alpha}+\lambda\xi^2})\right\}\ast\mathcal{F}\{f_2\}(t,\xi) \\
			&=\left[\mathcal{L}^{-1}\left\{\exp(-l_1\sqrt{s^{2\alpha}+\lambda\xi^2})\right\}\right]\ast\mathcal{F}\{f_2\}(t,\xi) \nonumber\\
			&=\left[\int_0^{\infty}\exp(-\lambda\xi^2\tau)F(\tau,t)d\tau\right]\ast\mathcal{F}\{f_2\}(t,\xi), \nonumber
		\end{align}
		where, $F(\tau,t) = 2\alpha\tau t^{-(2\alpha+1)} M_{\alpha}(\tau t^{-2\alpha}) \frac{l_1}{\sqrt{4\pi t^3}} \exp(-\frac{l_1^2}{4\tau})$.
		Taking inverse Fourier on both sides of \eqref{lam_9_2} we have
		\begin{align*}
			&f_1(t,x)\\ &= \frac{1}{2\pi}\int_{-\infty}^{\infty}d\xi\exp(i\xi x)\int_0^{t}d\tau_1\int_0^{\infty}d\tau\exp(-\lambda\xi^2\tau)F(\tau,\tau_1)\mathcal{F}\{f_2\}(t-\tau_1,\xi)\\
			&= \frac{1}{2\pi}\int_0^{t}d\tau_1\int_0^{\infty}d\tau F(\tau,\tau_1)\int_{-\infty}^{\infty}d\xi\exp(i\xi x)\exp(-\lambda\xi^2\tau)\mathcal{F}\{f_2\}(t-\tau_1,\xi)\\
			&= \frac{1}{2\pi}\int_0^{t}d\tau_1\int_0^{\infty}d\tau F(\tau,\tau_1)\int_{-\infty}^{\infty}d\xi\exp(i\xi x)\exp(-\lambda\xi^2\tau)\int_{-\infty}^{\infty}dy\exp(-i\xi y)f_2(t-\tau_1,y)\\
			&= \frac{1}{2\pi}\int_0^{t}d\tau_1\int_0^{\infty}d\tau F(\tau,\tau_1)\int_{-\infty}^{\infty}dy f_2(t-\tau_1,y)\int_{-\infty}^{\infty}d\xi\exp(i\xi (x-y))\exp(-\lambda\xi^2\tau)\\	
			&= \int_0^{t}d\tau_1\int_0^{\infty}d\tau F(\tau,\tau_1)\int_{-\infty}^{\infty}dyf_2(t-\tau_1,y)\frac{\exp(-(x-y)^2/4\lambda\tau)}{\sqrt{4\pi\lambda\tau}}.\\	
		\end{align*}
		As $F(\tau,\tau_1)>0$ for $\tau,\tau_1>0$, so taking absolute value on both sides we have
		\begin{align*}
			|f_1(t,x)| &\leq \int_0^{t}d\tau_1\int_0^{\infty}d\tau F(\tau,\tau_1)\int_{-\infty}^{\infty}dy|f_2(t-\tau_1,y)|\frac{\exp(-(x-y)^2/4\lambda\tau)}{\sqrt{4\pi\lambda\tau}}\\
			&\leq \int_0^{t}d\tau_1\int_0^{\infty}d\tau F(\tau,\tau_1)\|f_2(t-\tau_1,.)\|_{L^{\infty}}\int_{-\infty}^{\infty}dy\frac{\exp(-(x-y)^2/4\lambda\tau)}{\sqrt{4\pi\lambda\tau}}\\	
			&\leq \int_0^{t}d\tau_1\int_0^{\infty}d\tau F(\tau,\tau_1)\|f_2(t-\tau_1,.)\|_{L^{\infty}}.
		\end{align*}
		From Efros theorem we know 
		\begin{equation*}
			\int_0^{\infty} F(\tau,\tau_1) d\tau = \mathcal{L}^{-1}\left\{\exp(-l_1s^{\alpha})\right\},
		\end{equation*}
		which gives 
		\begin{equation*}
			|f_1(t,x)| \leq \|f_2(.,.)\|_{L^{\infty}(0,t;L^{\infty})}\int_0^{t}d\tau_1\mathcal{L}^{-1}\left\{\exp(-l_1s^{\alpha})\right\}.
		\end{equation*}
		Therefore from part (iii) of Lemma \ref{invL_exp} we have:
		\begin{equation*}
			|f_1(t,x)| \leq \exp\left(-\Lambda\left(\frac{l_1}{t^{\alpha}}\right)^{1/(1-\alpha)}\right) \|f_2(.,.)\|_{L^{\infty}(0,t;L^{\infty})}.
		\end{equation*}
		\item [(ii)]
		Taking inverse Laplace transform on both sides of \eqref{lam_9_1} and using part (ii) of Lemma \ref{lam_8}, we have		
		\begin{align*}
			\mathcal{F}\{f_1\}(t,\xi) &= \mathcal{L}^{-1}\left\{\exp(-(l_1+nl_2)\sqrt{s^{2\alpha}+\xi^2})\right\}\ast\mathcal{F}\{f_2\}(t,\xi)\\
			&=\left[\mathcal{L}^{-1}\left\{\exp(-(l_1+nl_2)\sqrt{s^{2\alpha}+\xi^2})\right\}\right]\ast\mathcal{F}\{f_2\}(t,\xi)\\
			&=\left[\int_0^{\infty}J_0(\xi\tau)F(\tau,t)d\tau\right]\ast\mathcal{F}\{f_2\}(t,\xi),
		\end{align*}
		where, $F(\tau,t) = -(l_1+nl_2)\alpha t^{-(\alpha+1)} \frac{d}{d\tau} M_{\alpha}(\sqrt{\tau^2 +(l_1+nl_2)^2} t^{-\alpha})$.
		Taking inverse Fourier on both sides, we get
		\begin{align*}
			&f_1(t,x)\\ 
			&= \frac{1}{2\pi}\int_{-\infty}^{\infty}d\xi\exp(i\xi x)\int_0^{t}d\tau_1\int_0^{\infty}d\tau J_0(\xi\tau)F(\tau,\tau_1)\mathcal{F}\{f_2\}(t-\tau_1,\xi)\\
			&= \frac{1}{2\pi}\int_0^{t}d\tau_1\int_0^{\infty}d\tau F(\tau,\tau_1)\int_{-\infty}^{\infty}d\xi\exp(i\xi x)J_0(\xi\tau)\mathcal{F}\{f_2\}(t-\tau_1,\xi)\\
			&= \frac{1}{2\pi}\int_0^{t}d\tau_1\int_0^{\infty}d\tau F(\tau,\tau_1)\int_{-\infty}^{\infty}d\xi\exp(i\xi x)J_0(\xi\tau)\int_{-\infty}^{\infty}dy\exp(-i\xi y)f_2(t-\tau_1,y)\\
			&= \frac{1}{2\pi}\int_0^{t}d\tau_1\int_0^{\infty}d\tau F(\tau,\tau_1)\int_{-\infty}^{\infty}dyf_2(t-\tau_1,y)\int_{-\infty}^{\infty}d\xi\exp(i\xi (x-y))J_0(\xi\tau)\\	
			&= \frac{1}{\pi}\int_0^{t}d\tau_1\int_0^{\infty}d\tau F(\tau,\tau_1)\int_{-\infty}^{\infty}dyf_2(t-\tau_1,y)\frac{\mathbbm{1}_{-\tau<x-y<\tau}}{\sqrt{\tau^2 - (x-y)^2}}.\\	
		\end{align*}
		Taking absolute value on both sides give:
		\begin{align*}
			|f_1(t,x)| &\leq \frac{1}{\pi}\int_0^{t}d\tau_1\int_0^{\infty}d\tau |F(\tau,\tau_1)|\int_{-\infty}^{\infty}dy|f_2(t-\tau_1,y)|\frac{\mathbbm{1}_{-\tau<x-y<\tau}}{\sqrt{\tau^2 - (x-y)^2}}\\
			&\leq \frac{1}{\pi}\int_0^{t}d\tau_1\int_0^{\infty}d\tau |F(\tau,\tau_1)|\|f_2(t-\tau_1,.)\|_{L^{\infty}}\int_{-\infty}^{\infty}dy\frac{\mathbbm{1}_{-\tau<x-y<\tau}}{\sqrt{\tau^2 - (x-y)^2}}\\	
			&\leq \int_0^{t}d\tau_1\int_0^{\infty}d\tau |F(\tau,\tau_1)|\|f_2(t-\tau_1,.)\|_{L^{\infty}}.
		\end{align*}
		Now, using part (iii) of Lemma \ref{lam_8} and for $l_1 > \alpha^{1-\alpha}t^{\alpha}/(1-\alpha)^{1-\alpha}\alpha^{\alpha}$, we get
		\begin{equation*}
			|f_1(t,x)| \leq \|f_2(.,.)\|_{L^{\infty}(0,t;L^{\infty})}\left( \int_0^{t}d\tau_1 \mathcal{L}^{-1}\left\{\exp(-(l_1+nl_2)s^{\alpha})\right\}\right).
		\end{equation*}
		Hence from part (iii) of Lemma \ref{invL_exp}, we obtain
		\begin{equation*}
			|f_1(t,x)| \leq \exp\left(-\Lambda \left(\frac{l_1}{t^{\alpha}}\right)^{1/(1-\alpha)}\right) \|f_2(.,.)\|_{L^{\infty}(0,t;L^{\infty})}.
		\end{equation*}
	\end{enumerate}
	\hfill\end{proof}

For the sake of convenience, we have considered for this theorem the notation $\|.\|$ for $\|.\|_{L^{\infty}(0,T;L^{\infty})}$.
 \begin{theorem}[Convergence of NNWR in 2D]
 	For $\theta = 1/4$, the NNWR algorithm in 2D for sub-diffusion problem converges superlinearly with the estimate:
 	\begin{equation*}
 		\|w^{(k)}\| \leq \left[\frac{\left(1+\exp\left(-P(2|B-A|)^{1/(1-\nu)}\right)\right)^2}{\left(1-\exp(-PEF)\right)\left(1-\exp(-PEH)\right)}\right]^k \exp(-2PEk^{1/(1-\nu)})\|w^{(0)}\|,
 	\end{equation*}
where, $E = (2\min(A,B))^{1/(1-\nu)}$, $F = \left[(2\max(A,B)/\min(A,B) + k)^c - k^c\right]^{1/c(1-\nu)}$, $H = \left[(2 + k)^c - k^c\right]^{1/c(1-\nu)}$, $c = \floor*{\frac{1}{1-\nu}}$, and $P = (1-\nu)(\nu/t)^{\nu/(1-\nu)}$.
A similar estimate holds for diffusion wave case for all $k >K$ s.t. $kB >\nu^{1-\nu}t^{\nu}/(1-\nu)^{1-\nu}\nu^{\nu}$. 
 \end{theorem}
\begin{proof}
	When $A>B$, we choose  $\widehat{\mathcal{F}\{v\}}^{(k)} := \sinh^k(2B\sqrt{s^{2\nu}+\kappa\xi^2})\widehat{\mathcal{F}\{w\}}^{(k)}$; therefore, equation \eqref{NNWR_2d_1} reduce to:
	\begin{align} \label{2DNN3}
		&\widehat{\mathcal{F}\{v\}}^{(k)}  \\ 
		&=\frac{\sinh^{2k}((A-B)\sqrt{s^{2\nu}+\kappa\xi^2})}{\sinh^k(2A\sqrt{s^{2\nu}+\kappa\xi^2})}\widehat{\mathcal{F}\{v\}}^{(0)}  \nonumber\\
		&= \frac{1}{2^k}\sum_{i=0}^{2k}(-1)^i \binom{2k}{i}\sum_{m=0}^{\infty}\binom{m+k-1}{m}\mathcal{L}^{-1}\{\exp(-2(2mA+kB+i(A-B))\sqrt{s^{2\nu}+\kappa\xi^2})\}\widehat{\mathcal{F}\{v\}}^{(0)}.\nonumber
	\end{align}
When $0< \nu \leq 1/2$, using the part (i) of Lemma \ref{estimate_9} and $P = (1-\nu)(\nu/t)^{1/(1-\nu)}$ on equation \eqref{2DNN3} gives
 \begin{align} \label{2DNN1}
	&|v^{(k)}| \\
    &\leq \frac{1}{2^k}\sum_{i=0}^{2k} \binom{2k}{i}\sum_{m=0}^{\infty}\binom{m+k-1}{m}\exp(-P(2(2mA+kB+i(A-B)))^{1/(1-\nu)})\|v^{(0)}\| \nonumber\\
	& \leq \frac{1}{2^k}\sum_{i=0}^{2k}\exp(-P(2(i(A-B)))^{1/(1-\nu)})\sum_{m=0}^{\infty}\binom{m+k-1}{m}\exp(-P(2(2mA+kB))^{1/(1-\nu)})\|v^{(0)}\|. \nonumber
\end{align}
Using \eqref{lam_4_2} we obtain:
\begin{align*}
	|v^{(k)}| \leq \frac{1}{2^k}\left[1+\exp\left(-P(2(A-B))^{1/(1-\nu)}\right)\right]^{2k}\frac{\exp(-PE_1k^{1/(1-\nu)})}{\left(1-\exp(-PE_1F_1)\right)^k}\|v^{(0)}\|, \nonumber
\end{align*}
where, $E_1 = (2B)^{1/(1-\nu)}$ and $F_1 = \left[(2A/B + k)^c - k^c\right]^{1/c(1-\nu)}$.\\
When $1/2< \nu \leq 1$, using the part (ii) of Lemma \ref{estimate_9} and $P = (1-\nu)(\nu/t)^{1/(1-\nu)}$ and $\exists k > K $ s.t. for $kB >\nu^{1-\nu}t^{\nu}/(1-\nu)^{1-\nu}\nu^{\nu}$ similar result as \eqref{2DNN1} holds for \eqref{2DNN3}.
Now $\hat{\tilde{w}}^{(k)} = \frac{1}{\sinh^k(2B\sqrt{s^{2\nu}+\kappa\xi^2})}\hat{\tilde{v}}^{(k)}$.
Using the similar estimate, we obtain
\begin{equation} \label{2DNN2}
	|w^{(k)}| \leq \left[\frac{2}{1-\exp(-PE_1H)}\right]^k \exp(-PEk^{1/(1-\nu)})\|v^{(k)}\|,
\end{equation}
where, $H = \left[(2 + k)^c - k^c\right]^{1/c(1-\nu)}$.
Finally, combining \eqref{2DNN1} and \eqref{2DNN2} and taking the norm on left hand side we get:
\begin{equation} \label{2DNN4}
	\|w^{(k)}\| \leq \left[\frac{\left(1+\exp\left(-P(2(A-B))^{1/(1-\nu)}\right)\right)^2}{\left(1-\exp(-PE_1F_1)\right)\left(1-\exp(-PE_1H)\right)}\right]^k \exp(-2PE_1k^{1/(1-\nu)})\|w^{(0)}\|.
\end{equation}
Similarly, for $B>A$ we can prove
\begin{equation}\label{2DNN5}
	\|w^{(k)}\| \leq \left[\frac{\left(1+\exp\left(-P(2(B-A))^{1/(1-\nu)}\right)\right)^2}{\left(1-\exp(-PE_2F_2)\right)\left(1-\exp(-PE_2H)\right)}\right]^k \exp(-2PE_2k^{1/(1-\nu)})\|w^{(0)}\|,
\end{equation}
where, $E_2 = (2A)^{1/(1-\nu)}$ and $F_2 = \left[(2B/A + k)^c - k^c\right]^{1/c(1-\nu)}$.
Finally, combining \eqref{2DNN4} \& \eqref{2DNN5} we have the proof for all $k \in \mathbb{N}$ for subdiffusion case and for all $k >K$ s.t. $kB >\nu^{1-\nu}t^{\nu}/(1-\nu)^{1-\nu}\nu^{\nu}$ for diffusion wave case.
\hfill\end{proof}

%% file: Numerical.tex
\section{Numerical Experiments}\label{Section7}
We now perform the numerical experiments to measure the convergence rate and verify the optimized relaxation parameter for DNWR and NNWR algorithms for the model problem \cite{sun2006fully}:
\begin{equation}\label{NumericalModelProblem}
\begin{cases}
  D^{2\nu}_{t}u-\nabla\cdot\left(\kappa(\boldsymbol{x})\nabla u\right) = f(\boldsymbol{x}), \qquad& (\boldsymbol{x},t)\in\Omega\times (0,T),\\
  u(\boldsymbol{x},t)=0, & (\boldsymbol{x},t) \in \partial\Omega \times (0,T), \\
  u(\boldsymbol{x},0)=g(\boldsymbol{x}), \; \partial_t u(\boldsymbol{x},0) = 0, & \boldsymbol{x}\in\Omega.
\end{cases}
\end{equation}
In the sub-diffusion case, we discretize \eqref{NumericalModelProblem} using central finite difference in space and $L1$ scheme on graded mesh, \cite{stynes2017error} for the Caputo fractional time derivative. In the diffusion-wave case, we use a fully discrete difference scheme; see \cite{sun2006fully}. We choose the spatial grid according to the subdomain sizes. When the spatial grid sizes are different for subdomains overlapping grid is used for the ghost points. Thus one can incorporate different physics of a problem on different subdomains. For more details see \cite{henshaw2009composite,giles1997stability}

\subsection{DNWR Algorithm}
We have taken the model problem ~\eqref{NumericalModelProblem} with $F(x) = \sin(\pi x/2)$ and $g(x) = 0$ on the spatial domain $\Omega = (0,2)$ and for the time window $T = 1$. The spatial grid size is $\Delta x = 0.01$, and the number of temporal grids is $2^6$ as we have chosen graded mesh, so the temporal grid size varies for the sub-diffusion case. Our experiments will illustrate the DNWR method for two subdomain cases.

In Figure~\ref{NumFig01}, we compare the convergence rate of the DNWR for different values of $\theta$ for fractional order $2\nu = 0.5$. we consider On the left, $a= 0.5, b = 1.5$, and on the middle, $a=b= 1$, and on the right $a= 1.5, b= 0.5$. We run the same set of experiments in ~\ref{NumFig02} where $2\nu = 1$ on the left and $2\nu = 1.5$ on the right, respectively, for $a<b$. The $a>b$ case behaves similarly to $a<b$. In Figure~\ref{NumFig03}, we run the same set of experiments for the heterogeneous space grid, choosing $\Delta x_1 = 0.01$ and $\Delta x_2 = 0.005$ with diffusion parameter $\kappa_{1} = 1$ and $\kappa_{2} = 0.25$. Comparing these plots, we can say that $\theta = 1/(1+\sqrt{\kappa_{1}/\kappa_{2}})$  may not always be an optimal convergence rate up to tolerance, but it gives superlinear one. For more details, see \cite{mandal2021substructuring}.

In Figure~\ref{NumFig05}, we compare the convergence rates for different values of the fractional order. We can see that the larger the fractional order, the faster the convergence. 

In Figures~\ref{NumFig06}, ~\ref{NumFig07}, ~\ref{NumFig08} and  ~\ref{NumFig09}, we compare the numerical convergence rate, theoretical convergence rate, and superlinear error bound for sub-diffusion and super-diffusion case, choosing  $\kappa_{1} = 1$ and $\kappa_{2} = 0.25$. In Figure~\ref{NumFig06} and ~\ref{NumFig07}, we show the comparison for  $a>b$, that corresponds to the case of $A = 1.5$, $B = 1$, $A > B$ in Theorem \ref{Theorem2}. We consider the initial guess $h^{(0)}(t)=1, t\in(0,T]$. 

In Figure~\ref{NumFig08} and ~\ref{NumFig09}, we repeat the experiments by swapping the roles of the two subdomains so that $a<b$. This corresponds to $A=0.5$ and $B=3$, as in Theorem~\ref{Theorem3}. The diffusion coefficients and initial guesses are the same as earlier.
\begin{figure}
  \centering
  \includegraphics[width=0.30\textwidth]{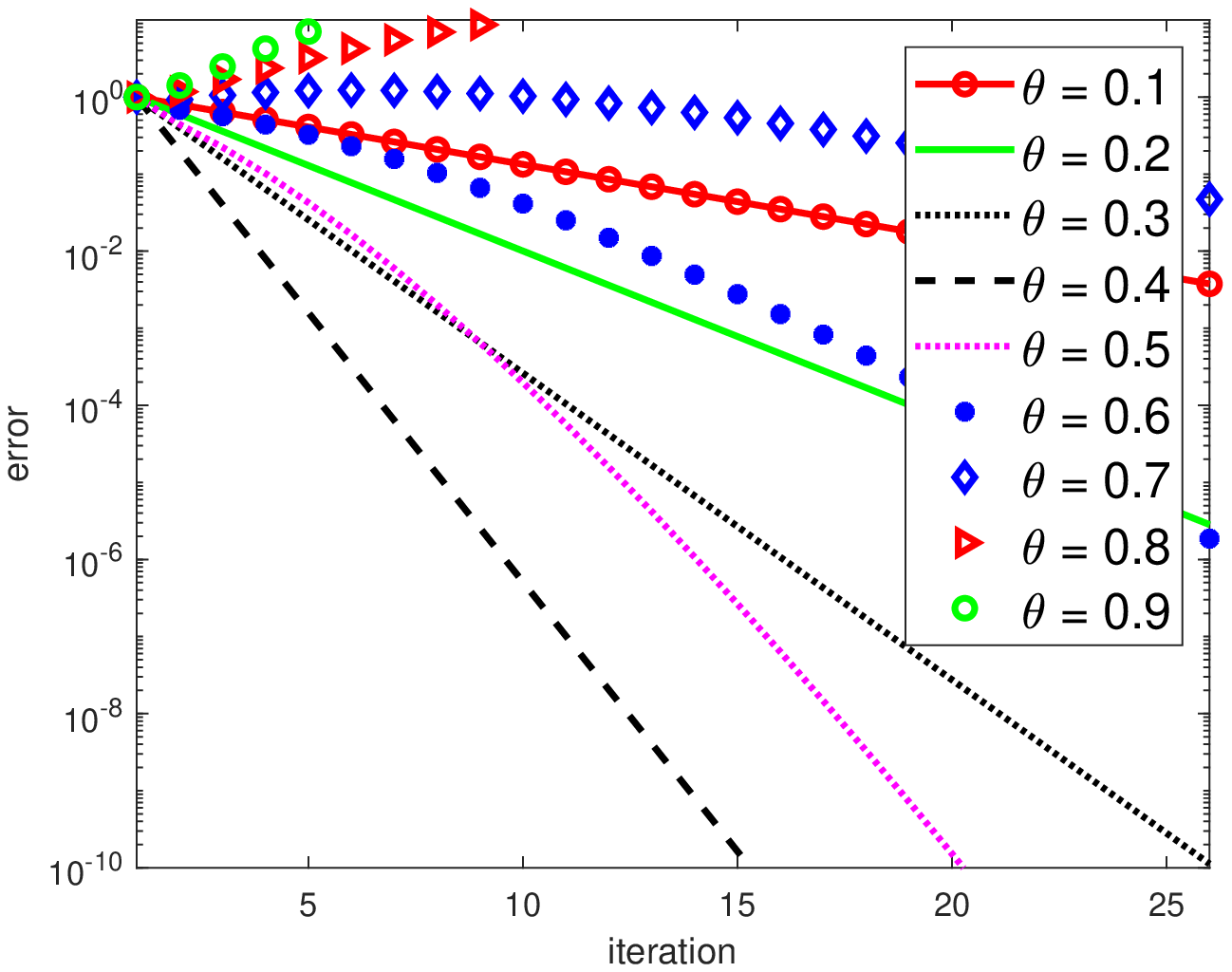}
  \includegraphics[width=0.30\textwidth]{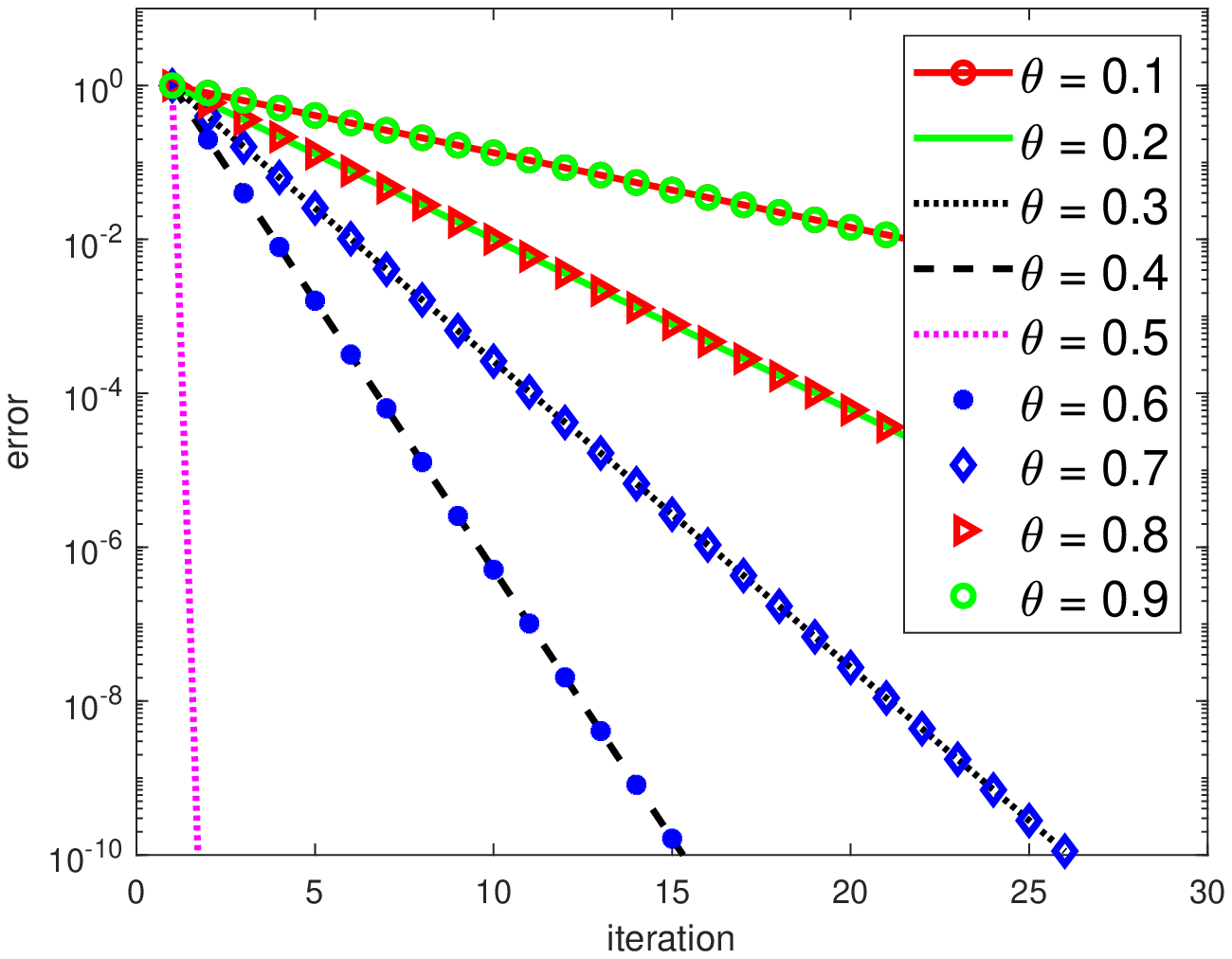}
  \includegraphics[width=0.30\textwidth]{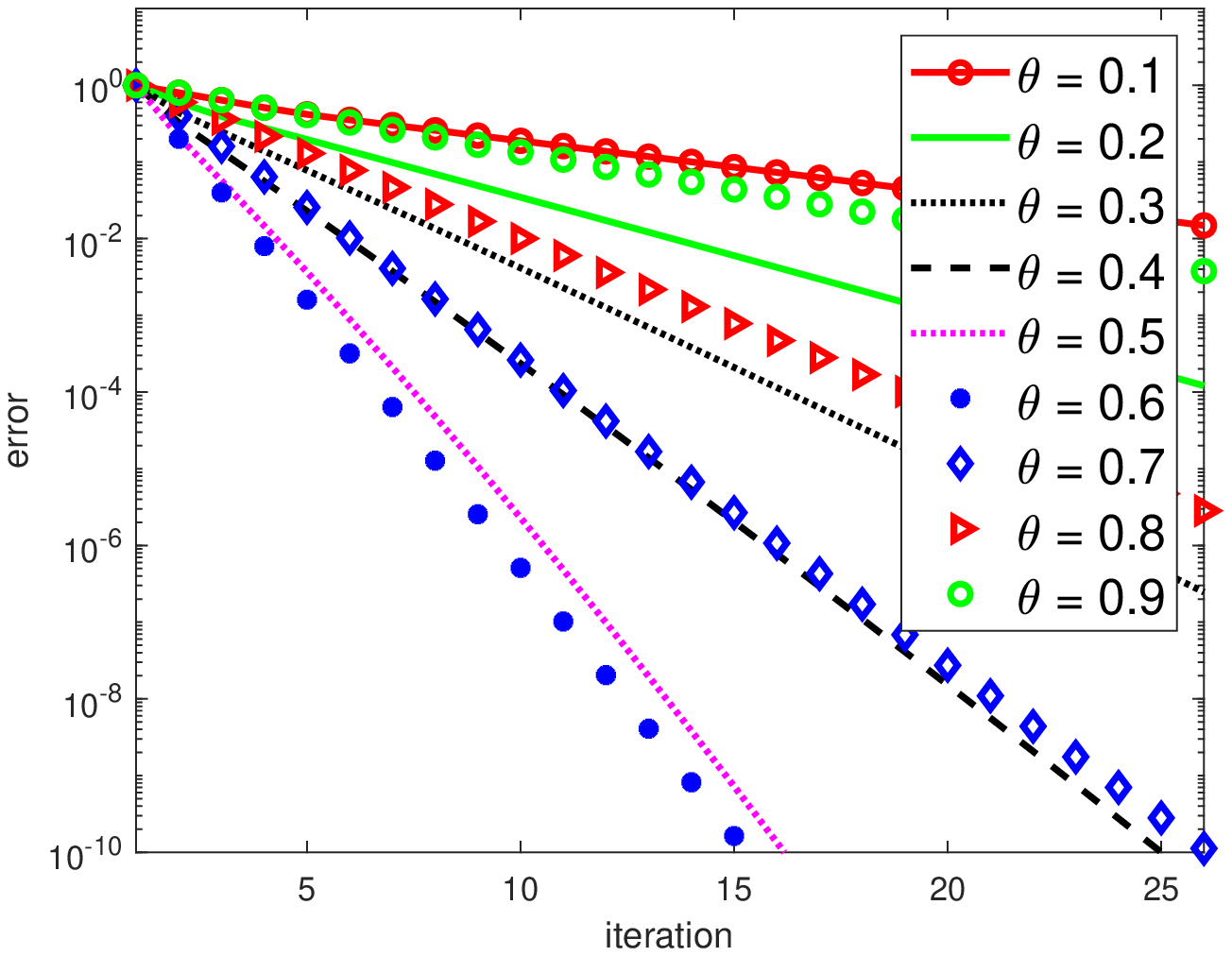}
  \caption{Convergence of DNWR for $2\nu=0.5$ for various relaxation
    parameters $\theta$ for $T=1$, on the left for $a<b$ at the middle $a=b$ and on the right for $a>b$}
  \label{NumFig01}
\end{figure}
\begin{figure}
  \centering
  \includegraphics[width=0.45\textwidth]{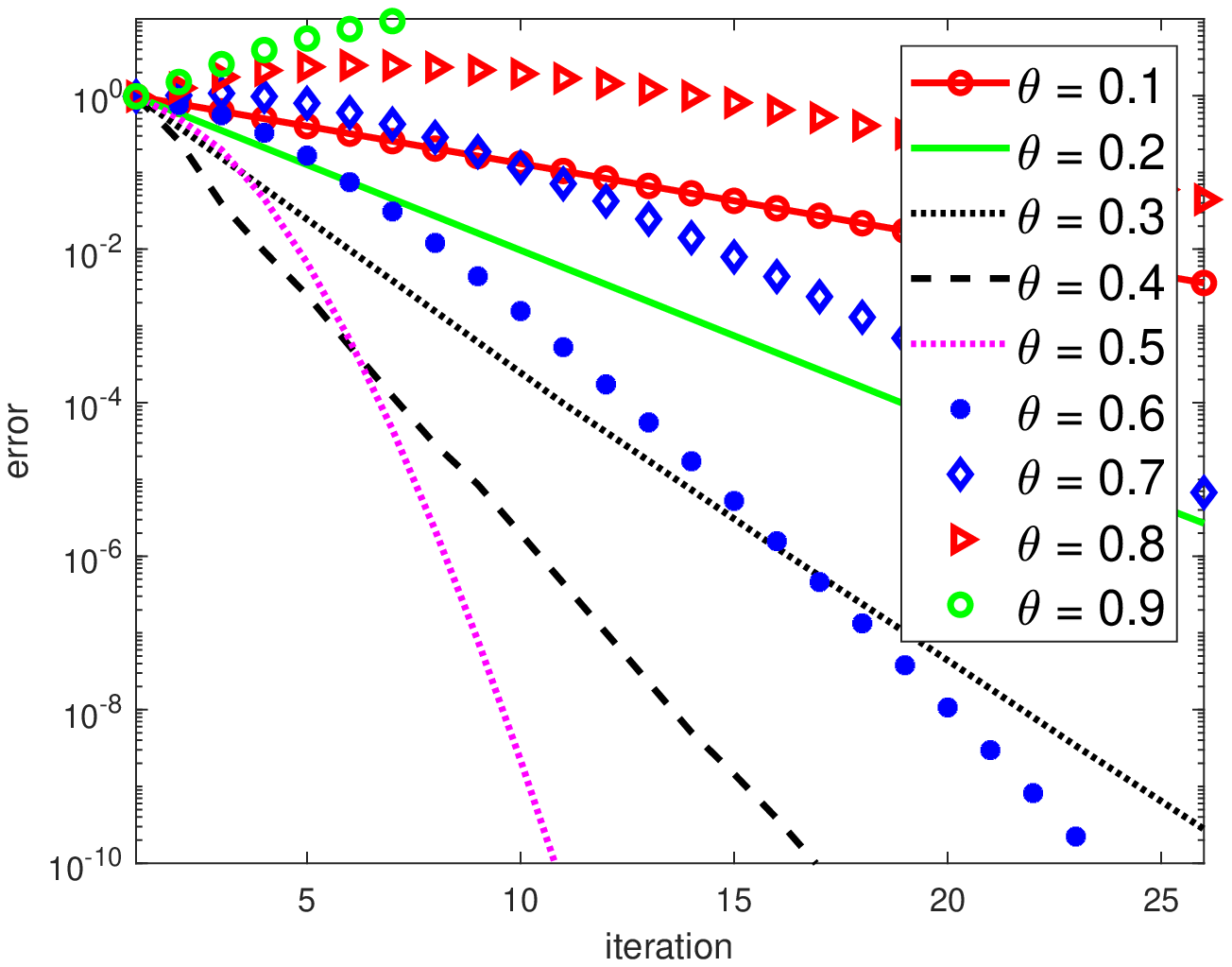}
  \includegraphics[width=0.45\textwidth]{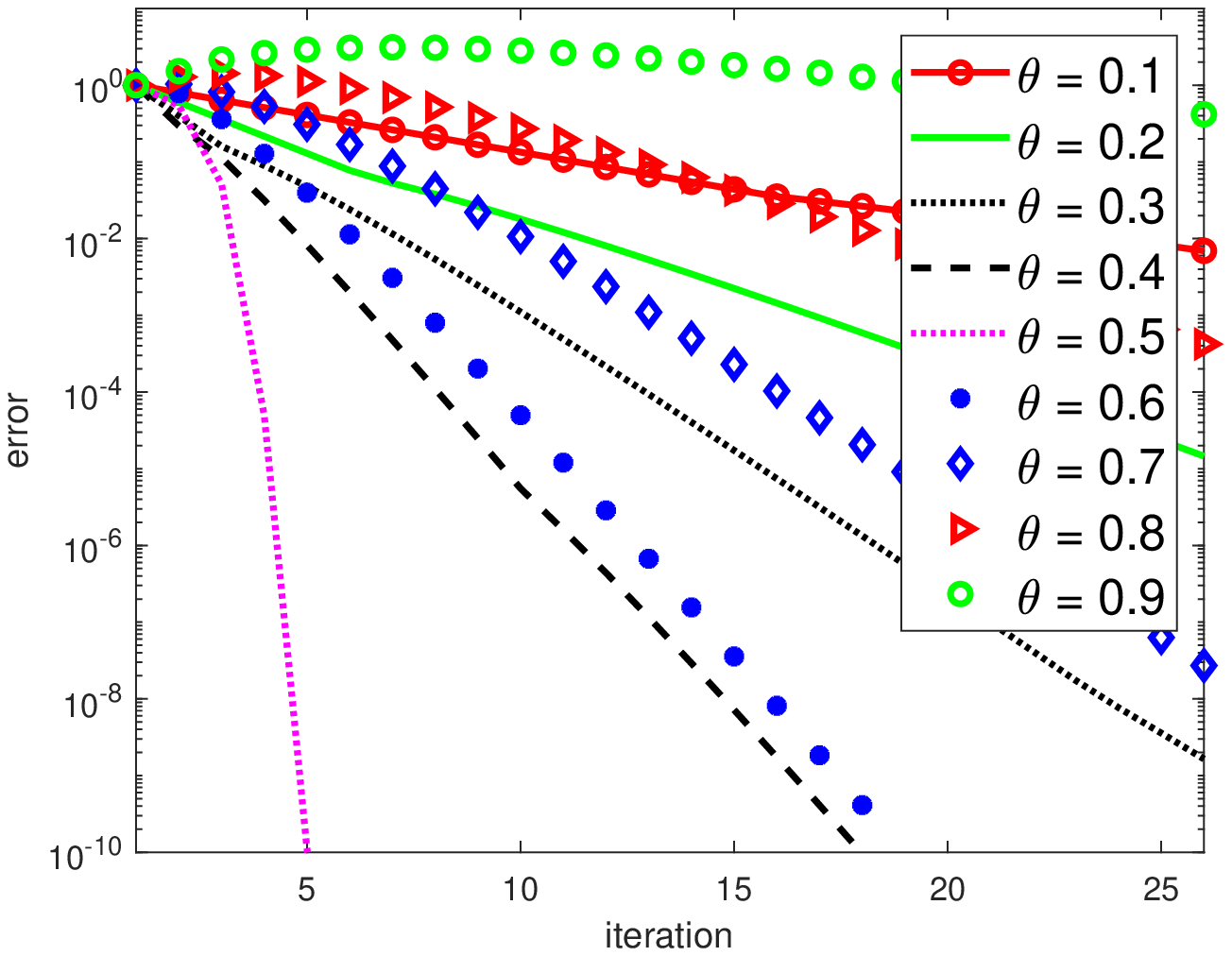}
 \caption{Convergence of DNWR for $a<b$ using various relaxation
 	parameters $\theta$ for $T=1$, on the left for $2\nu = 1$ and on the right for $2\nu = 1.5$}
  \label{NumFig02}
\end{figure}
\begin{figure}
	\centering
\includegraphics[width=0.30\textwidth]{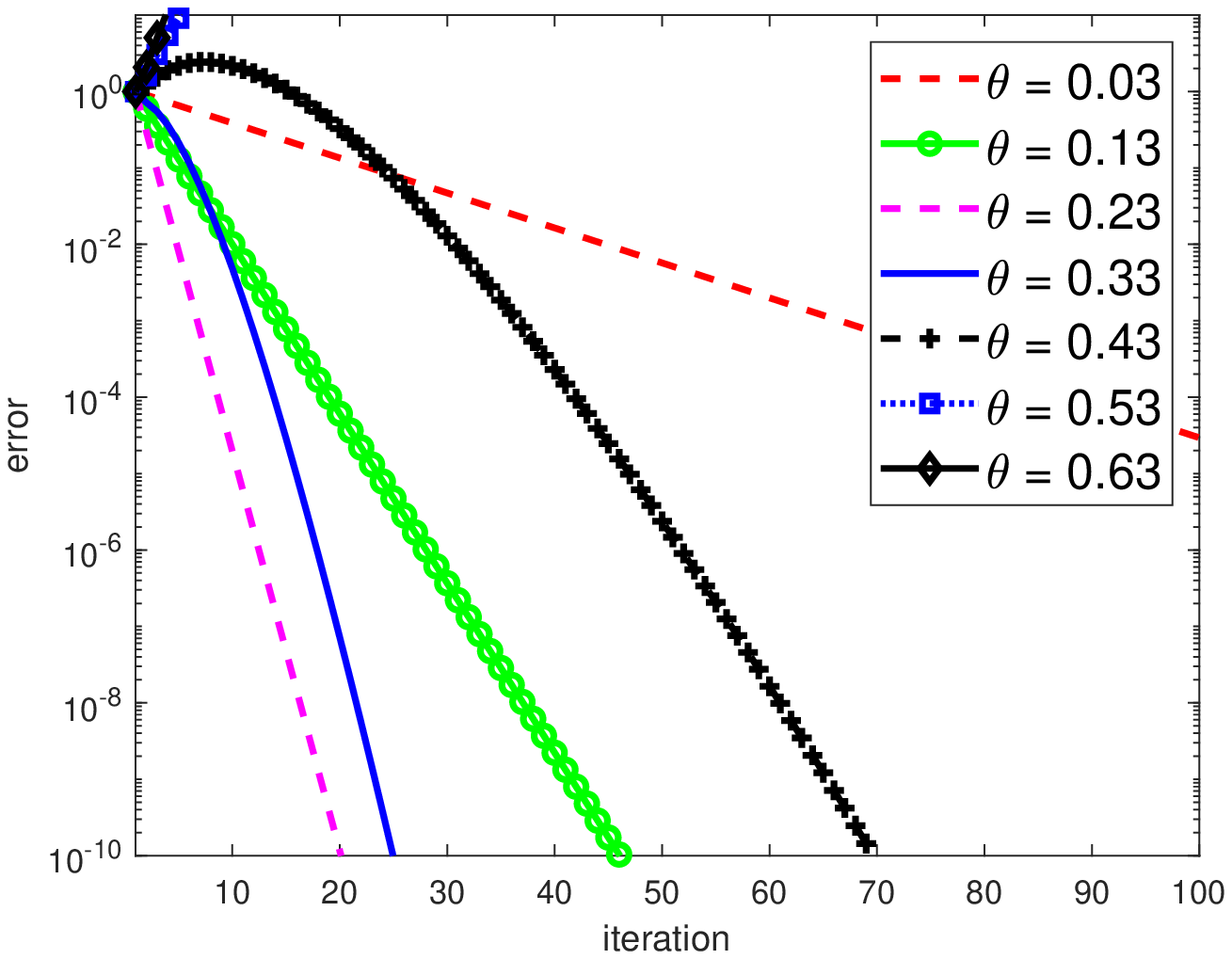}
\includegraphics[width=0.30\textwidth]{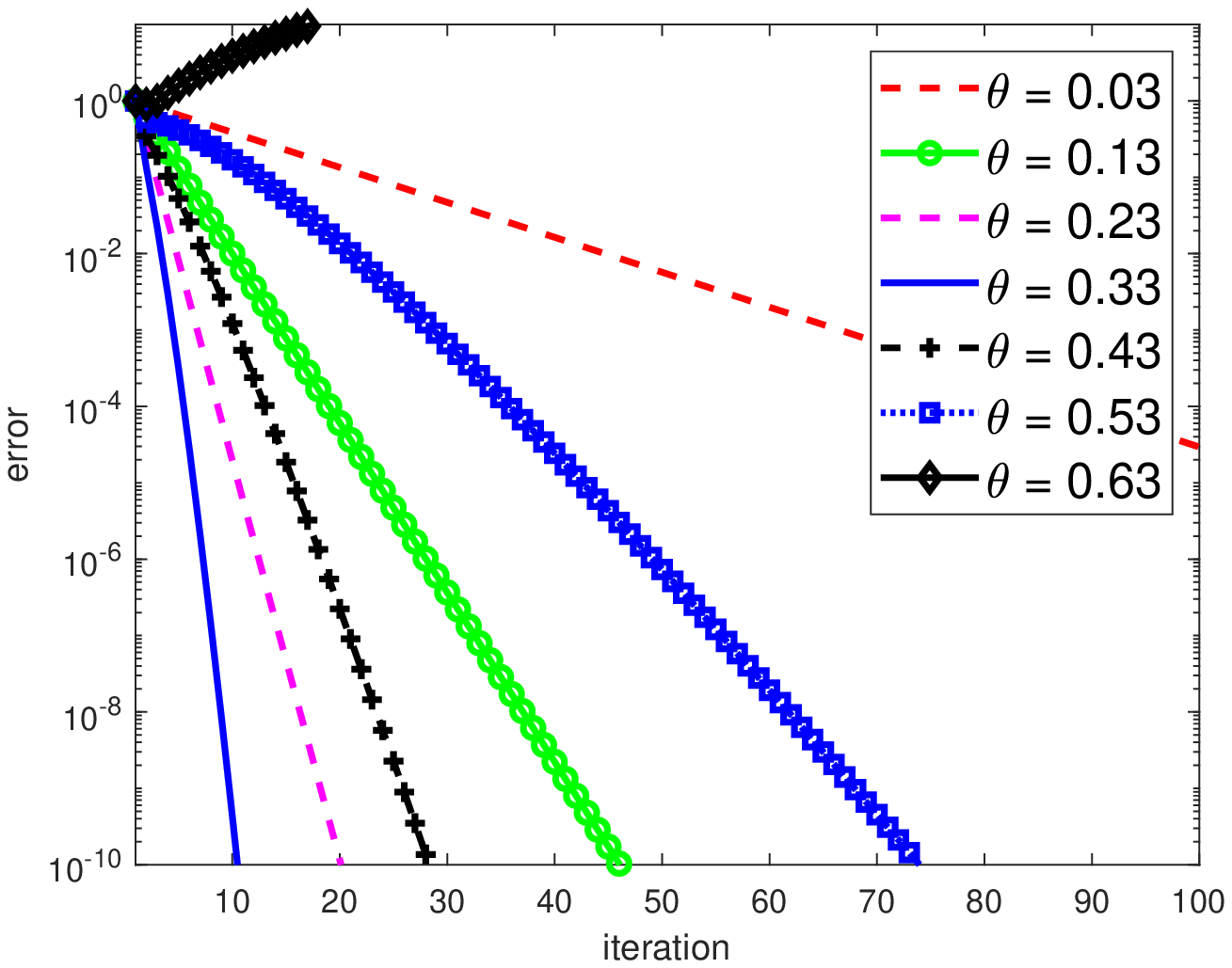}
\includegraphics[width=0.30\textwidth]{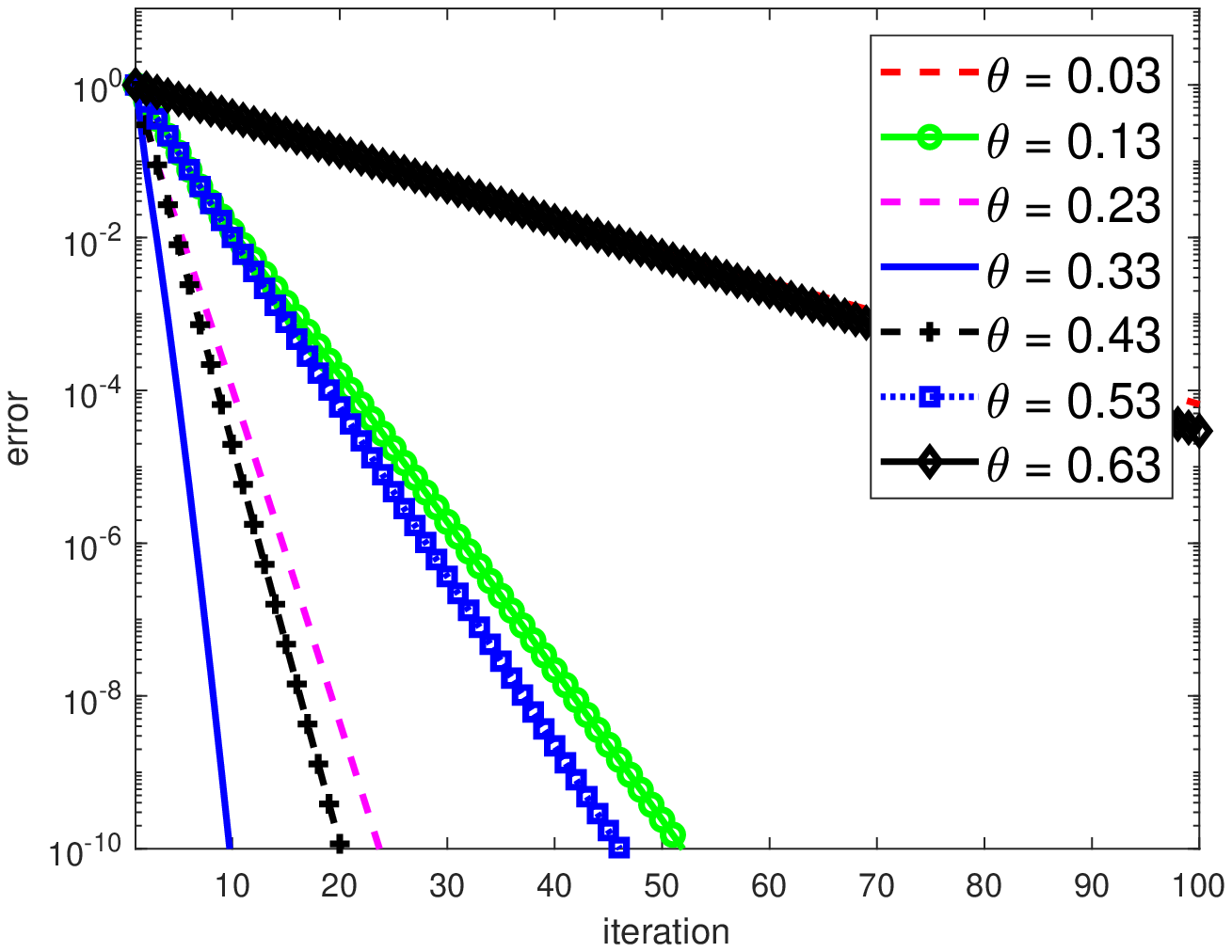}
	\caption{Convergence of DNWR for $2\nu=0.5$ using various relaxation
		parameters $\theta$ and heterogeneous space grid for $T=1$, on the left for $a<b$, middle $a=b$ and on the right for $a>b$}
	\label{NumFig03}
\end{figure}
\begin{figure}
	\centering
	\includegraphics[width=0.45\textwidth]{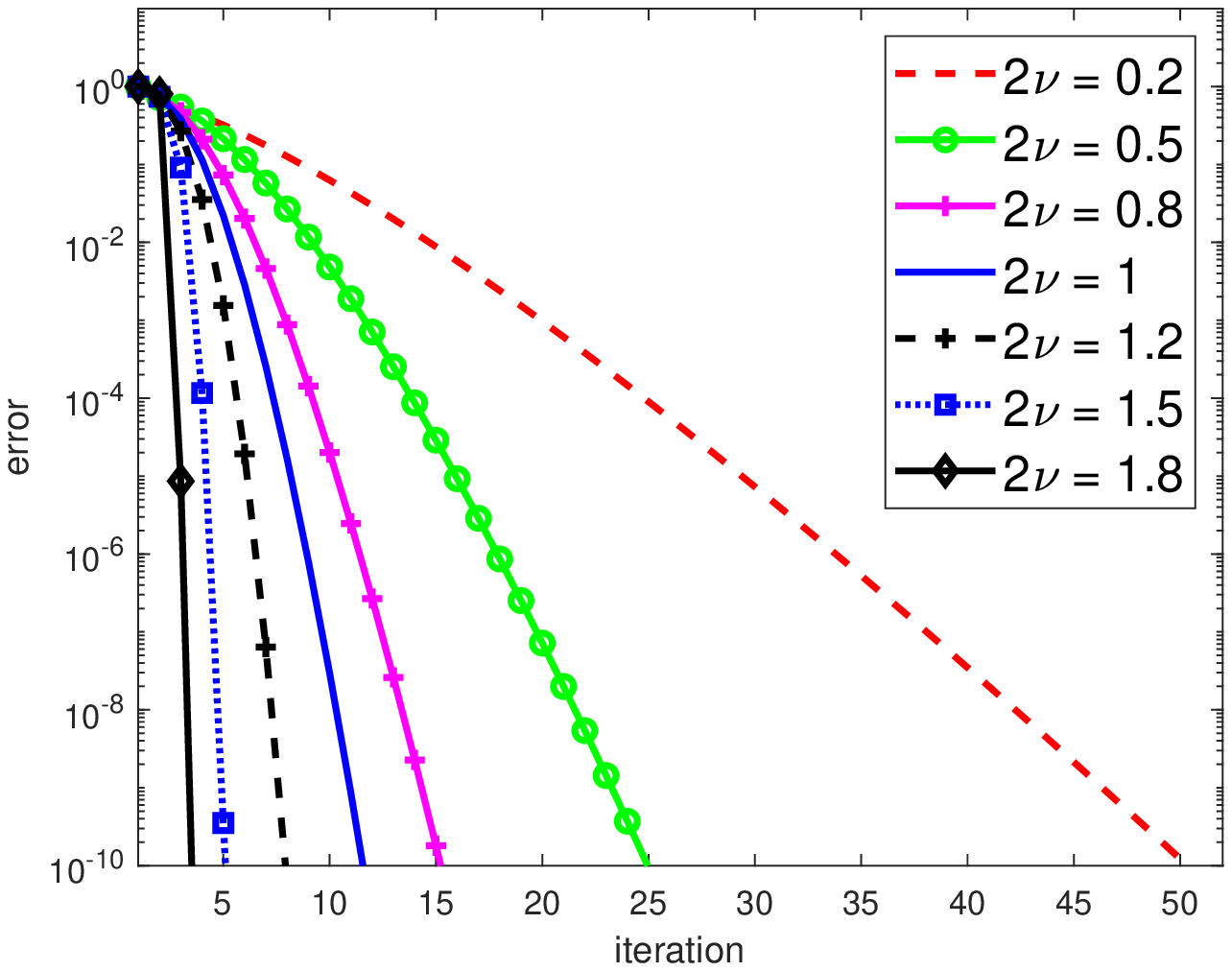}
	\includegraphics[width=0.45\textwidth]{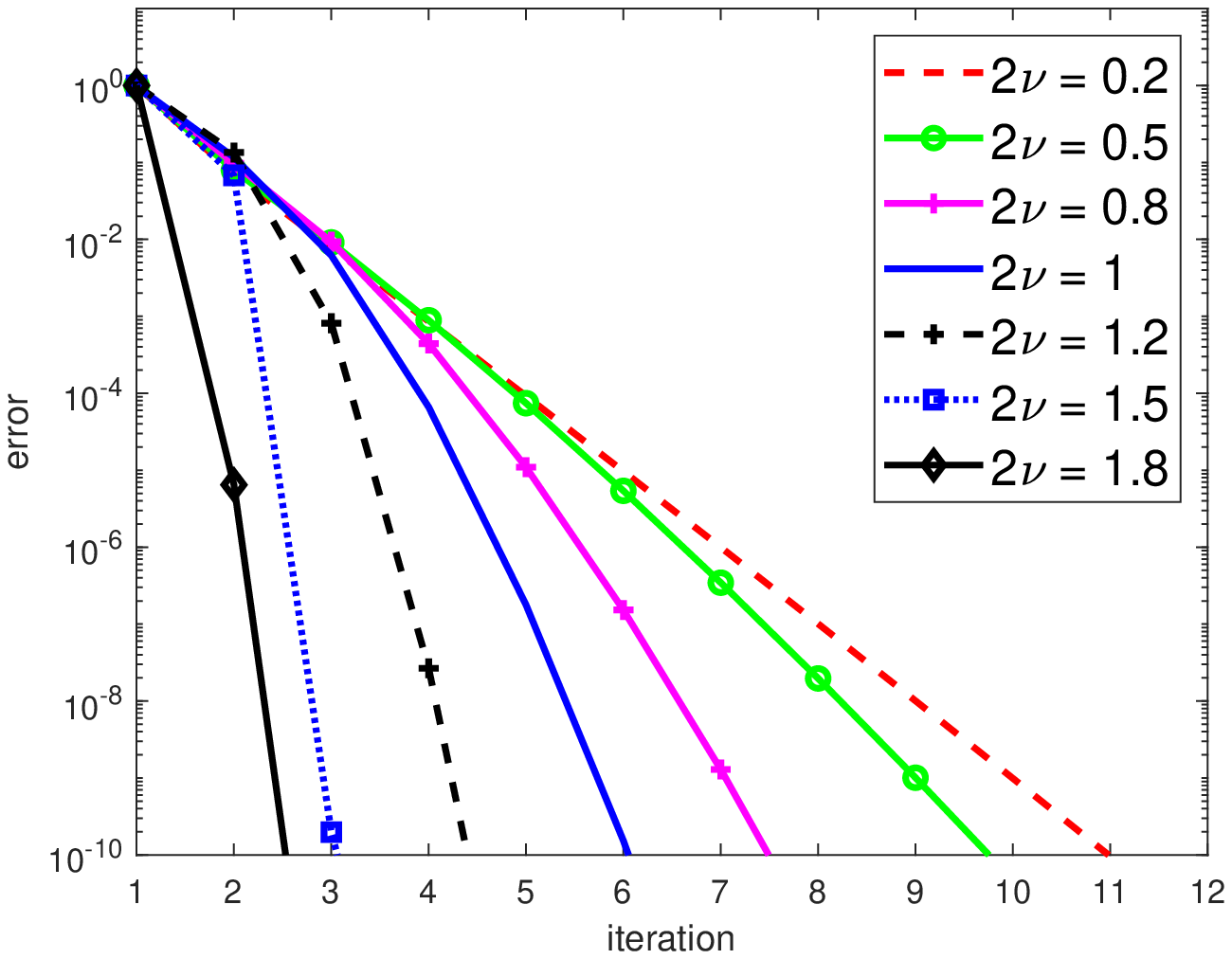}
	\caption{Convergence of DNWR for $\theta = 0.33$ using various fractional order $2\nu$ for $T=1$, on the left for $a<b$ and on the right for $a>b$}
	\label{NumFig05}
\end{figure}
\begin{figure}
	\centering
	\includegraphics[width=0.30\textwidth]{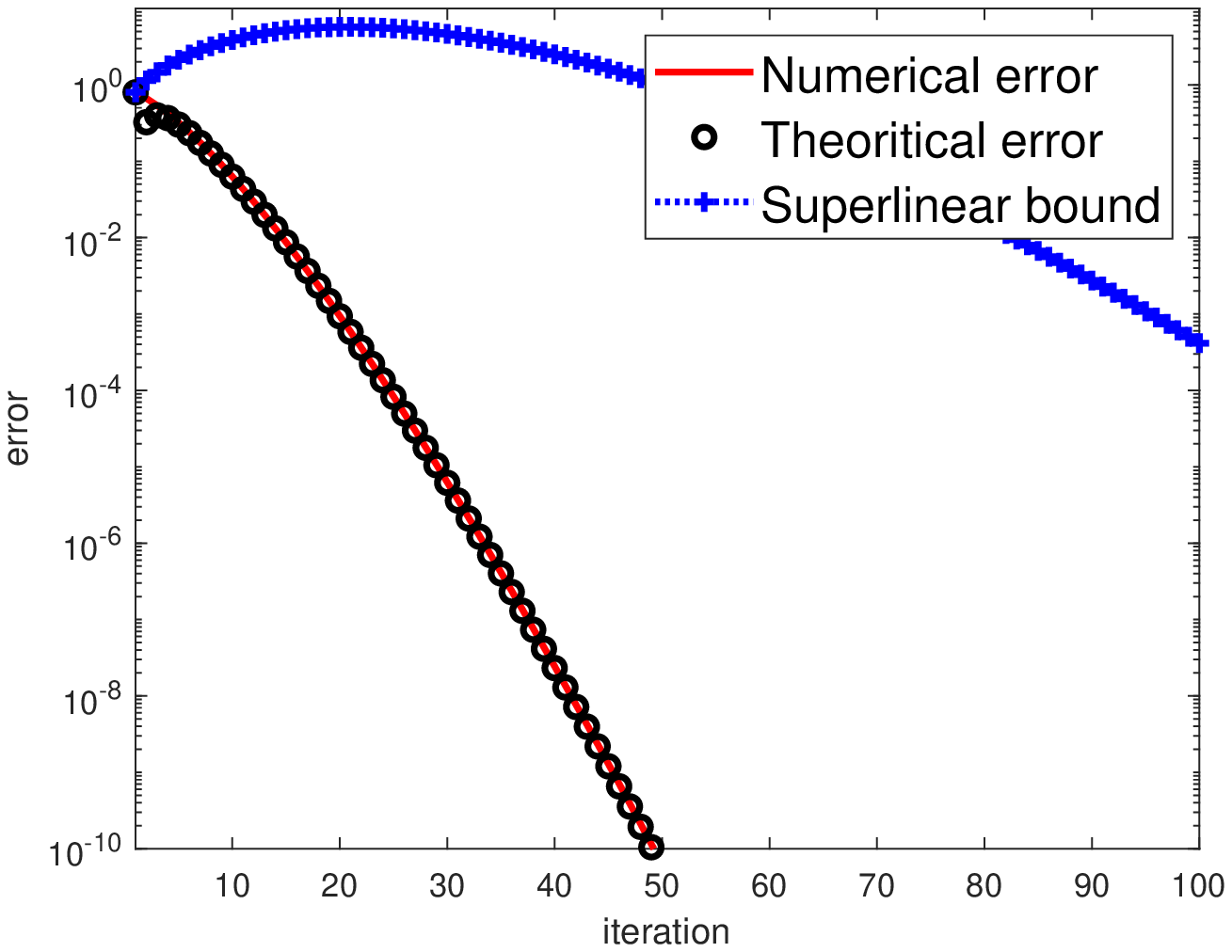}
	\includegraphics[width=0.30\textwidth]{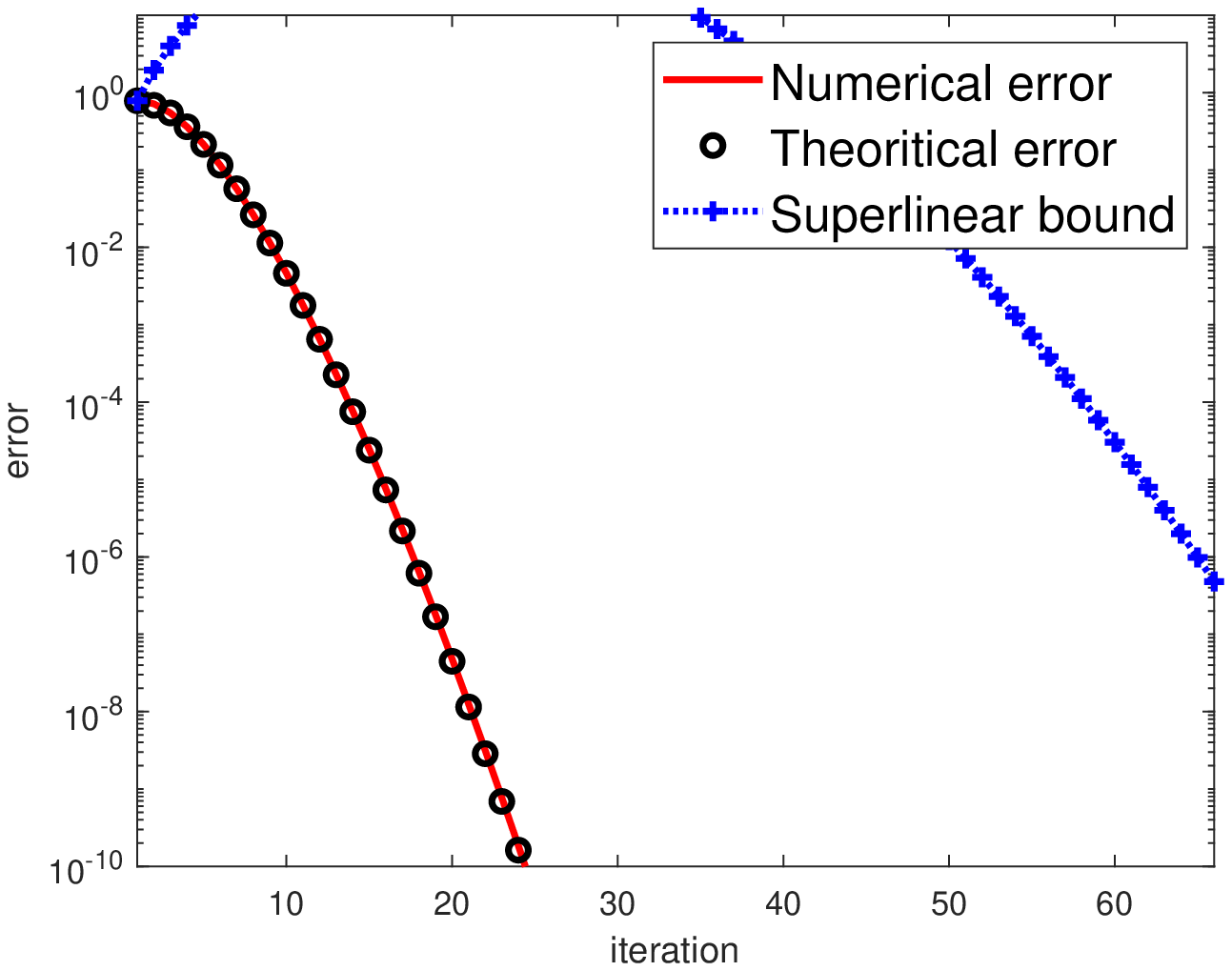}
    \includegraphics[width=0.30\textwidth]{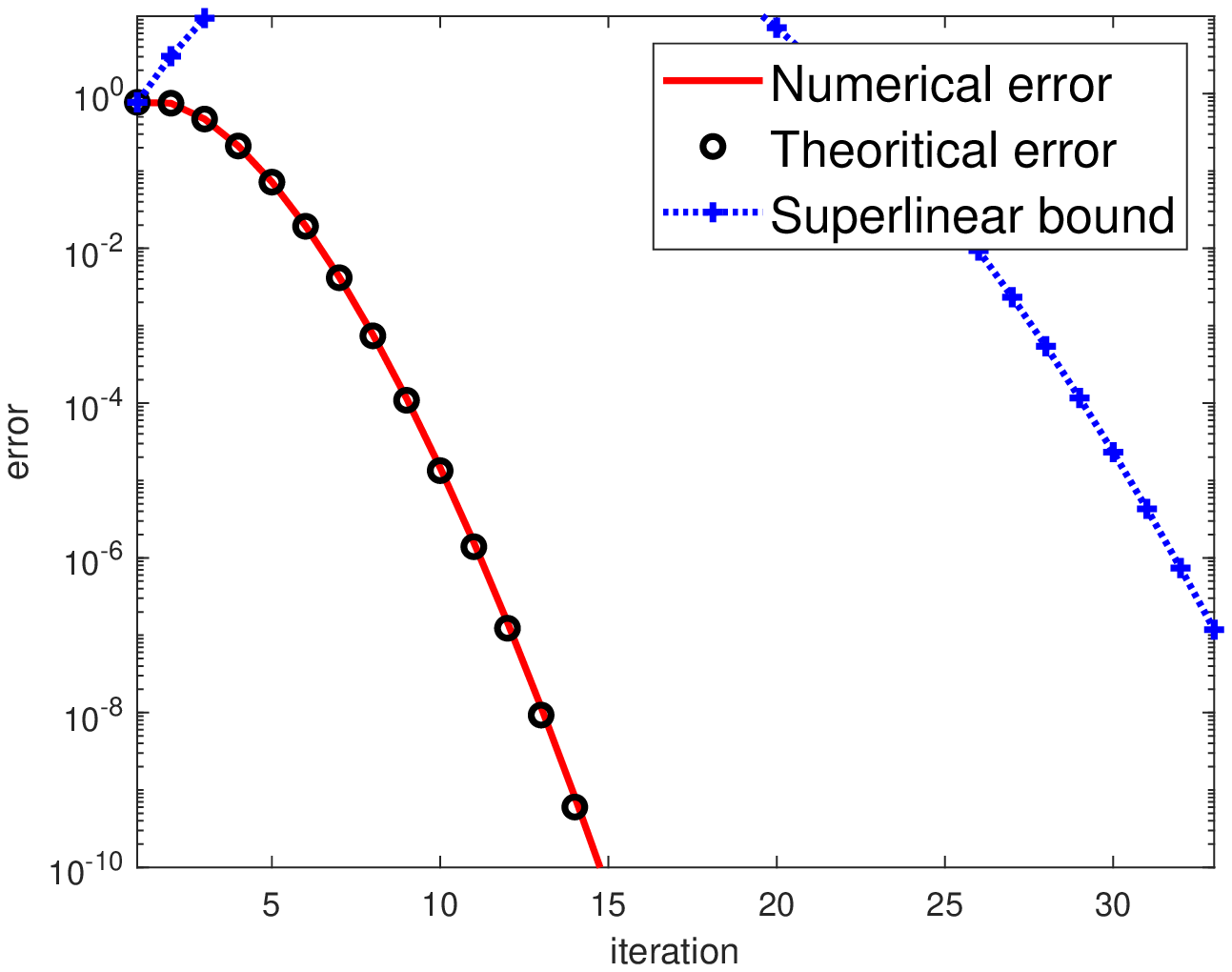}
	\caption{Comparison for $a>b$ among numerically measured convergence rate theoretical error at $T=1$, on the left for $2\nu = 0.2$, middle for $2\nu = 0.5$ and on the right for $2\nu = 0.8$}
	\label{NumFig06}
\end{figure}
\begin{figure}
	\centering
	\includegraphics[width=0.30\textwidth]{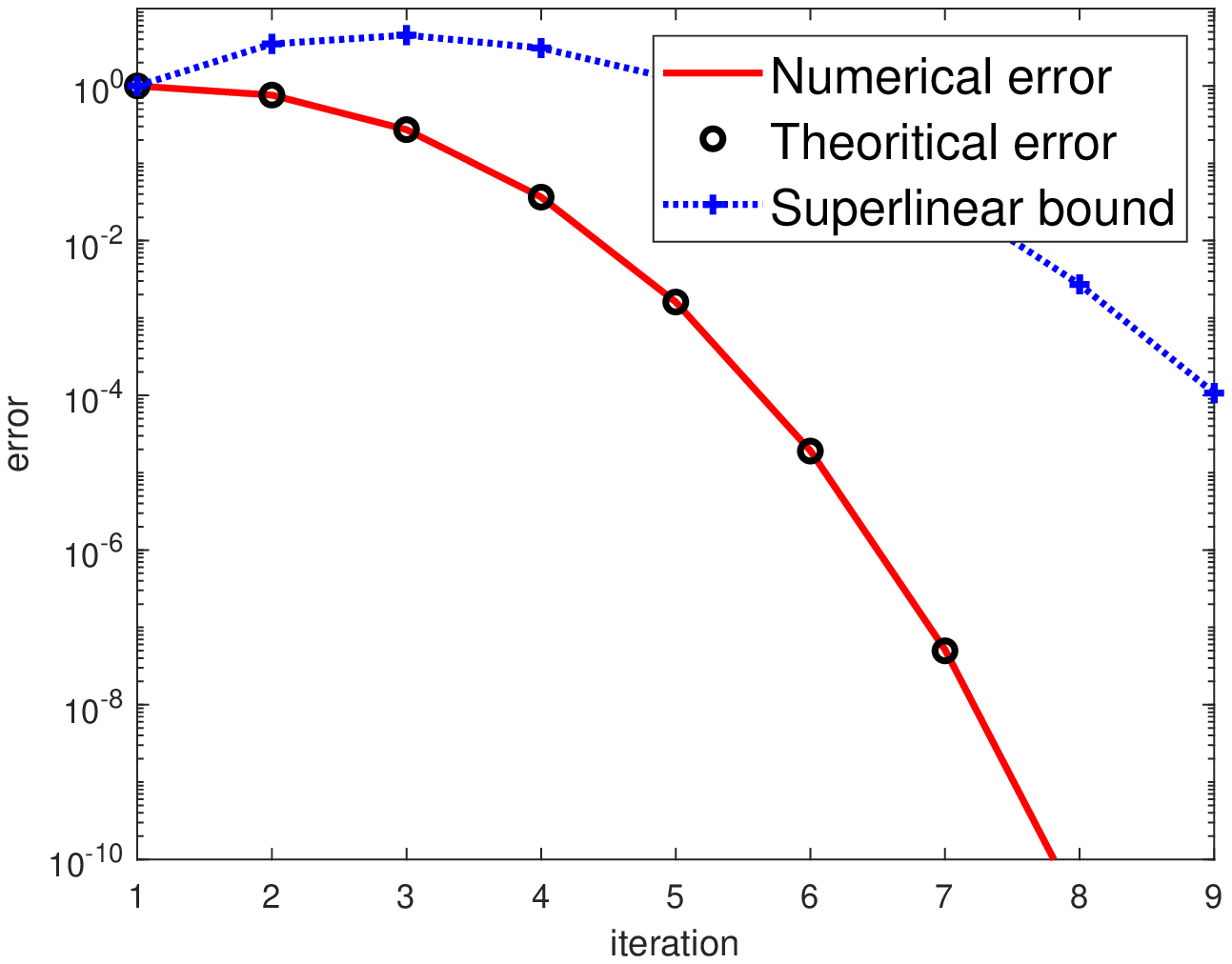}
	\includegraphics[width=0.30\textwidth]{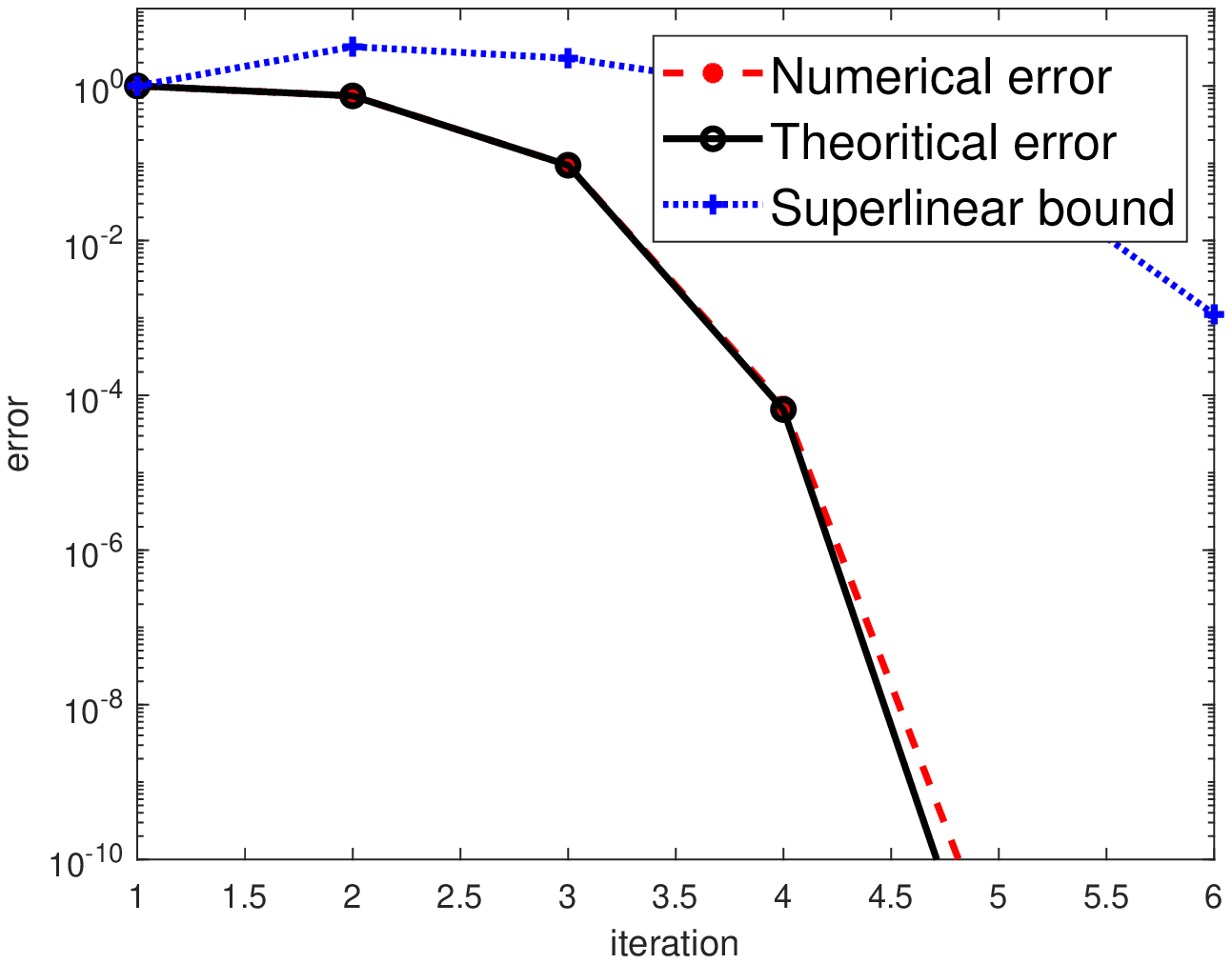}
	\includegraphics[width=0.30\textwidth]{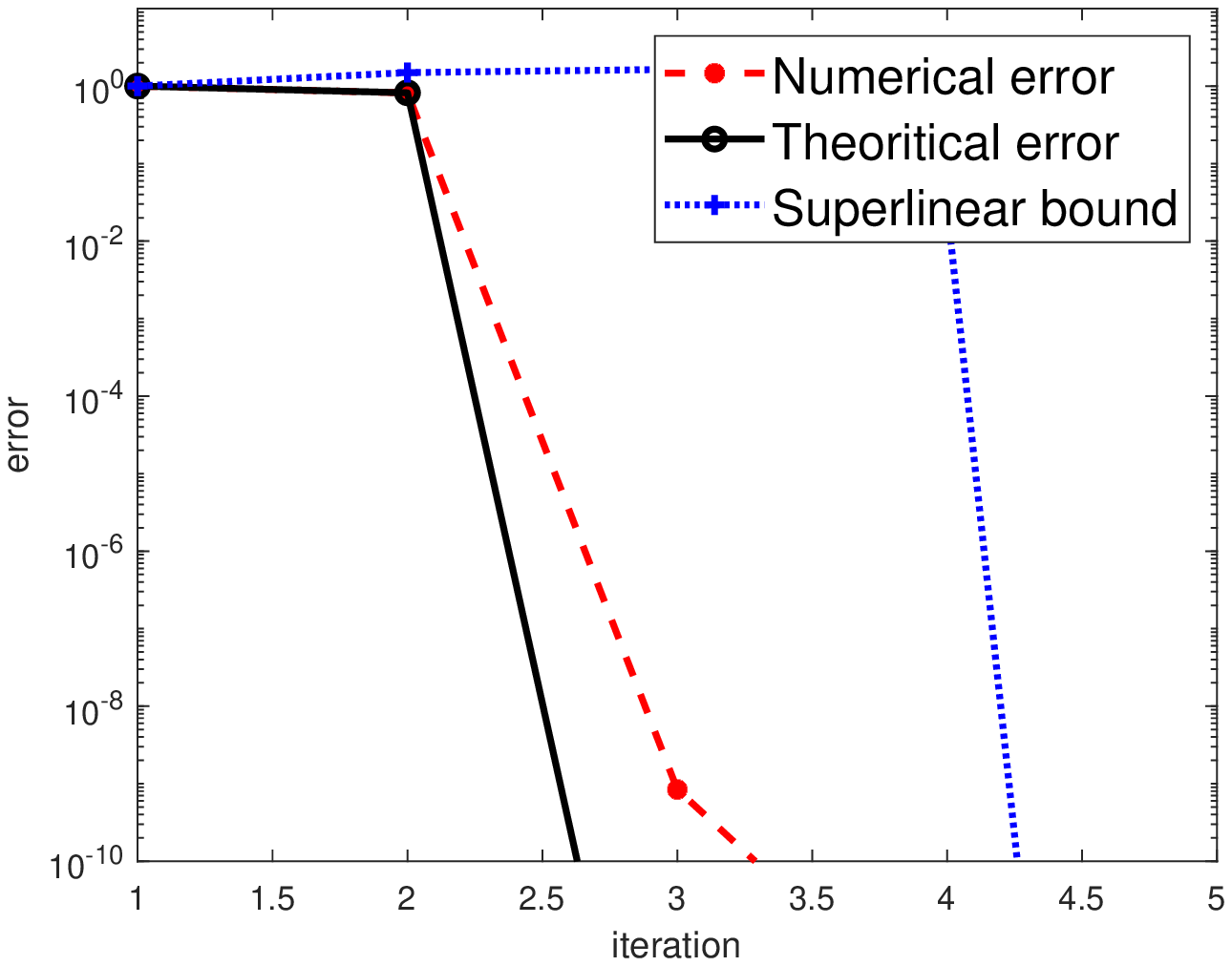}
	\caption{Comparison for $a>b$ among numerically measured convergence rate theoretical error at $T=1$, on the left for $2\nu = 1.2$, middle for $2\nu = 1.5$ and on the right for $2\nu = 1.8$}
	\label{NumFig07}
\end{figure}
\begin{figure}
	\centering
	\includegraphics[width=0.30\textwidth]{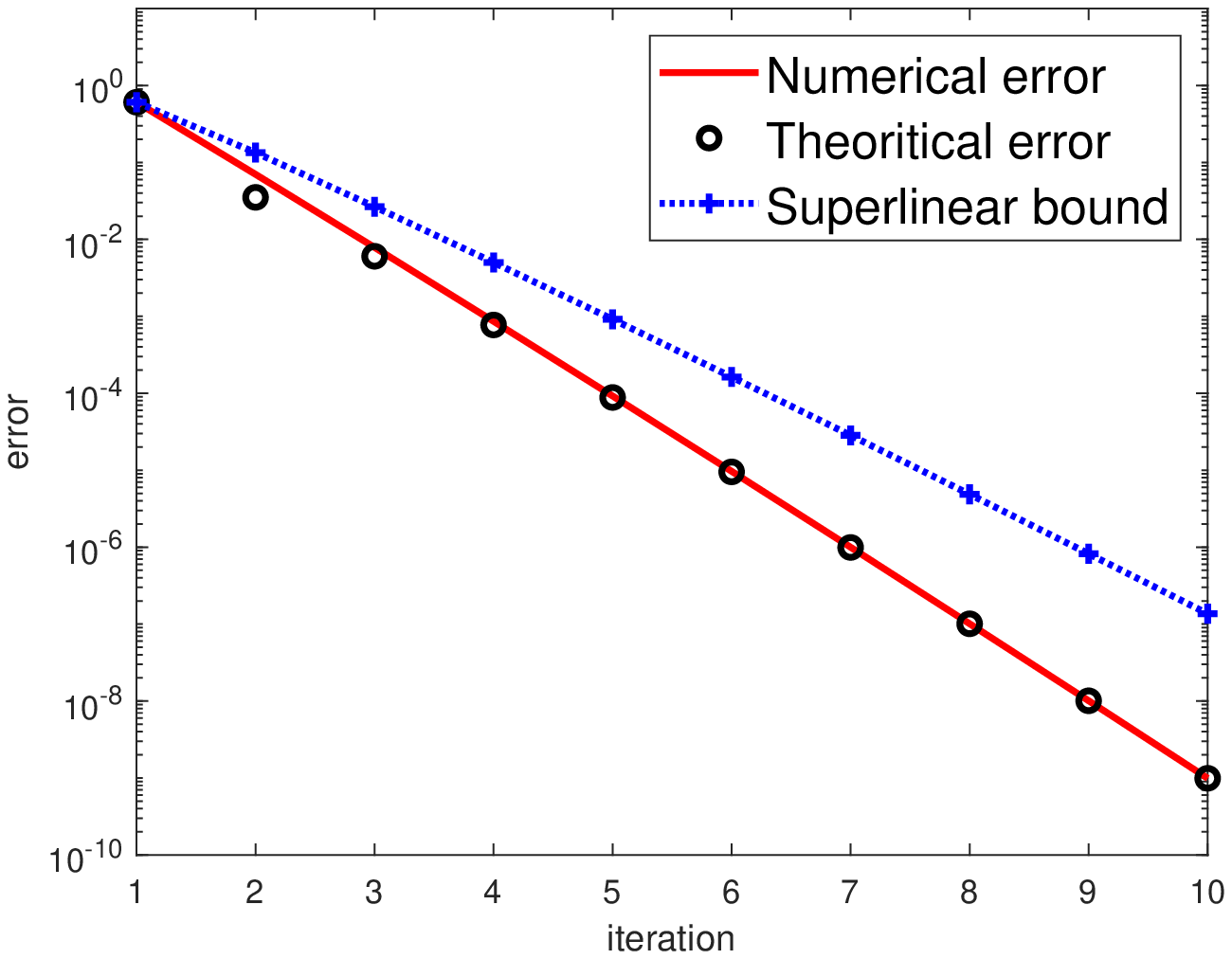}
	\includegraphics[width=0.30\textwidth]{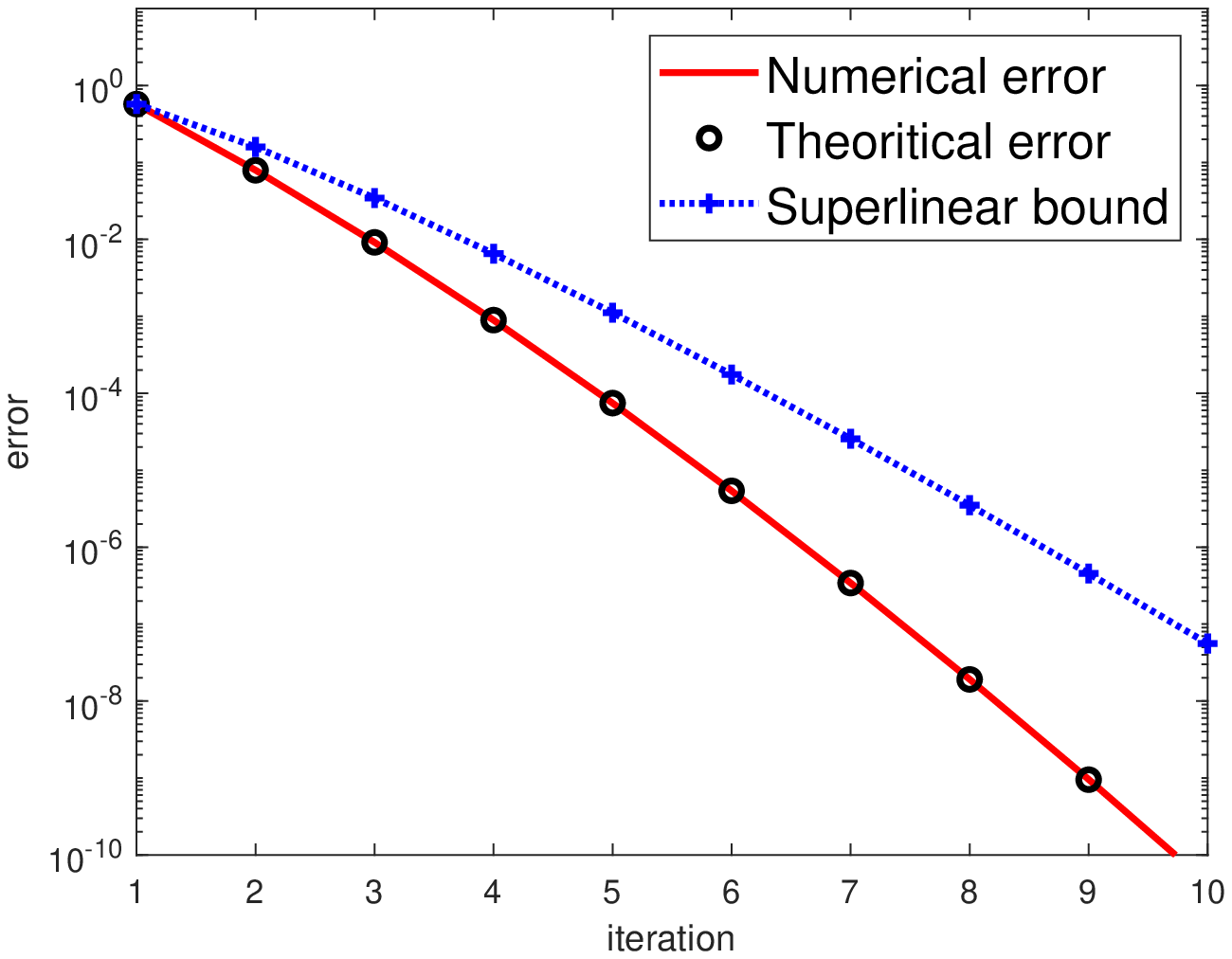}
	\includegraphics[width=0.30\textwidth]{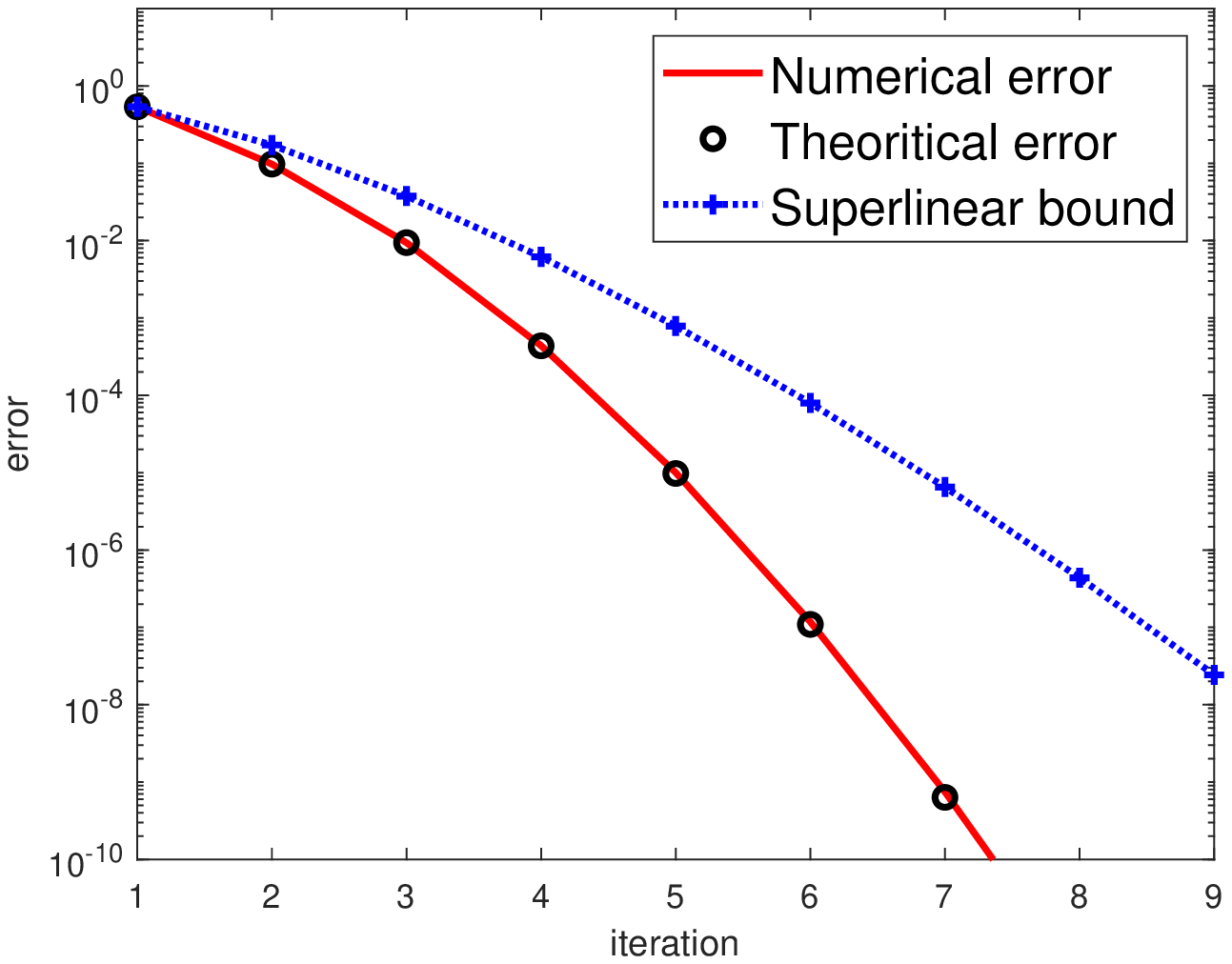}
	\caption{Comparison for $a<b$ among numerically measured convergence rate theoretical error at $T=1$, on the left for $2\nu = .2$, middle for $2\nu = .5$ and on the right for $2\nu = .8$}
	\label{NumFig08}
\end{figure}
\begin{figure}
	\centering
	\includegraphics[width=0.30\textwidth]{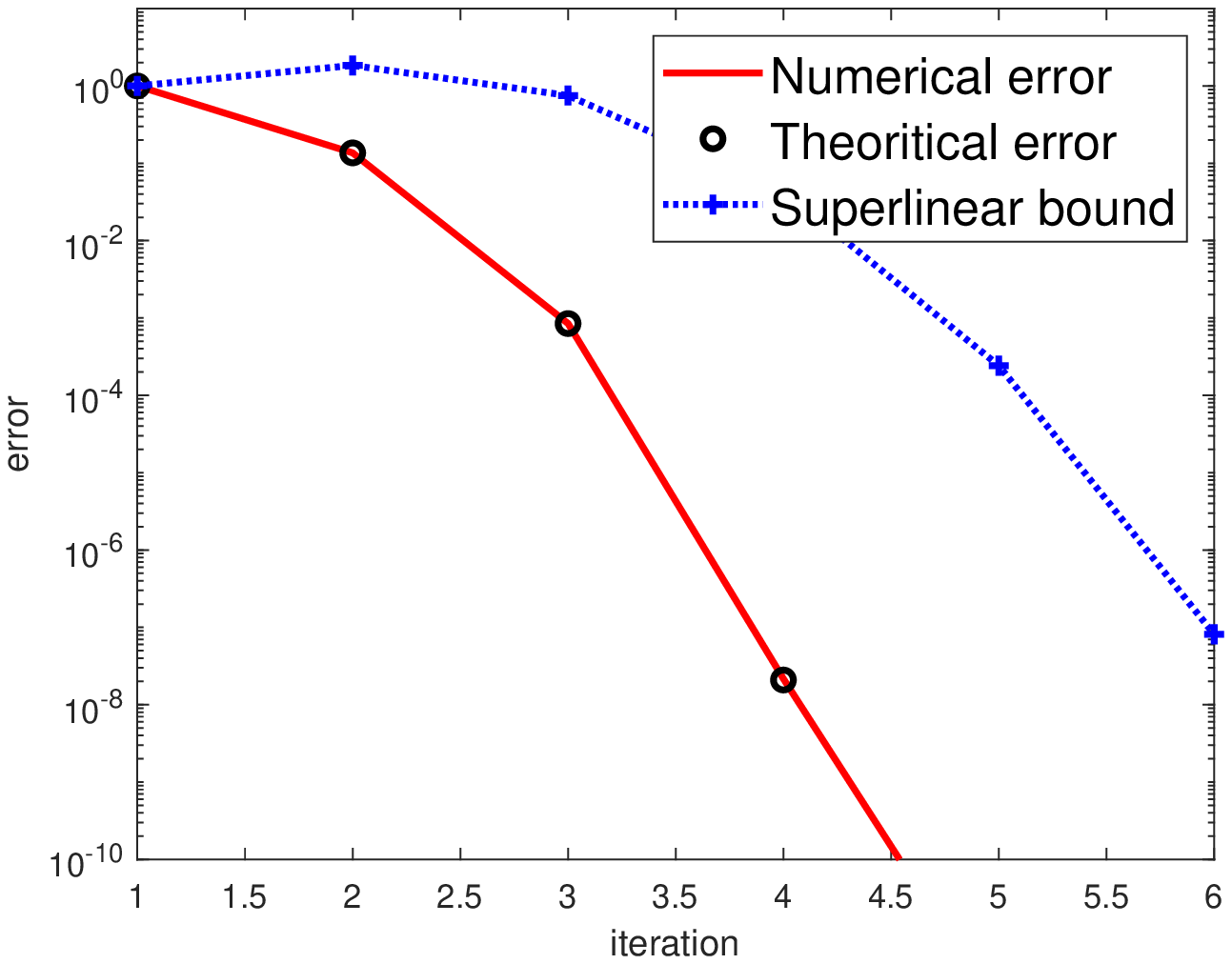}
	\includegraphics[width=0.30\textwidth]{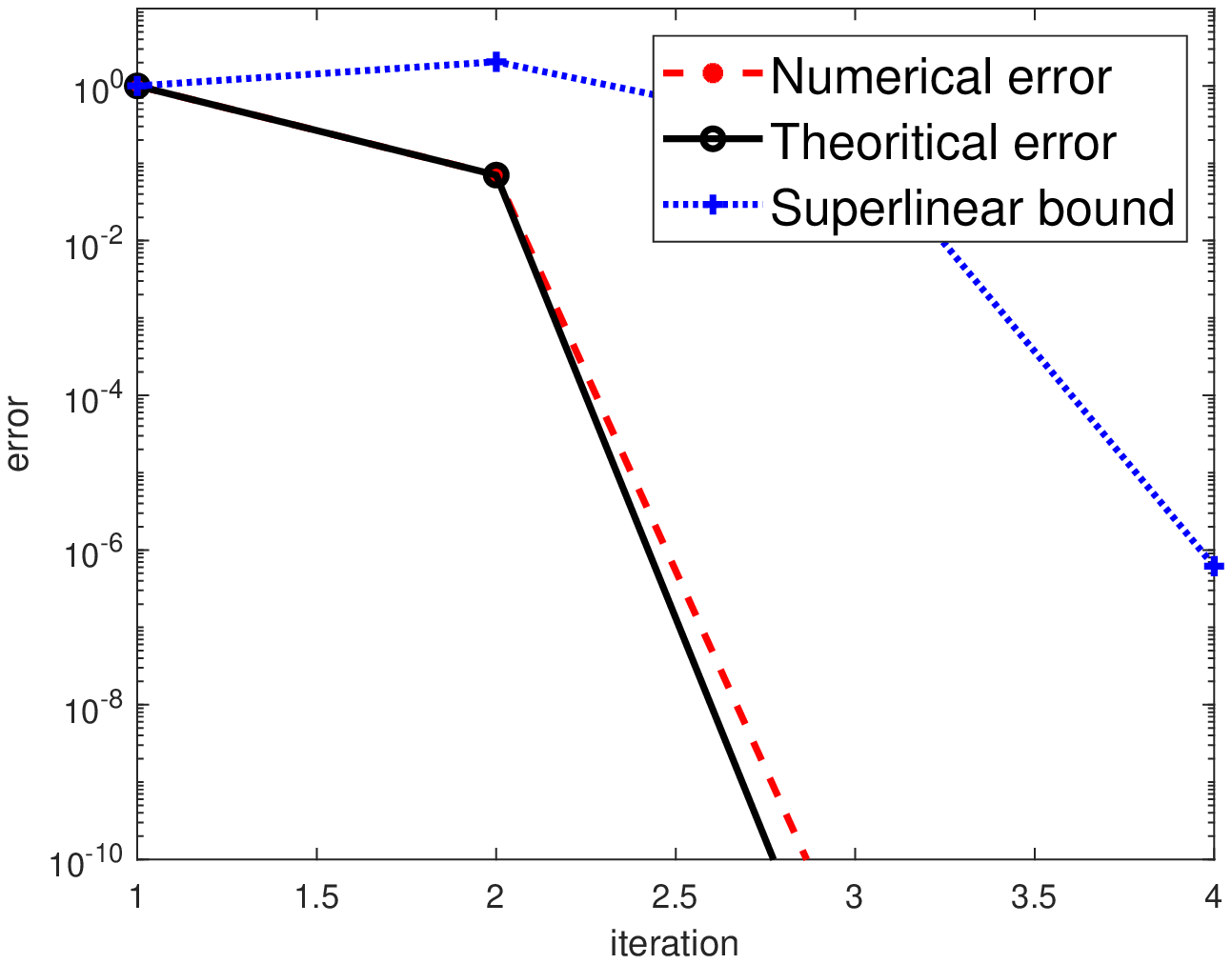}
	\includegraphics[width=0.30\textwidth]{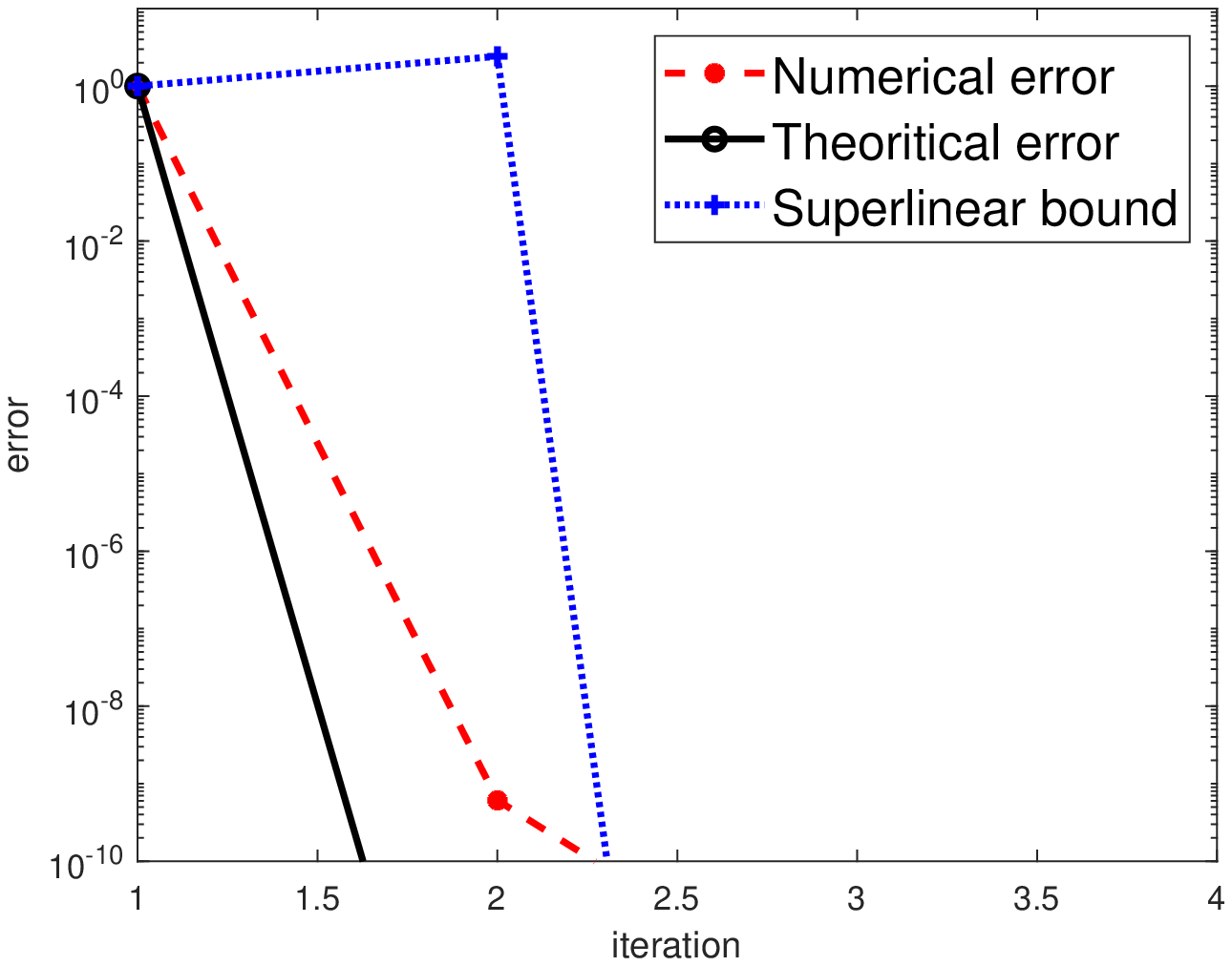}
	\caption{Comparison for $a<b$ among numerically measured convergence rate theoretical error at $T=1$, on the left for $2\nu = 1.2$, middle for $2\nu = 1.5$ and on the right for $2\nu = 1.8$}
	\label{NumFig09}
\end{figure}
\subsection{NNWR Algorithm in 1D}
We perform the NNWR experiment on the model problem ~\eqref{NumericalModelProblem} by choosing $F(x) = \sin(\pi x/16)$ and $g(x) = x(16-x)/64$. For the first set of experiments, the domain $\Omega = (0,16)$ is divided into five subdomains with spatial grid size according to subdomain size (mentioned below) and temporal step size $\Delta t = 0.015$ for super-diffusion on a time window $T = 4$.

In Figure~\ref{NumFig2}, we compare the error for different values of the relaxation parameter $\theta$ for equal (left) and unequal (right) subdomains with fractional order $2\nu =0.5$ with the grid size $\Delta x = 0.01$. The unequal subdomains are respectively $\Omega_1 = (0,3.5), \Omega_2 = (3.5,5.5), \Omega_3 = (5.5,10),\Omega_4 = (10,12), \Omega_5 = (12,16)$. We run the same experiments for fractional order $2\nu = 1.5$ in Figure~\ref{NumFig4}. From these experiments, we observe that $\theta = 0.25$ give the super-linear optimal convergence, and the other values of $\theta$ give linear convergence.

In Figure~\ref{NumFig5}, we compare the numerical errors for different values of the fractional order $2\nu$ for the same equal and unequal subdomains, as mentioned earlier. We observe that the large value of the fractional order gives faster convergence, which is expected as per theoretical results.

So far, we have chosen $\kappa = 1$ for NNWR experiments. Now we run the experiments by considering diffusion coefficient $\kappa$ as $\kappa_1 = 0.25, \kappa_2 = 1,\kappa_3 = 0.25,\kappa_4 = 4,\kappa_5 = 1$, and show in Figure~\ref{NumFig6}. We consider $\theta_i = 1/(2+\sqrt{\kappa_i/\kappa_{i+1}}+\sqrt{\kappa_{i+1}/\kappa_i})$ for superlinear convergence.

Finally, we compare the numerical behavior of the NNWR algorithm with the theoretical estimates obtained and plot in Figure~\ref{NumFig8} \& \ref{NumFig9}. We consider four, eight, and twelve subdomain cases. Here, we divide domain $\Omega = (0,16)$ into equal subdomain size for each case and take the diffusion coefficient $\kappa_i = 1/4^{(i-1)}, i = 1,2,...,N/2$ as shown in table~\ref{Table2}. 

\begin {table}
\begin{center}
\caption {diffusion coefficient used for different subdomains in NNWR experiments in Fig.~\ref{NumFig8} and Fig.~\ref{NumFig9}. \label{Table2}}
\begin{tabular}{|c|c|c|c|c|c|c|}
\hline
No. of subdomains & $\kappa_{1}$ & $\kappa_{2}$ & $\kappa_{3}$ & $\kappa_{4}$ & $\kappa_{5}$ & $\kappa_{6}$ \tabularnewline
\hline
4 & 1,4 & 2,3 &  &  &  & \tabularnewline
\hline
8 & 1,8 & 2,7 & 3,6 & 4,5 &  &  \tabularnewline
\hline
12 & 1,12 & 2,11 & 3,10 & 4,9 & 5,8 & 6,7 \tabularnewline
\hline
\end{tabular}
\end{center}
\end {table}

\begin{figure}
  \centering
  \includegraphics[width=0.45\textwidth]{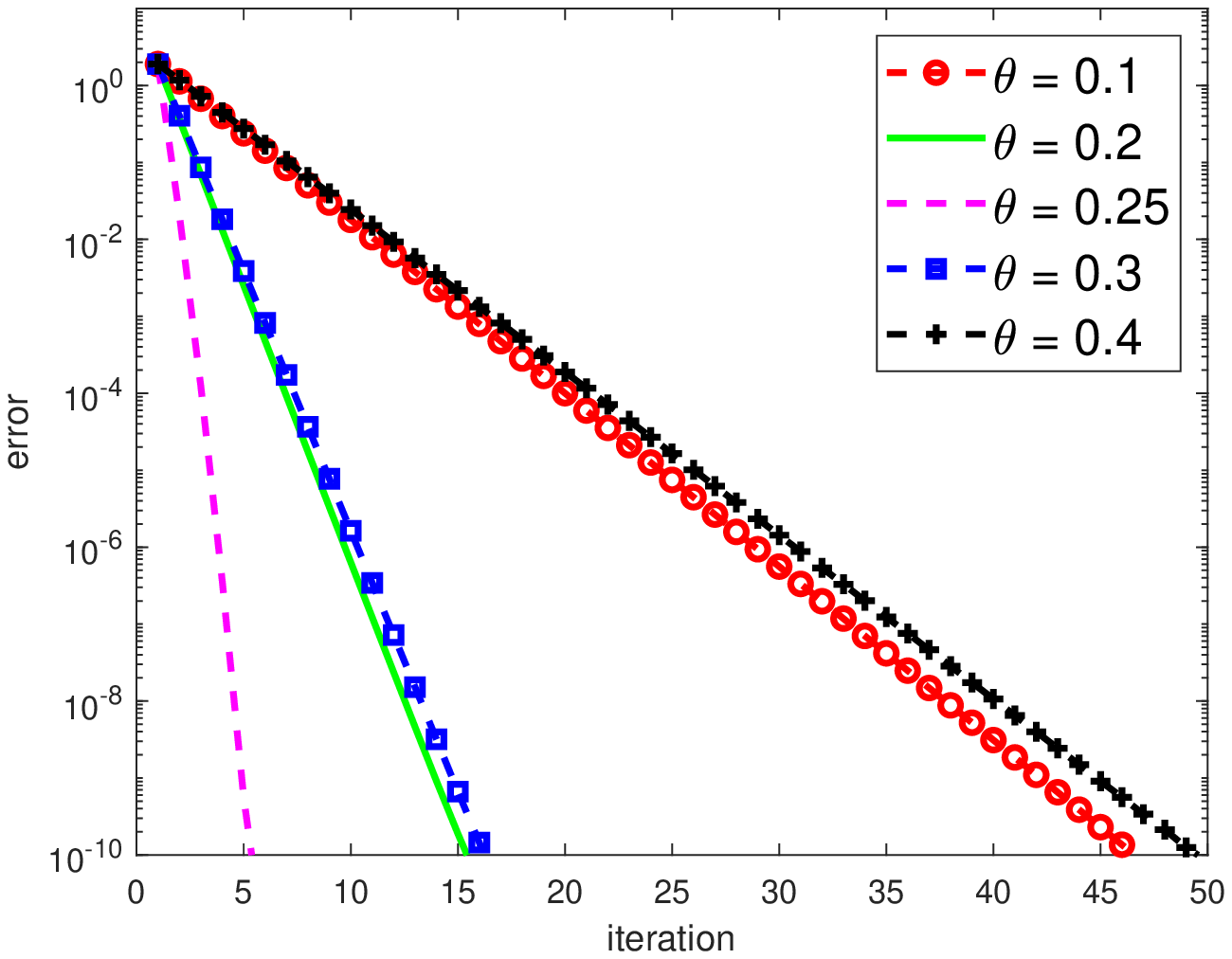}
  \includegraphics[width=0.45\textwidth]{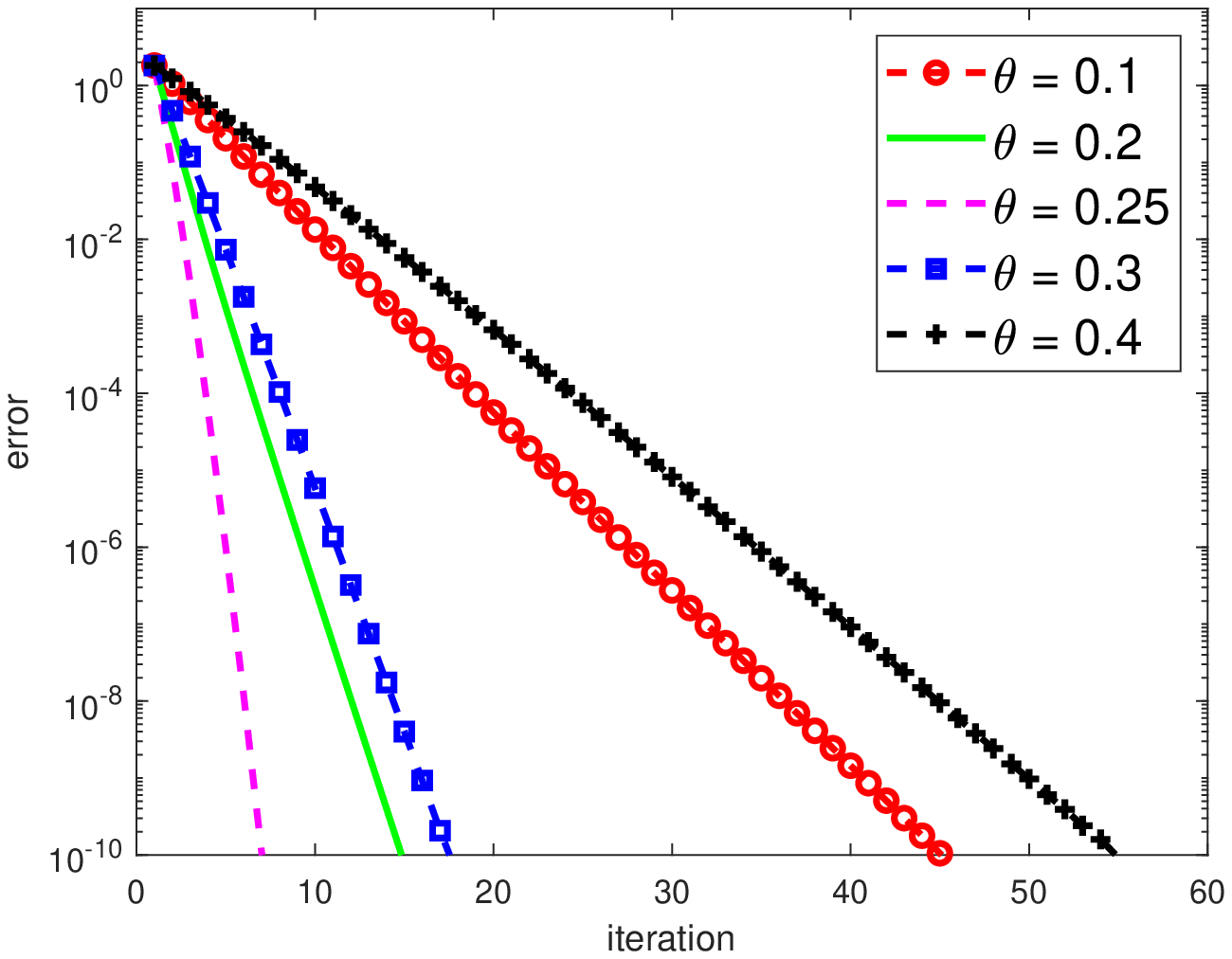}
  \caption{Convergence of NNWR with five subdomains for $2\nu=.5$ for $T=4$ and various
    relaxation parameters on the left equal subdomain, and on the right unequal subdomain}
  \label{NumFig2}
\end{figure}
\begin{figure}
	\centering
	\includegraphics[width=0.45\textwidth]{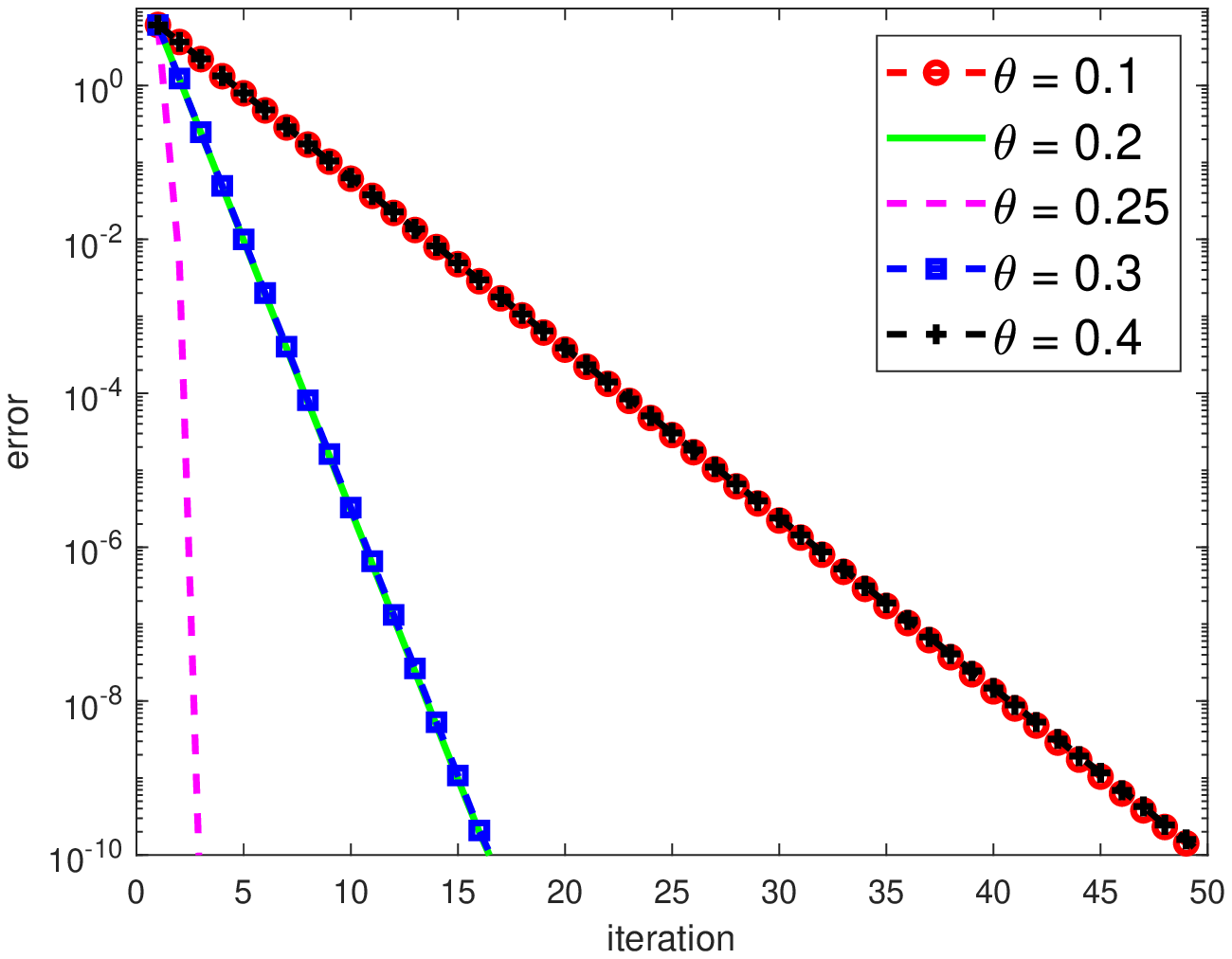}
	\includegraphics[width=0.45\textwidth]{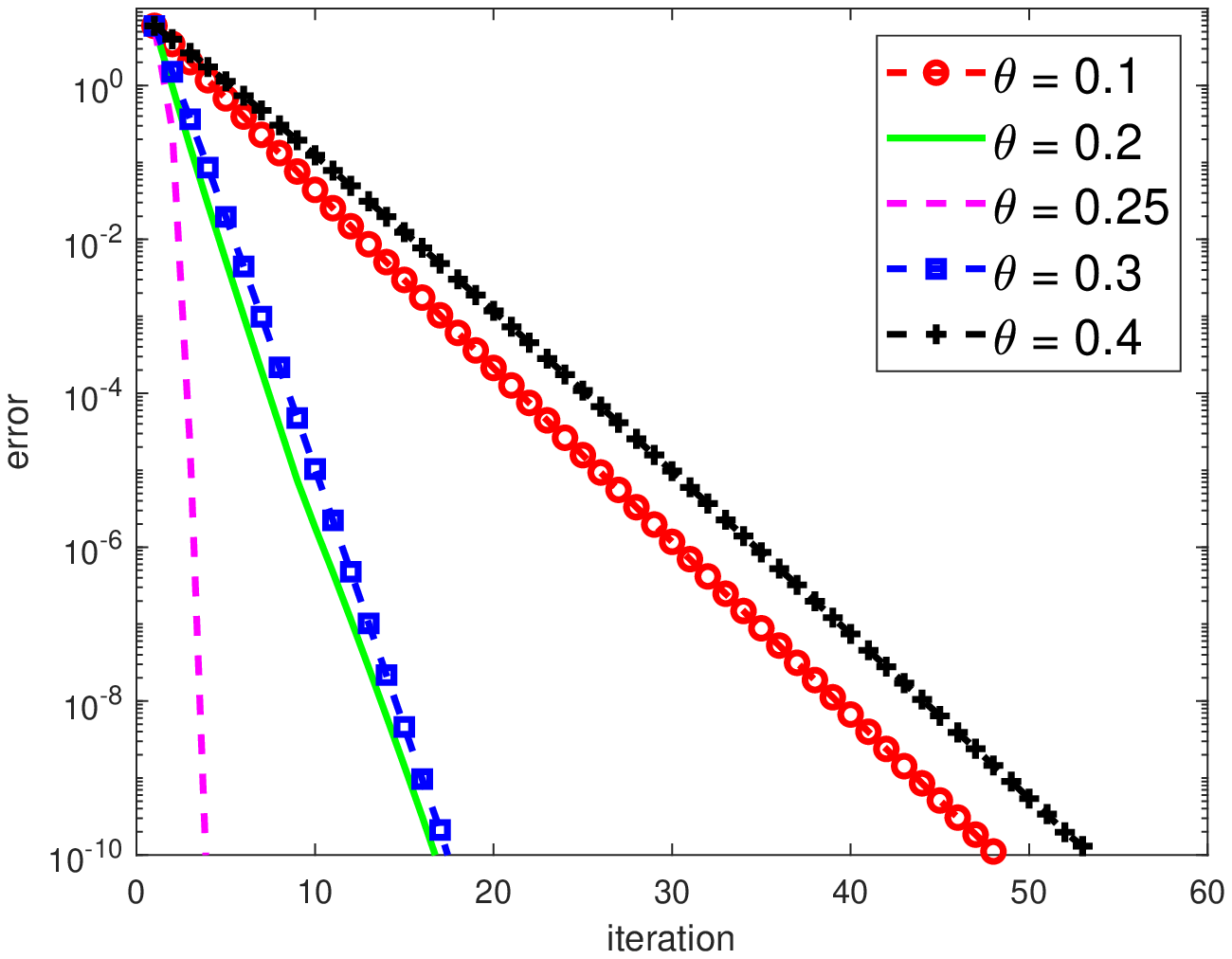}
	\caption{Convergence of NNWR with five subdomains for $2\nu=1.5$ for $T=4$ and various
		relaxation parameters on the left equal subdomain, and on the right unequal subdomain}
	\label{NumFig4}
\end{figure}
\begin{figure}
	\centering
	\includegraphics[width=0.45\textwidth]{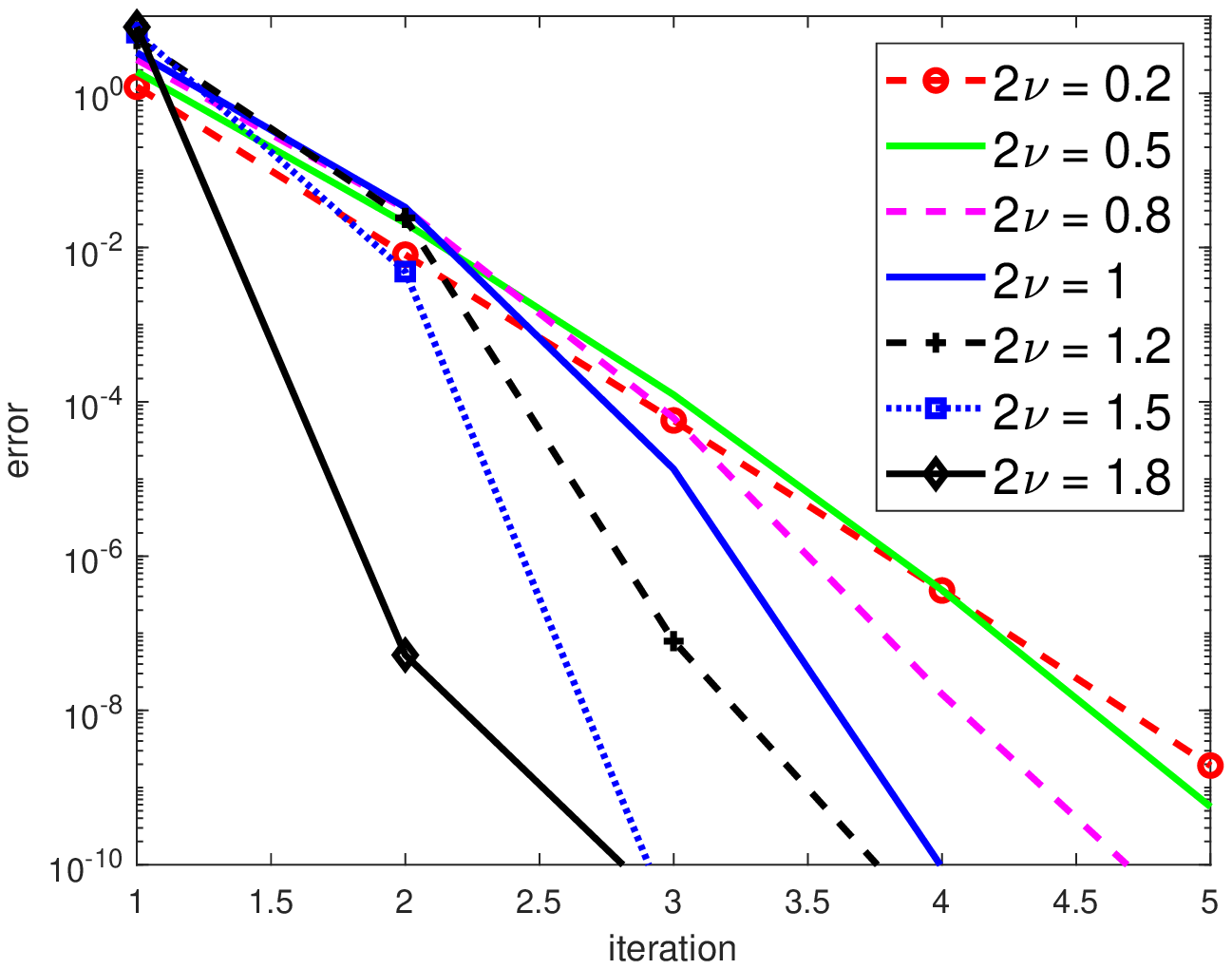}
	\includegraphics[width=0.45\textwidth]{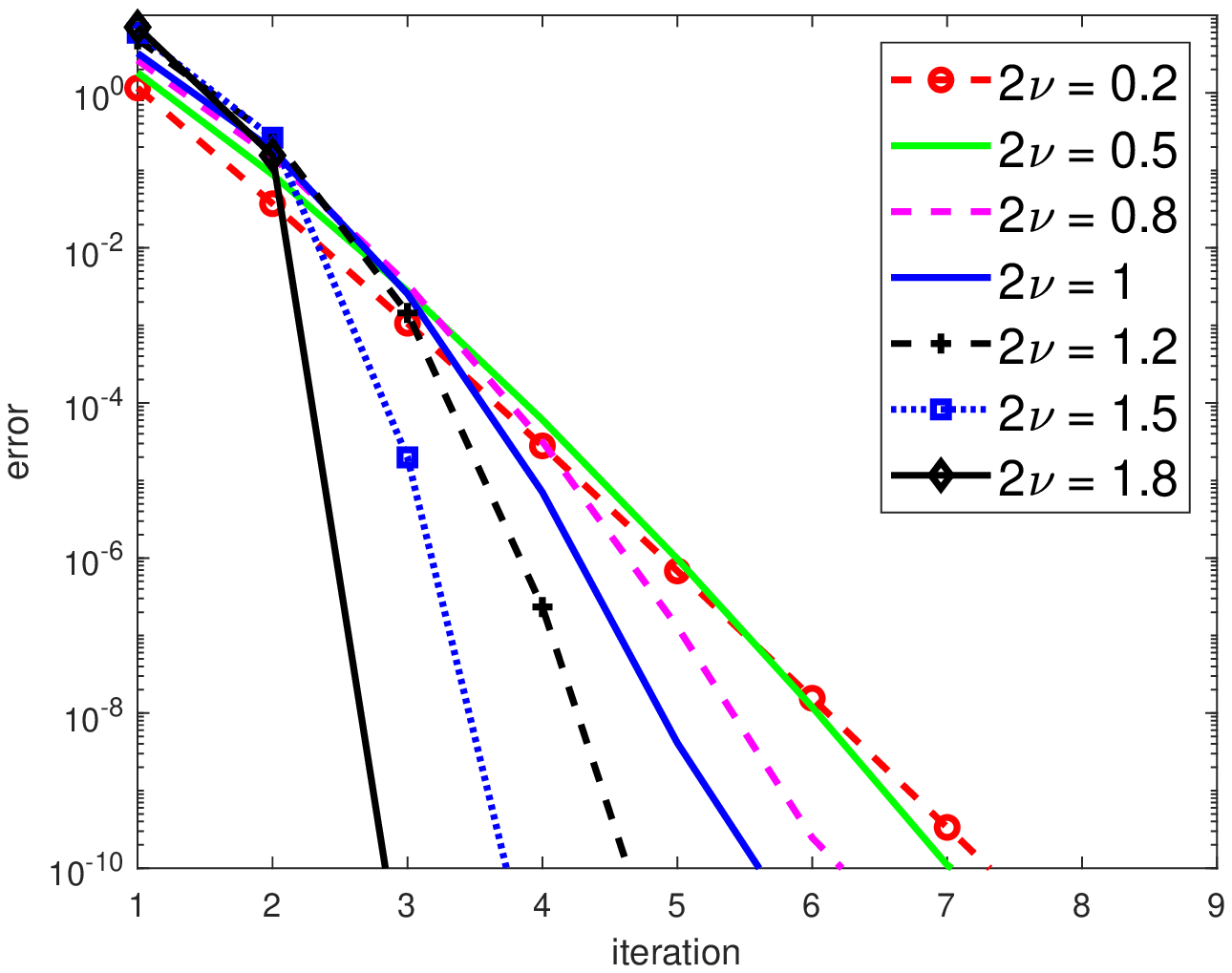}
	\caption{Convergence of NNWR with five subdomains for $T=4$ and various
		fractional order on the left equal kappa and equal subdomain size, and on the right equal kappa and unequal subdomain size}
	\label{NumFig5}
\end{figure}
\begin{figure}
	\centering
	\includegraphics[width=0.45\textwidth]{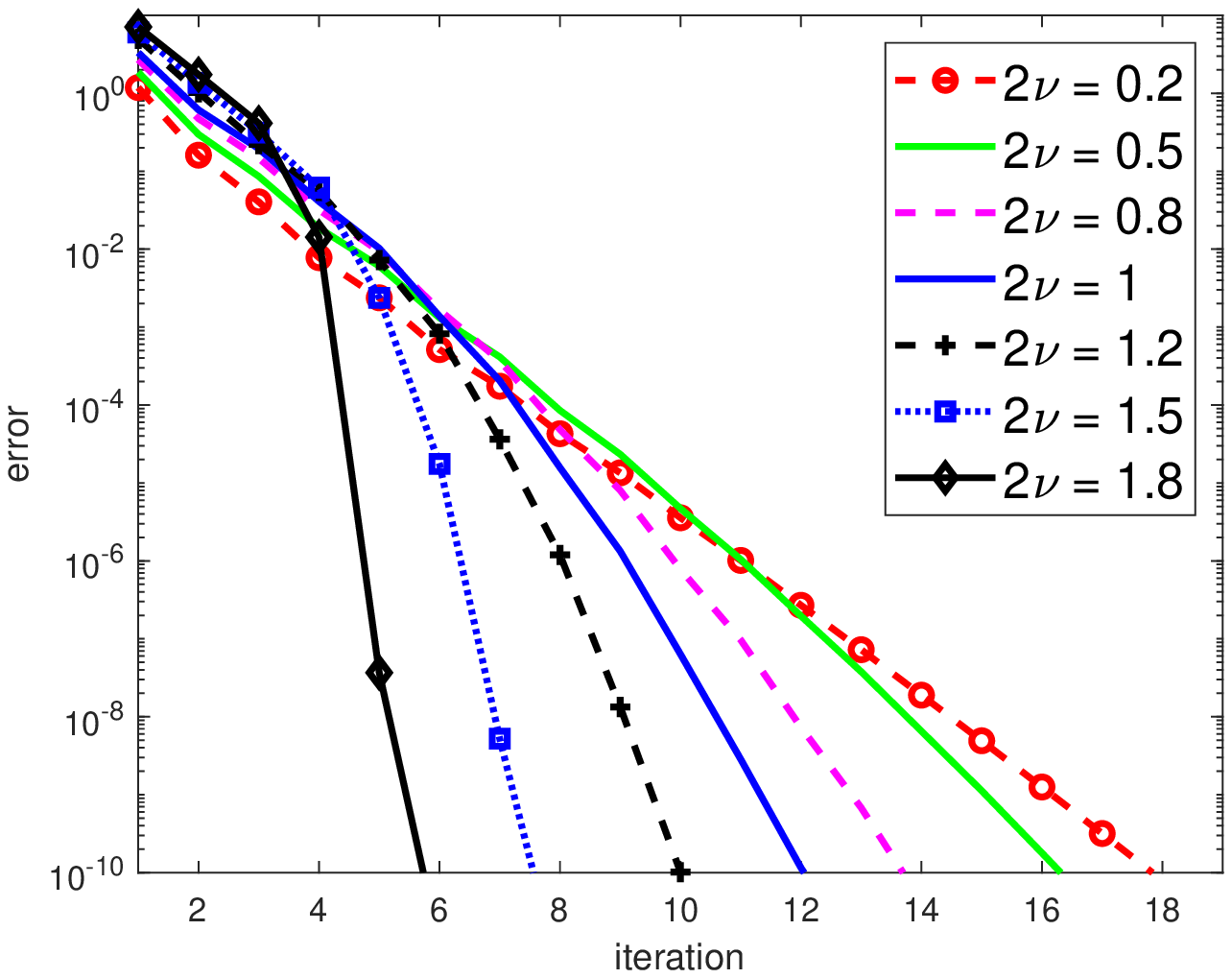}
	\includegraphics[width=0.45\textwidth]{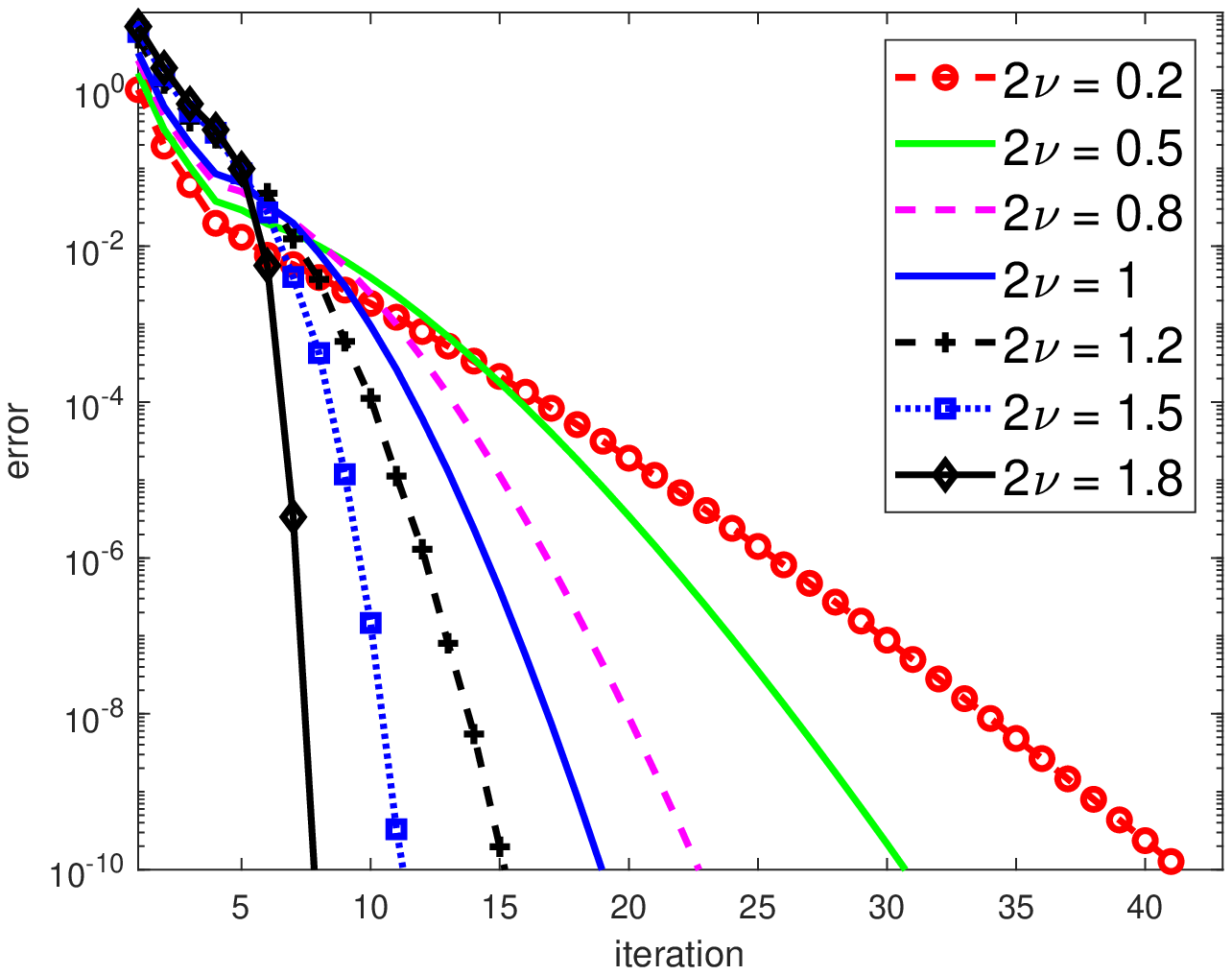}
	\caption{Convergence of NNWR with five subdomains for $T=4$ and various
		fractional order on the left unequal kappa equal subdomain size, and on the right unequal kappa and unequal subdomain size}
	\label{NumFig6}
\end{figure}
\begin{figure}
	\centering
	\includegraphics[width=0.30\textwidth]{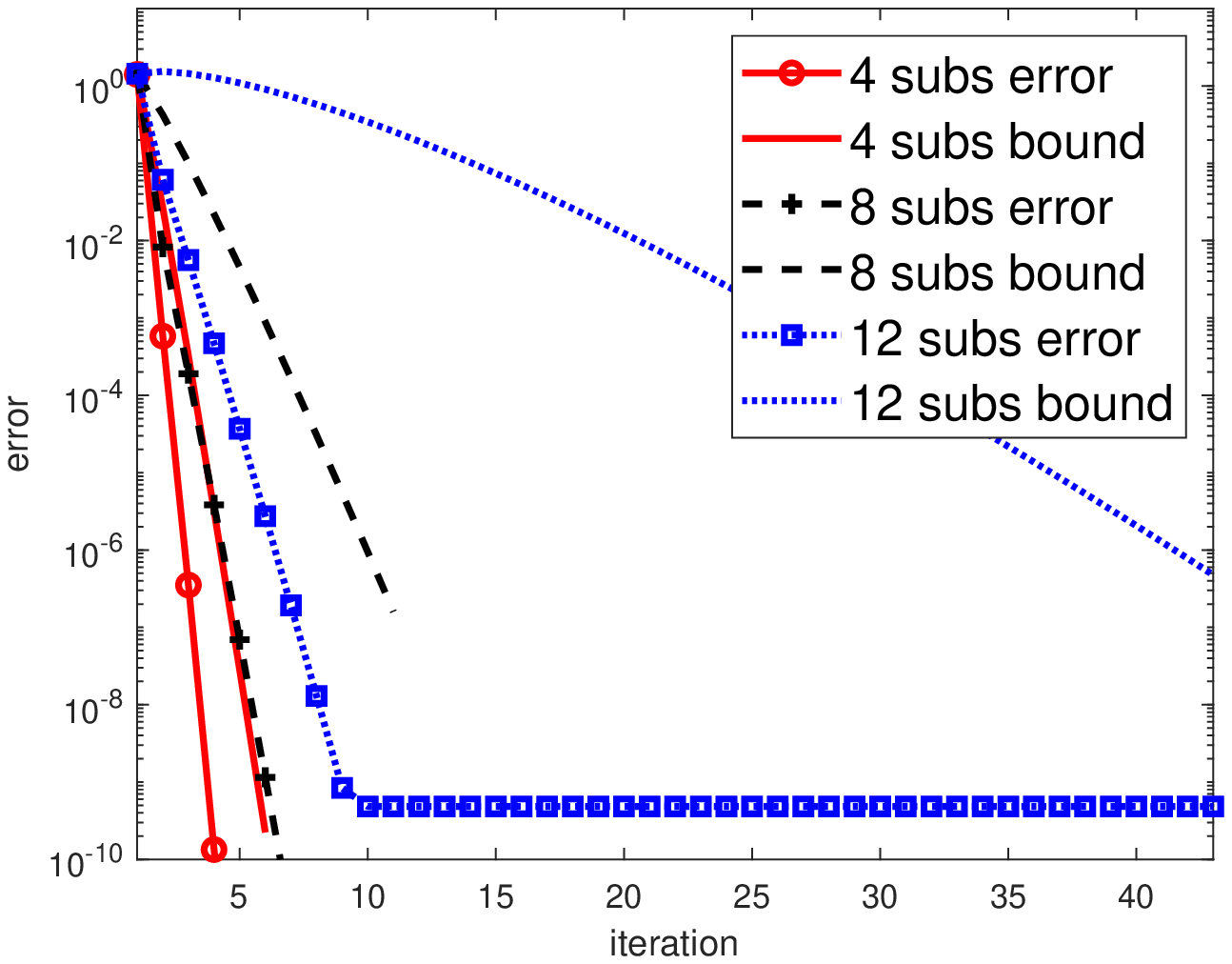}
	\includegraphics[width=0.30\textwidth]{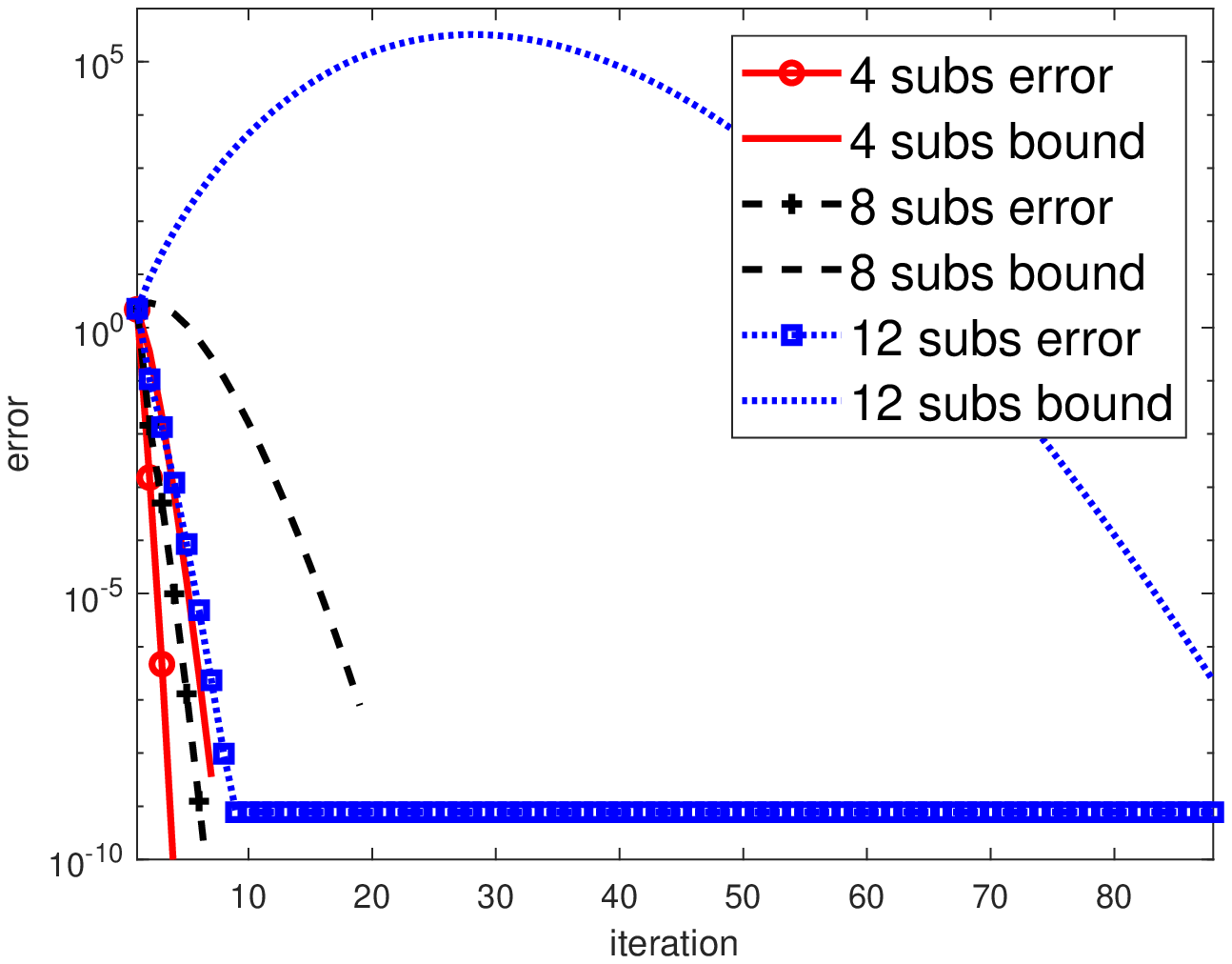}
	\includegraphics[width=0.30\textwidth]{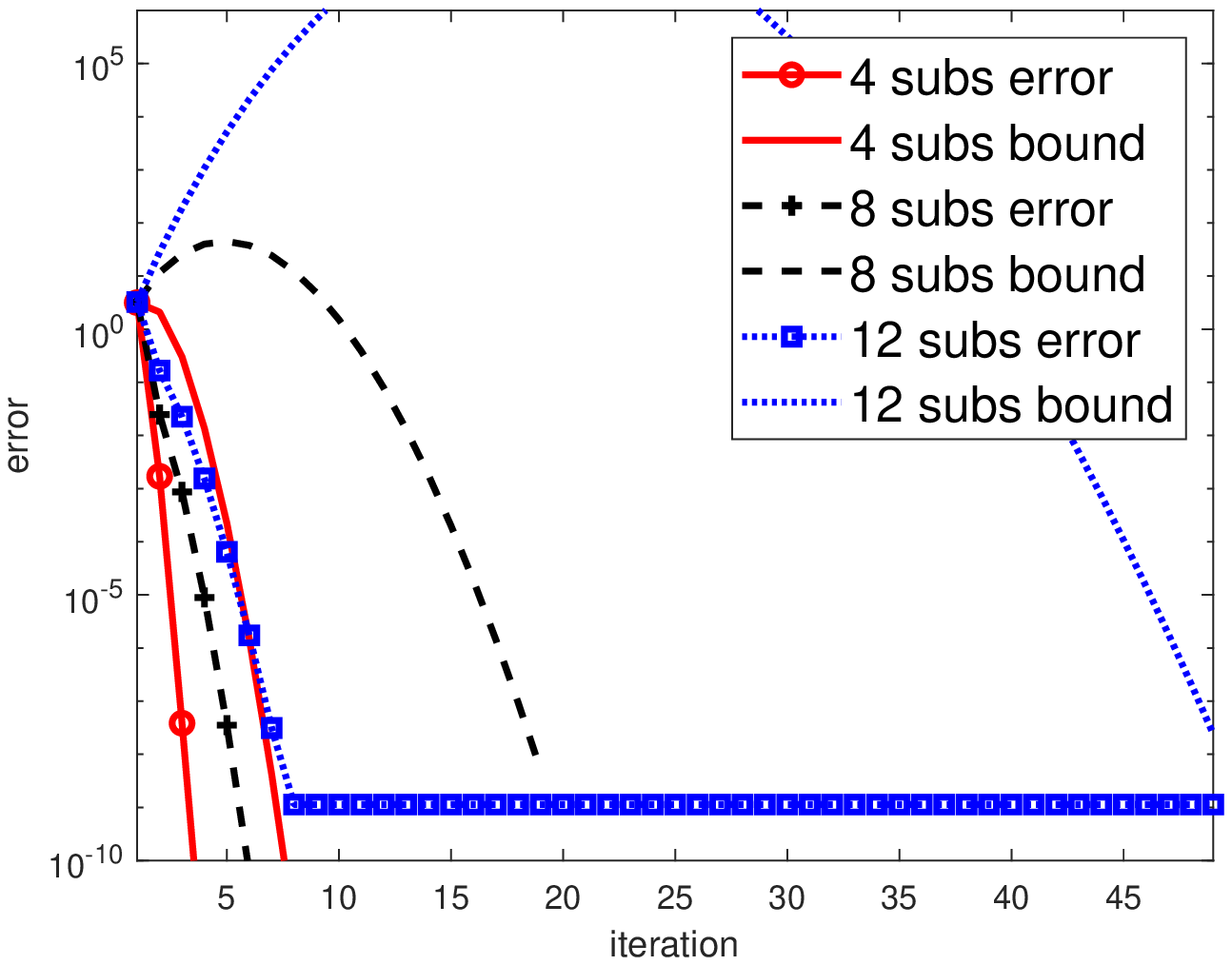}
	\caption{Comparison among numerically measured convergence rate theoritical error at $T=1$, on the left for $2\nu = 0.2$, middle for $2\nu = 0.5$ and on the right for $2\nu = 0.8$}
	\label{NumFig8}
\end{figure}
\begin{figure}
	\centering
	\includegraphics[width=0.30\textwidth]{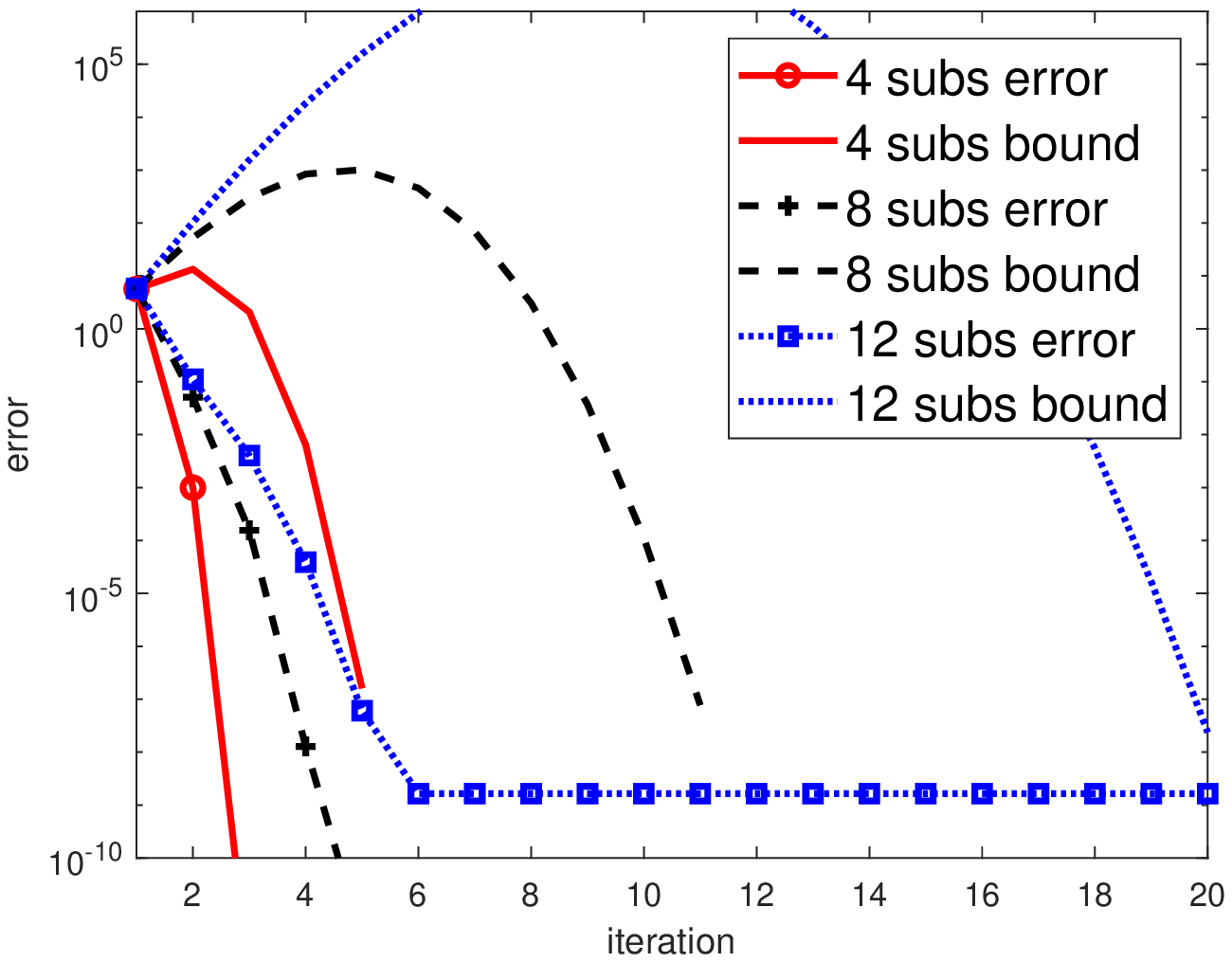}
	\includegraphics[width=0.30\textwidth]{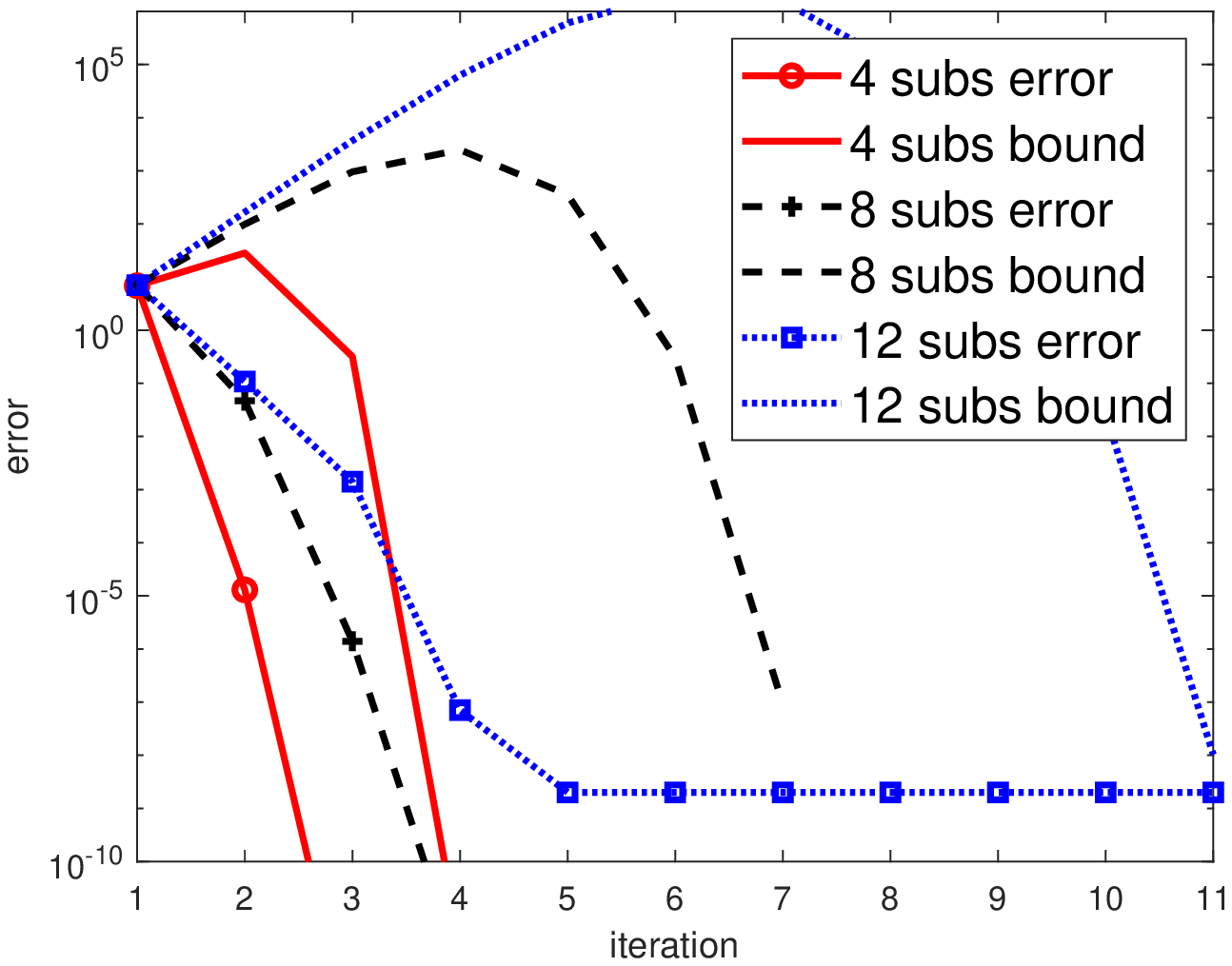}
	\includegraphics[width=0.30\textwidth]{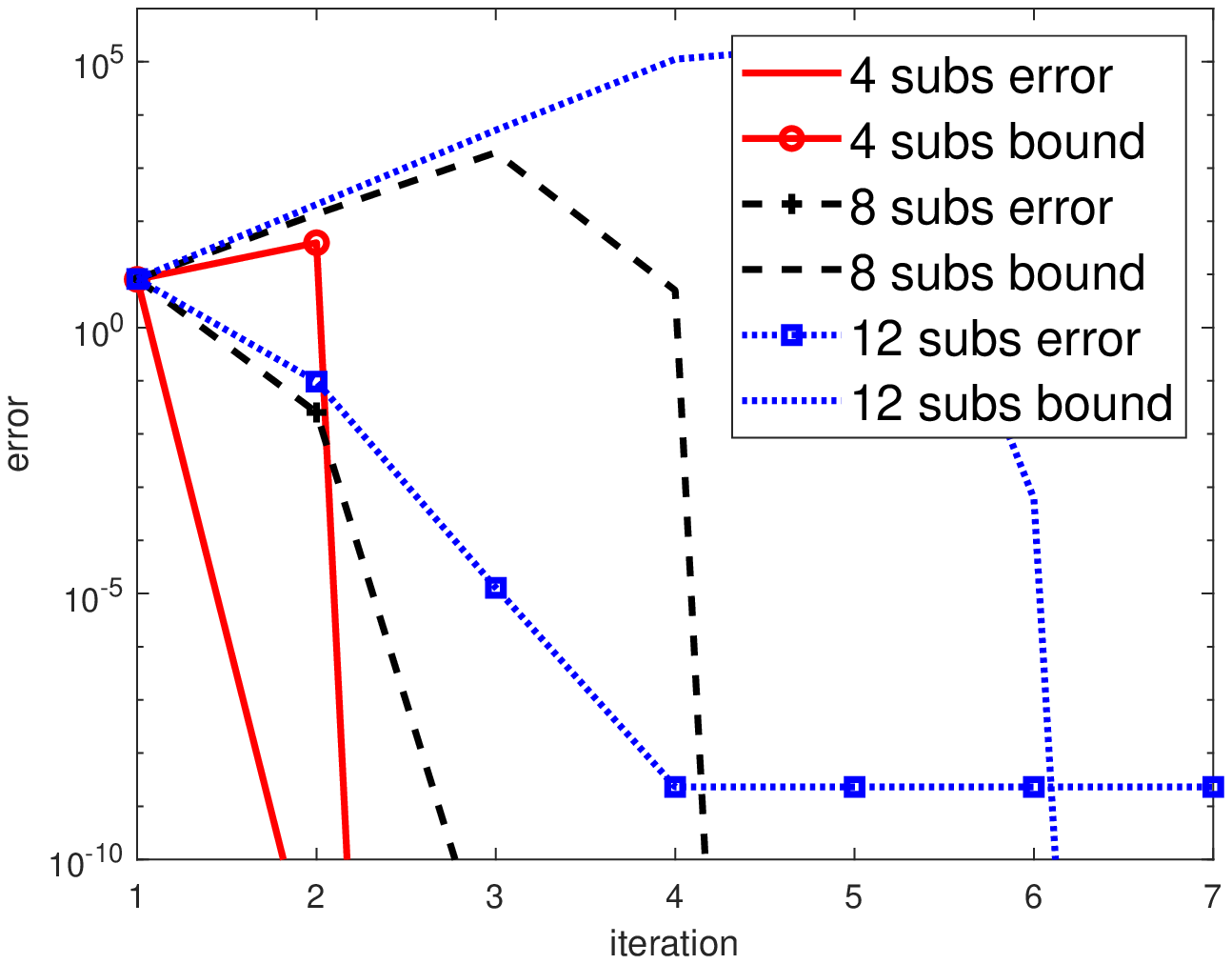}
	\caption{Comparison among numerically measured convergence rate theoritical error at $T=4$, on the left for $2\nu = 1.2$, middle for $2\nu = 1.5$ and on the right for $2\nu = 1.8$}
	\label{NumFig9}
\end{figure}

\subsection{NNWR Algorithm in 2D}
We have used the model problem ~\eqref{NumericalModelProblem} with $F(x,y) = 0$ and $g(x,y) = x(2-x)\exp(-10y^2)$ on the spatial domain $\Omega = (0,2)\times(-5,5)$ and for the time window $T = 1$. We have taken spatial grid sizes $\Delta x = 0.01$, $\Delta y = 0.1$, and the number of temporal grids is $2^8$, as we have chosen graded mesh, so temporal grid size varies for the sub-diffusion case. Our experiments will illustrate the NNWR method in 2D for two subdomains cases where $\Omega_1 = (0,0.5)\times(-5,5)$ and $\Omega_2 = (0.5,2)\times(-5,5)$. In Figure~\ref{NumFig10}, we compare the numerical error with the bounded estimate of the NNWR algorithm for different fractional order $2\nu$.  
\begin{figure}
	\centering
	\includegraphics[width=0.45\textwidth]{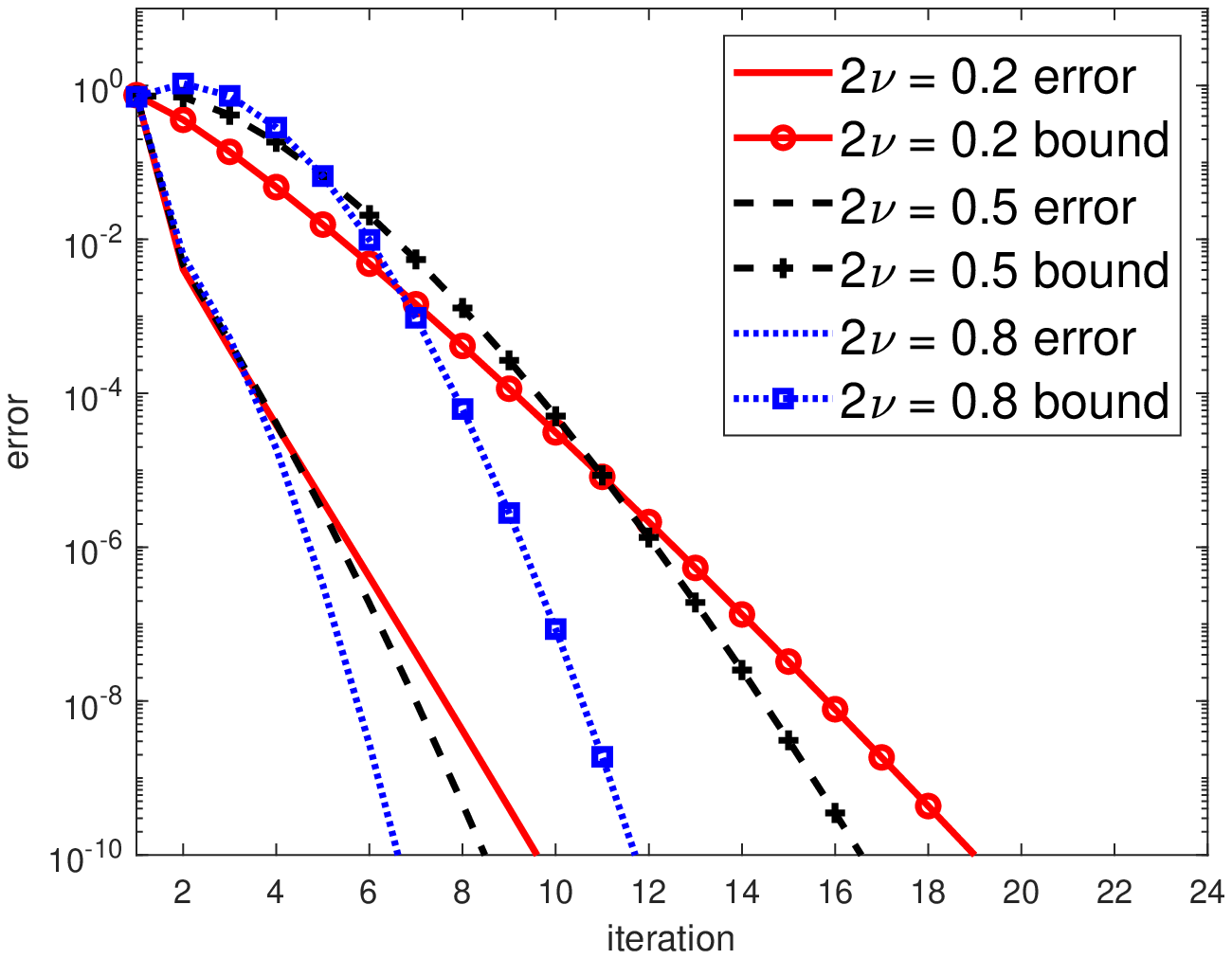}
	\includegraphics[width=0.45\textwidth]{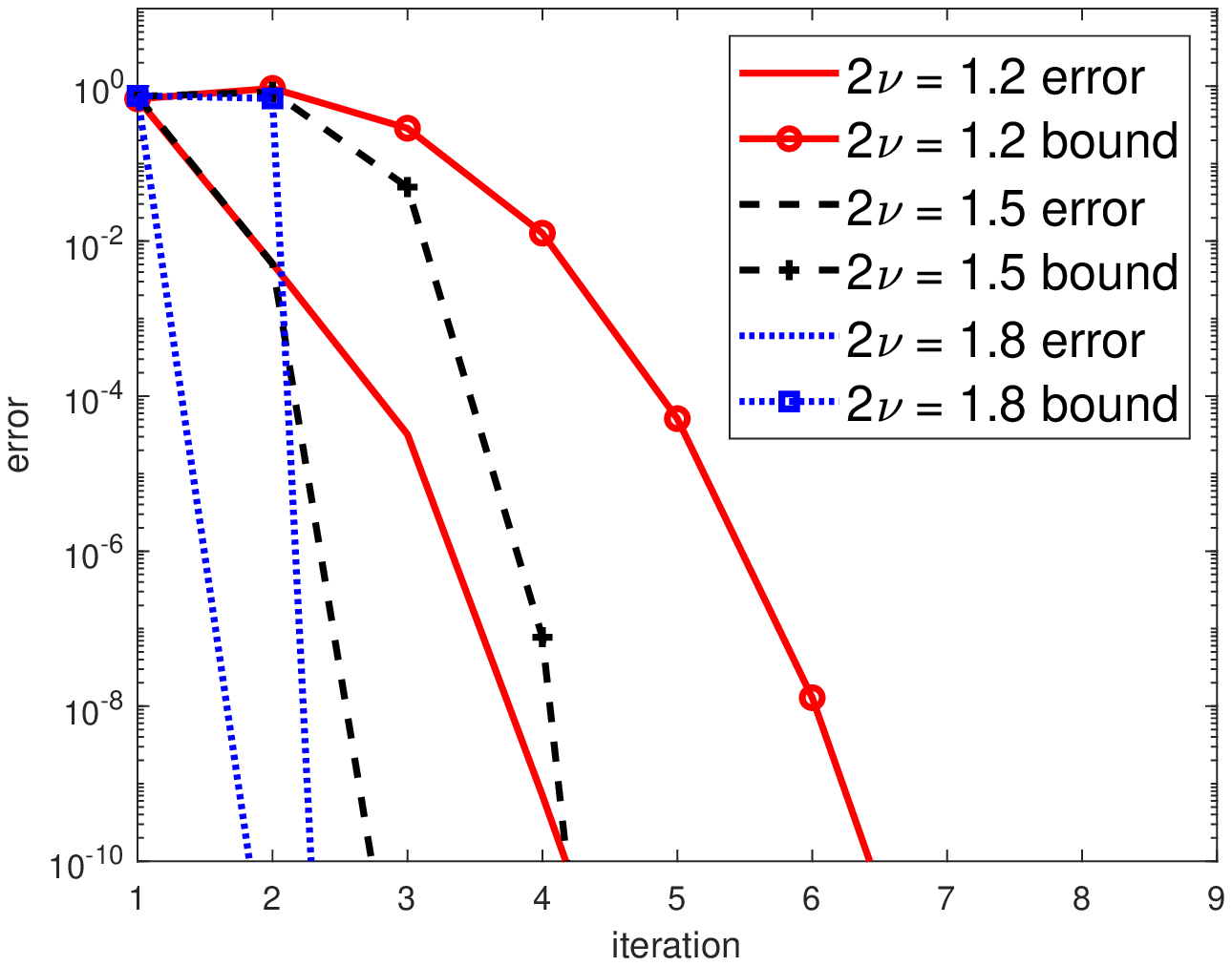}
	\caption{Comparing the numerical convergence and error bound of NNWR in 2D with two subdomains for $T=1$ and various
		fractional order on the left sub-diffusive region, and on the right diffusion wave region}
	\label{NumFig10}
\end{figure}

%% file: Conclusion.tex
\section{Conclusions}
We have extended the two classes of space-time algorithms, the Dirichlet-Neumann waveform relaxation (DNWR) and the Neumann-Neumann waveform relaxation (NNWR) algorithms, for time-fractional sub-diffusion and diffusion-wave problems. We have proved rigorously the convergence estimates for those cases in 1D, where we have taken different diffusion coefficients in different subdomains, which leads to the optimal choice of relaxation parameters for each artificial boundary. Using these optimal parameters, our estimate captures the increment of superlinear convergence rate as fractional order increases, and it goes to finite step convergence as fractional goes to two. We have numerically verified all the relaxation parameters and tested all the analytical estimates accordingly. In $2D$ case, we have taken finite lengths in $X$ axis and the entire $Y$ axis to obtain the estimates and numerically verified those.